\documentclass[11pt]{amsart}
\usepackage[breakable]{tcolorbox}
\tcbuselibrary{theorems}
\usepackage[margin=1in]{geometry} 
\usepackage{amsmath,amsthm,amssymb,graphicx,mathtools,tikz,mathrsfs,mathtools,xcolor}
\usepackage[shortlabels]{enumitem}
\usetikzlibrary{positioning}
\usepackage{pgf,tikz,pgfplots}
\pgfplotsset{compat=1.15}
\usepackage{mathrsfs}
\usetikzlibrary{arrows}
 
\usepackage{fancyhdr}
\usepackage{tikz-cd}
\usepackage{dsfont}
\usepackage{bbm}
\usepackage{enumitem}

\newcommand{\J}{\mathbf{J}}
\newcommand{\rc}{{\textsc{rc}}}
\newcommand{\prc}{\pi^\rc}
\newcommand{\tv}{\textsc{tv}}
\newcommand{\TV}[2]{d_{\tv}\left({#1,#2}\right)} 
\newcommand{\sat}{\mathsf{Sat}}
\newcommand{\E}{\mathbb{E}}
\newcommand{\p}{\mathbb{P}}

\newcommand{\real}{\mathbb{R}}

\newcommand{\abs}[1]{|#1|}

\newcommand{\ol}[1]{\overline{#1}}

\newcommand{\tmix}{t_{\mathsf{mix}}}
\newcommand{\tin}{\tau^{\mathsf{in}}}

\newtcbtheorem[auto counter, number within= section]{ques}
{Question}{colback=white,colframe=black,fonttitle=\bfseries, breakable, before upper={\parindent20pt\noindent}, parbox = false, boxrule = 1pt, arc = 2pt}{ques}

\newcommand{\sgn}[1]{\text{sign}\left(#1\right)}
\newcommand{\SK}{\mathsf{S}}

\newtheorem{theorem}{Theorem}[section]
  \newtheorem{proposition}[theorem]{Proposition}
    \newtheorem{definition}[theorem]{Definition}
  
  \newtheorem{lemma}[theorem]{Lemma}

  \newtheorem{remark}{Remark}[section]

\numberwithin{equation}{section}
\newcommand{\ind}{\text{ind}}
\newcommand{\gap}{\mathsf{gap}}

\usepackage[colorlinks=true, linkcolor=red, citecolor=blue]{hyperref}

\title[Metastability for heavy-tailed spin glasses]{Metastability in Glauber dynamics for heavy-tailed \\ spin glasses}

\author{Reza Gheissari and Curtis Grant}
\address{R.\ Gheissari\hfill\break
Department of Mathematics \\ Northwestern University }
\email{gheissari@northwestern.edu}
\address{C.\ Grant\hfill\break
Department of Mathematics \\ Northwestern University }
\email{curtisgrant2026@u.northwestern.edu}

\begin{document}

\begin{abstract}
   We study the Glauber dynamics for heavy-tailed spin glasses, in which the couplings are in the domain of attraction of an $\alpha$-stable law for $\alpha\in (0,1)$. We show a sharp description of metastability on exponential timescales, in a form that is believed to hold for Glauber/Langevin dynamics for many mean-field spin glass models, but only known rigorously for the Random Energy Models. Namely, we establish a decomposition of the state space into sub-exponentially many wells, and show that the projection of the Glauber dynamics onto which well it resides in, asymptotically behaves like a Markov chain on wells with certain explicit transition rates. In particular, mixing inside wells occurs on much shorter timescales than transit times between wells, and the law of the next well the Glauber dynamics will fall into depends only on which well it currently resides in, not its full configuration. We can deduce consequences like an exact expression for the two-time autocorrelation functions that appear in the activated aging literature. 
\end{abstract}

\maketitle

\section{Introduction}

Equilibrium and out-of-equilibrium spin glasses are of significant interest in both the statistical physics and mathematics literature, and serve as prototypes for complexities of thermodynamic systems in the presence of quenched disorder.

The most famous spin glass models are the short-range Edwards--Anderson model~\cite{EA} and the mean-field models like the Sherrington--Kirkpatrick model~\cite{SK}. 
Towards making progress on short-range spin glasses, a variety of models of a sparser nature have been introduced to interpolate between mean-field models and short range ones. This includes dilute $p$-spin models, and the Viana--Bray model~\cite{KanterSompolinsky,Viana-Bray}. 
Another family of such interpolating models are heavy-tailed mean-field spin glasses, introduced by Cizeau and Bouchaud in~\cite{Cizeau}. Even though these models are fully connected, the subgraph generated by bonds on a given scale is sparse. In particular, the thermodynamics, at least at the level of free energy, are well-approximated by its sparsification---closely related to the Poissonian version of the Viana--Bray model~\cite{Janzen-et-al}. The heavy-tailed models are also of intrinsic interest in the context of literature on heavy-tailed random matrices and localization/delocalization phenomena that arise therein: see e.g.,~\cite{Bordenave-Guionnet-Localization-Delocalization,aggarwal2021goe}. 

Much of the interest in spin glasses in the last three decades is in their off-equilibrium dynamics. The complex energy and free energy landscapes of spin glasses (exponential in the dimension many local minima/metastable states~\cite{ABC13}) and complicated saddle geometries separating them, make dynamical analysis especially intricate. Conjectural consequences of this complexity include rich phenomena like \emph{metastability} and \emph{memory and aging} in the natural reversible off-equilibrium Markov processes (most notably, the Glauber dynamics in discrete state spaces and Langevin dynamics in continuous state spaces). 

At the most basic level, metastability can be described by a set of basins or wells, $(\mathcal C_i)_{i}$ in the state space such that the escape time of the dynamics from these wells is much larger than the time to fall into them. Given such a collection of metastable wells, one is then interested in developing a much more precise picture, in the spirit of Freidlin--Wentzell theory~\cite{FreidlinWentzell}: see~\cite{BovDen2015} for a book treatment. In particular, the asprirational results when concerning metastability for reversible stochastic dynamics are to say the dynamics inside a well (quasi-)equilibrate quickly compared to the transit times between wells. As a consequence, rescaling time by the typical transit time of a well of a certain depth, a projection of the dynamics which simply tracks which well the dynamics is in becomes asymptotically Markovian, with jumps between wells that are themselves exponentially distributed according to the heights of (free) energy barriers between the wells. See~\cite{BEGK-01,BEGK-02,BEGK-JEMS04,BGK-JEMS05} for references developing this rich theory for Markov processes in fixed dimension. 

Rigorous results of this form, showing metastable motion between wells in high-dimensional dynamics in statistical physics have been the subject of much interest for the last several decades, but are still somewhat rare. One notable series of work was in the context of ferromagnetic Ising Glauber dynamics on the lattice, in the zero-temperature limit~\cite{Neves-Schonmann-metastability,BenArous-Cerf-Metastability}. By approximating the random-field Curie--Weiss model's magnetization by a 1D chain,~\cite{BEGK-01} characterized its metastability. Of most relevance to the spin glass literature and the current work is the work~\cite{BABoGa1} where precise metastability, and in particular a coupling to a simplified trap model was shown for the Glauber dynamics for the random energy model (REM) (a simplified model of a spin glass landscape in which the energy function is i.i.d.\ Gaussian, and the lack of smoothness of the landscape means the wells are just single vertices in the hypercube). This latter work was essential to the establishment of aging for the Glauber dynamics of the REM in~\cite{BABoGa2}.

In this paper, we establish a precise metastability picture for the Glauber dynamics of the heavy-tailed spin glass models of~\cite{Cizeau}. Our results apply in a window of timescales that are exponential in the dimension, but pre-relaxation so the dynamics is out-of-equilibrium. For each timescale of the form $e^{N^\gamma}$, there is a relevant decomposition of the state space into wells $(\mathcal C_i)_{i}$ (each of exponential size) inside of which the mixing time is $o(e^{N^\gamma})$ whereas the transits between which are $\Theta(e^{N^\gamma})$. In particular, when time is rescaled as $s t_{a,\gamma} = e^{ a N^\gamma}$, the process $\SK(s t_{a,\gamma})$ which tracks the well to which the dynamics belongs is asymptotically approximated by  a Markov jump process $Y(s)$, whose rates are precisely governed by exponential in a certain free energy difference.

\subsection{Metastability and a Markov jump process between wells} 
We now get more precise in order to state our main results. The Levy spin glass model on $N$ vertices, defined in~\cite{Cizeau}, is defined by the following (random) Hamiltonian $H_{\mathbf{J}}$ on state space $\Sigma_N = \{\pm 1\}^N$,
\begin{align}
    H_{\mathbf{J}}(\sigma) = \sum_{1\le i<j\le N} J_{ij} \sigma_i \sigma_j\,, \qquad \text{where} \qquad J_{ij}=J_{ji} \sim \nu_{\alpha,N} \text{ are i.i.d.}\,,
\end{align}
where $d\nu_{\alpha,N}$ is in the domain of attraction of an $\alpha$-stable law  with $\alpha \in (0,2)$ and scaled to have typical size $O(N^{-1/\alpha})$. For concreteness, let us work with the density function 
\begin{align}\label{eq:coupling-distribution}
    d\nu_{\alpha,N}(x)= \frac{\alpha}{2N} |x|^{-(1+\alpha)}\mathbf{1}\{|x|\ge N^{-\frac{1}{\alpha}}\} dx\,,
\end{align}
though Remark~\ref{rem:general-J} explains how the proof applies more generally. Throughout the paper, let $\mathbb P_{\mathbf{J}}$ denote the law over $\mathbf{J} = (J_{ij})_{i<j}$, i.e., $\mathbb P_{\mathbf{J}} =  \nu_{\alpha,N}^{\otimes \binom{N}{2}}$. The Gibbs measure induced by this Hamiltonian at inverse temperature $\beta>0$ is the ($\mathbb P_{\mathbf{J}}$-random) measure
\begin{align}\label{eq:Gibbs-measure}
    \pi_{\beta,\mathbf{J}}(\sigma) \propto e^{\beta H(\sigma)}\,.
\end{align}
Recently,~\cite{JL} showed that when $\alpha\in (1,2)$ the limiting free energy exists, and ~\cite{CKS} gave an explicit characterization of the limiting free energy in the $\alpha \in (1,2)$ regime. In the $\alpha\in (0,1)$ regime,~\cite{CKS} showed that the free energy is dominated by the largest $O(1)$-many bonds amongst $\mathbf{J}$, all of which will tend to be satisfied in a typical sample from equilibrium: in particular, the partition function is exponential in $N^{1/\alpha}$ rather than $N$. 
These properties, and the satisfaction of the largest select number of bonds, will play a large role in the off-equilibrium  dynamics we study. 

The Glauber dynamics $(X(t))_{t\ge 0}$ are the continuous-time Markov chain which assigns each of the sites in $[N]$ a rate-1 Poisson clock, and if the clock at time $t$ rings at a vertex $v$, it resamples 
\begin{align*}
    X_{v}(t) \sim \pi_{\beta,\mathbf{J}}\big(\sigma_v\in \cdot \mid \sigma_{w} = X_w(t^-) \text{ for }w\ne v\big)\,;
\end{align*}
and keeps all other spins unchanged, i.e., $X_w(t) = X_{w}(t^-)$ for $w\ne v$. This Markov chain is easily checked to be reversible with respect to $\pi_{\beta,\mathbf J}$ and since its introduction in 1963 by~\cite{Glauber}, has been considered the prototypical model for out-of-equilibrium dynamics of statistical mechanics systems. 

The mixing time, or time for the law of $X(t)$ to approximate $\pi_{\beta,\mathbf{J}}$ from a worst-case initialization in total-variation distance, can be shown to be $\exp(2\beta(1+o(1))\|\J\|_\infty)$, as it is governed by the time for the spins incident to the largest bond $J_e$ in $\mathbf{J}$ to flip from $\{+,+\}$ (resp., $\{+,-\}$) to $\{-,-\}$ (resp., $\{-,+\}$) if $J_e>0$ (resp., $J_e<0$): indeed, this is a special case of our Theorem~\ref{thm:mixing-below-frozen-bond}. This is one of many sources of exponentially slow mixing, e.g., for a large bond $e$,  the transit between the two satisfying assignments of spins to its endpoint vertices, takes an exponentially long time in $2\beta |J_e|$. These large (but not largest) bonds will serve as the source of metastability on activated (exponential) but pre-equilibration timescales. Of course, each such transit also depends on the values of all spins on the complement of the large bonds, and much of the work in this paper is to show that those coordinates all quasi-equilibrate (conditional on the values of the largest bonds) on faster timescales than the transits between satisfying assignments on the largest bonds. 

To state our main theorems, suppose we are interested in understanding metastability on timescales greater than  $t_{a,\gamma} = e^{a N^\gamma}$ for a fixed $\gamma<1/\alpha$. This will dictate the minimal resolution at which we observe the skeleton process tracking which well the Glauber dynamics is in. 
\begin{definition}\label{def:sat-assignments}
    For a realization of $\mathbf{J}$, for fixed $(a,\gamma)$, the $t_{a,\gamma}$-relevant coordinates are  
    \begin{align}\label{eq:t-relevant-coordinates}
        V_{a,\gamma} = \{i:  (i,j) \in E_{a,\gamma} \text{ for some $j$}\}\,, \qquad \text{where} \qquad E_{a,\gamma} = \{ (i,j) : |J_{ij}| \ge \frac{a}{2\beta} N^\gamma\}\,.
    \end{align}
    We can also define $K=|E_{a,\gamma}|$, and observe that $K \approx  \frac{(2\beta)^{\alpha}}{2a^{\alpha}} N^{1-\alpha \gamma}$. 
\end{definition}

The $t$-relevant coordinates are the ones that are unlikely to flip on $o(t)$ times, representing the directions where exponential in the energy barrier height between wells is at least of order $t$. 
We use the $t$-relevant coordinates to define the $t$-metastable wells. In what follows we enumerate the bonds $J_{ij}$ in decreasing order (in absolute value) as $\abs{J_{(1)}} > \abs{J_{(2)}} > ... \abs{J_{(\binom{N}{2})}}$. It is easy to check that for $\gamma \in (\frac{1}{2\alpha},\frac{1}{\alpha})$, with $\mathbb P_{\mathbf{J}}$-probability $1-o(1)$, the edges $e_1,...,e_K$ are vertex disjoint (see Lemma~\ref{lem:energy-drop-off-J}). {Henceforth, we will work on this event, and therefore} we may label the corresponding edges $e_l=(v_l,w_l)$ for the edge such that $J_{e_l}=J_{(l)}$, and have $\{v_l,w_l\} \cap \{v_{l'},w_{l'}\} =\emptyset$ for all $1\le l<l'\le K$. 

\begin{definition}
    For a spin assignment $\sigma$, we let $\sat(\sigma)$ be the set of satisfied bonds, i.e., 
    \begin{align}\label{eq:sat(sigma)}
        \sat(\sigma) = \{e={(i,j)}: \sigma_i \sigma_j = \sgn{J_{ij}}\}\,.
    \end{align}
    For a relevant timescale $t_{a,\gamma}$, consider $\{\sigma:  \sat(\sigma) \supset E_{a,\gamma}\}$ and let $K = |E_{a,\gamma}|$. This set breaks into $2^{K}$ connected components of the hypercube (because the two satisfying assignments on an edge are not adjacent in $\{\pm 1\}^N$); call these components $(\mathcal C_i)_{i}$. 

    Each $t$-metastable well $\mathcal C_i$ is naturally identified with an element of $\{\pm 1\}^{K}$ by associating $\mathcal C_i$ with the vector given by $(\sigma_{v_1},...,\sigma_{v_K})$ (notice that given $\mathbf J$, the spin on $w_1,...,w_K$ is determined by this and the fact that all these edges are satisfied in $\mathcal C_i$). With this identification, {we say $\sigma$ is in a $t$-metastable well if $\sat(\sigma) \supset E_{a,\gamma}$, and when} $\sigma$ is in one of the $t$-metastable wells, we can let $\mathsf{S}(\sigma) \in \{\pm 1\}^{K}$ tell us what well $\sigma$ belongs to. 
\end{definition} 

Our main theorem says that the projection of the Glauber dynamics at a resolution of timescales greater than $t_{a,\gamma}$ asymptotically behaves like a Markov chain $Y$  on $\{\mathcal C_i\}_i$ (equivalently on $\{\pm 1\}^K$) with explicit transit rates between the metastable wells. To define the transition rates for $Y$, for a vertex $v \in \{1,...,N\}$ and a configuration $\sigma \in \{\pm 1 \}^N$ we define: 
\begin{align}
    \label{eq:def-of-Z-and-m} 
    Z_v(\sigma) = \exp(-2\beta  m_v(\sigma)) \qquad \text{where}\qquad  
    m_v(\sigma) = \sum_{w \neq v} J_{vw} \sigma_v \sigma_w \, .
\end{align}

The Markov chain $Y(s)$ on wells that will approximate $\SK(X(st_{a,\gamma}))$ is the following. 
\begin{definition}
    For any $a,\gamma$, {recall that} $K = |E_{a,\gamma}|$. Define the approximating Markov process $Y(s)$ on $\{ \pm 1 \}^K$ as follows. {From an initialization $Y(0)$,} if $Y$ is in state $Y(s)$ at time~$s$, 
    \begin{enumerate}
        \item Assign its $l$'th coordinate a Poisson clock of rate ${t_{a,\gamma} Z_l(Y(s))}$, where 
        \begin{align}\label{eq:Z(Y_s)}
            Z_{l}(Y) = { \E_{\pi_{\beta,\mathbf{J}}}} \big[ Z_{v_l}(\sigma) + Z_{w_l}(\sigma) \,|\, E_{a,\gamma} \subset \mathsf{Sat}(\sigma)\,,\,\mathsf{S}(\sigma) = Y \big]\,.
        \end{align}       
        \item When the clock of a coordinate {$l$ rings, resample the value of $Y_l$} uniformly from~$\{\pm 1\}$. 
    \end{enumerate}
\end{definition}

Having constructed $Y(s)$ we may now state our main theorem, which will apply to the range of $\gamma$'s in the open interval 
$${\gamma \in (\gamma_0,\frac{1}{\alpha})\qquad \text{where} \qquad\gamma_0 = \frac{5\alpha+1}{2\alpha(2\alpha+1)}< \frac{1}{\alpha}\,.}$$
(Notice that this interval length vanishes as $\alpha \uparrow 1$ and diverges as $\gamma \downarrow 0$.)

\begin{theorem}\label{thm:main1}
    For any $\alpha\in (0,1)$, with $\mathbb{P}_{\mathbf{J}}$-probability $1-o(1)$, the following holds. Fix any $\beta>0$, any $a>0$, and $\gamma \in (\gamma_0, \frac{1}{\alpha})$ and {recall that} $K = |E_{a,\gamma}|$. 

    Let $(X(t))_{t\ge 0}$ be Glauber dynamics initialized from $X(0) \sim \text{Unif}(\Sigma_N)$ and let $\mathsf{S}(t) = \mathsf{S}(X(t))$. 
    The finite-dimensional distributions of process $(\mathsf{S}(st_{a,\gamma}))$ are within total-variation distance $o(1)$  of the finite dimensional distributions of $Y(s)$ with $Y(0)\sim \text{Unif}(\{\pm 1\}^K)$.    
     To be more precise, if we pick any finite number of times $s_1<...<s_n$ (that may depend on $\J$ and $N$ arbitrarily so long as {there is a constant $c>0$ such that $|s_{i} - s_j|>c$ for all $N$ and all $i,j$}),  
     \begin{align*}
         d_\tv\Big( \mathbb P\big(\mathsf{S}({s_1 t_{a,\gamma}})\in \cdot ,...,\mathsf{S}({s_n t_{a,\gamma}})\in \cdot\big), \mathbb P\big(Y({s_1})\in \cdot,...,Y({s_n})\in \cdot\big)\Big) = o(1)\,.
     \end{align*}
\end{theorem}

\begin{remark}\label{rem:number-of-times}
    The number of times $n$ taken in the previous theorem can in fact be anything sub-exponential in $N$ and the result still holds true. On the other hand, if $n$ is larger than exponentially large in $N^\gamma$, there will be times when the well process does not take a value as the Glauber dynamics is between wells, and thus such an approximation could not possibly hold. 
\end{remark}

Theorem~\ref{thm:main1} characterizes the projection of the Markov chain into its metastable wells as itself a Markov process. Moreover, this holds for a wide-range of pre-equilibration activated timescales. And, as predicted by the theory of metastability, the jump times of that latter Markov process, between wells are governed by the easiest paths of transit between them in the restricted state space.  We emphasize that since the $s_i$ may depend on $N$ and the realization of $\J$, the theorem encapsulates looking exactly on the timescale of the transits, where the typical number of transits between $s_i$ and $s_{i+1}$ will be a tight Poisson random variable. These transition rates, depend primarily on the bond weight corresponding to that specific edge whose satisfying state has to be flipped, but also depend on the stationary law given the values in the other $t$-metastable coordinates (i.e., which well it is currently in). Notably, however, it is the stationary measure in that well, because the mixing time within the well is much faster than the relevant between-well transit timescale: a central feature of metastability.

\begin{remark}\label{rem:temperature-transition}
        In the context of Theorem~\ref{thm:main1}, there is a ``high-temperature" regime of $\beta<\beta_0$ where 
        \begin{align}\label{eq:beta-0}\beta_0 := \sup\{\beta: (2\beta)^{\alpha} \Gamma(1-\alpha) <1\} = \frac{1}{2 \Gamma(1-\alpha)^{1/\alpha}}\,,\end{align}
        where $\Gamma(\cdot)$ is the Gamma function. Then the stationary measure of $Y$ is uniform on $\{\pm 1\}^K$ (see Theorem~\ref{thm:high-temp-regime}), and Theorem~\ref{thm:main1} would hold if the rates in~\eqref{eq:Z(Y_s)} were replaced by $e^{-2\beta |J_{(l)}|}$. In that case, the rates would be independent of which well the configuration is in, and the skeleton process simply performs a simple random walk with different rates for different coordinates.   
\end{remark}

\subsection{Behavior of two-time auto-correlation functions}

As mentioned in the discussion of the REM, predictions of metastability in spin glass dynamics are closely related to the predicted phenomenon of aging. Roughly speaking, aging can be summarized by the phrase \emph{the older a system is, the longer it takes to forget its past}. 
This is often quantified in terms of some kind of two-time correlation function $C(s,t)$ decaying not with $t-s$ in the limit that $s,t$ get large together, but as a function of $t/s$ or perhaps even more extreme. This has led to a host of predictions of universal aging phenomena that arise in off-equilibrium dynamics of spin glasses. 

The timescales most pertinent to the metastability picture are activated (exponentially large, but pre-mixing) timescales. Here, the mechanism for  aging in spin glasses can be described precisely assuming a sharp metastability picture is known for the dynamics: the dynamics jumps between wells of different depths, each requiring different activation energies to leave, and there being exponentially more wells with shallower barriers. Each well traps the dynamics for time exponential in its (free) energy barrier, and when the dynamics escapes the well, it picks a next random well to fall into, so that at longer and longer timescales, the typical dynamics is trapped in a deeper and deeper well. 
The fact that at low-temperatures mixing is exponentially slow due to deep wells that trap the dynamics is now known for many of the canonical mean-field spin glass models (see~\cite{GhJa19,BAJ-low-temp-spin-glass-dynamics-1,BAJ-low-temp-spin-glass-dynamics-2}) but the full metastability picture requires much more refined information than one bottleneck, namely understanding both the equilibration properties within a well, and identifying sharp exponential rates for transit times between different wells. 

Using the sharp metastability result of Theorem~\ref{thm:main1}, one can precisely characterize the behavior of the two-time autocorrelation function on activated timescales {(although in our actual proof, it is easier to arrive at this theorem earlier than the full metastability result: see the beginning of Section~\ref{sec:proof-of-main-results})}. For this purpose, define
\begin{align}\label{eq:auto-correlation}
     C_N(t_N, s t_N) = \frac{1}{N} \sum_{v} X_v(t_N) X_v(s t_N)\,.
\end{align}

\begin{theorem}\label{thm:main-autocorrelation}
    Fix any $\beta>0$, $\gamma \in ({\gamma_0}, \frac{1}{\alpha})$ and let $t_{a,\gamma} = e^{ a N^{\gamma}}$. With $\mathbb P_{\mathbf{J}}$-probability $1-o(1)$, if $X(0)\sim \text{Unif}(\Sigma_N)$, for any fixed $s>1$, 
    \begin{align*}
        \TV{\text{Law}(C_N(t_{a,\gamma}, st_{a,\gamma}))}{ q_{a,\gamma}^{(N)}} = o(1)\,,
    \end{align*}
    where the limiting distribution is 
    \begin{align*}
        q_{a,\gamma}^{(N)} = \text{Law}_{a,\gamma}(\langle \sigma, \sigma'\rangle_N ) 
    \end{align*}
    where $\sigma, \sigma'$ are independent samples from $\pi_{\beta,\mathbf{J}} ( \cdot \mid \mathcal C_i)$ for a uniform-at-random chosen $\mathcal C_i$ among the set of $t_{a,\gamma}$-metastable wells.  
\end{theorem}

Either the replica quantity $q_{\alpha,\gamma}^{(N)}$ is zero, as it would be in the high-temperature regime $\beta<\beta_0$ (see Proposition ~\ref{prop:law-q-high-temp}), or it may be non-zero but still $s$-independent (only depending on the activated timescale we choose).  
In particular, the $s$-independence indicates that in heavy-tailed spin glass dynamics, hallmarks of aging seen by the autocorrelation function would require a different time-rescaling from the multiplicative one. This is basically due to the fact that the scaling of the minimal energy barriers between different wells are to first order the same regardless of which well the dynamics is residing in (in contrast to classical spin glasses where the wells' minimal energy barriers span a range). 
 Understanding this better, as well as determining the law of the overlap quantity $q_{a,\gamma}^{(N)}$ at large $\beta$ are interesting questions. 

\medskip
 Let us conclude by describing some other forms of aging statements that appear in the literature. A standard approach in the mathematics literature to investigating activated aging is to look at the quantity $\p( X(t_N) \ne X(st_N) )$ in lieu of $C_N$.  Bouchaud~\cite{Bouchaud} introduced a trap model, serving as a toy model for metastable dynamics, where wells are mimicked by single vertices, and upon escaping a well, the next well to fall into is simply chosen uniformly at random. The limit of the function $\p( X(t_N) \ne X(st_N))$  is then given by an appropriately chosen arcsine law. 
 Since that work, various high-dimensional dynamics have been shown to belong to its universality class; this in particular includes random hopping time dynamics (where simple random walk is time-changed by the landscape, but the path itself is independent of the landscape) for the REM, and later for $p$-spin models~\cite{BoGa,BoGaSv,BenArous-Gun-universal-aging}. 
In terms of Glauber/Langevin-type dynamics the only rigorous results for aging in the activated timescales are those of Gayrard~\cite{Gayrard-Aging-REM} for the random energy model (REM) where the energy is i.i.d.\ Gaussians, and the deep wells are still single vertices in the state space $\{\pm 1\}^N$. However, when considering Glauber dynamics on smoother landscapes (e.g., $p$-spin models, or the models in this paper) many coordinates flip in $O(1)$ timescales and wells have exponential volume in $\{\pm 1\}^N$, and so $\p( X(t_N) \ne X(st_N) )$ is essentially instantaneously zero for $s>1$.

Another form of aging that necessitates mention, is \emph{short-time} aging (where $C_N(s,t)$ is studied in the $N \to \infty$ while $t =O(1)$ regime). There, the mechanism for the aging is very different from activated aging and is not related to metastability. Short-time aging has been predicted by means of the Cugliandolo--Kurchan--Crisanti--Horner--Sommers (CKCHS) equations~\cite{CugKur93,Crisanti1993} in spherical $p$-spin glasses. The rigorous mathematics literature does not establish the aging picture, but there has been some significant progress: see~\cite{BDG01,BADG06} for the rigorous establishing of the CKCHS equations for the $N\to\infty$ limit of $C_N(s,t)$, and~\cite{BGJ18a,DeSu20,DeGh21,sellke2024threshold} for a sampling of other works examining short-time spherical spin glass dynamics.  

\begin{remark}\label{rem:general-J}
    All our results are stated for the explicit heavy-tailed law of~\eqref{eq:coupling-distribution}, but in fact hold in much greater generality for heavy-tailed distributions. Namely, replacing~\eqref{eq:coupling-distribution} with symmetric random variables $X_{ij}$ satisfying:
    \[
\p (\abs{X_{ij}} >t ) = \frac{L(t)}{t^{\alpha}}\,,
    \]
    for $L$ a slowly varying function, then all result can be checked to hold provided we define the Hamiltonian by:
    \[
H_N(\sigma) = \frac{1}{b_N} \sum_{1 \leq i < j \leq N} X_{ij} \sigma_i \sigma_j\,,
    \]
    where $b_N$ is defined by:
    \[
b_N = \inf \big\{ t \geq 0: \p( \abs{X_{ij}} > t) \leq \frac{1}{N} \big\}\,.
    \]
The minor steps needed to generalize our lemmas in Section~\ref{sec:coupling-matrix-preliminaries} on properties of $\mathbf{J}$ can be found in \cite{AS,CKS}, and the many sources therein. The lemmas in Section~\ref{sec:coupling-matrix-preliminaries} are then the only way in which the law of $\mathbf{J}$ influences our main results.
\end{remark}

\subsection{Proof ideas}
    In this section, we overview the key steps in the proof of Theorem~\ref{thm:main1}, and simultaneously further explore details of the precise metastability picture we establish for the Glauber dynamics. The key part of establishing metastability is two-fold: 
    \begin{enumerate}
        \item The mixing time interior to a well $\mathcal C_i$ (in our context, given all edges in $E_{a,\gamma}$ are frozen to a specific satisfying realization) is faster than the time to escape $\mathcal C_i$; 
        \item The time spent in between wells is negligible, and the relative rates of transit between wells is fully governed by which of the edges of $E_{a,\gamma}$ flips to unsatisfied and the stationary field on the endpoints of that edge.   
    \end{enumerate}
    Given these ingredients, the transit rates can be treated as Gibbs averages of the rates for the $l$'th edge in $E_{a,\gamma}$ to become unsatisfied, as claimed by Theorem~\ref{thm:main1}. In particular, when looking at timescales on the order of  $t_{a,\gamma}$  or greater, the coordinates that are not $t_{a,\gamma}$-relevant can be treated as instantaneously  being sampled from stationarity on the current well the dynamics is in, and transits between wells can also be treated as happening instantaenously.  

    The statement of item (1)  above can be viewed as independently interesting as a statement of mixing times when restricted to a phase, and can be formulated as follows. 

    \begin{theorem}\label{thm:mixing-within-a-well}
        { Let $\gamma \in (\gamma_0,\frac{1}{\alpha})$ and $a>0$, and recall} $t_{a,\gamma} = e^{aN^{\gamma}}$. { For} any $t_{a,\gamma}$-metastable well $\mathcal C_i$, the mixing time of the Glauber dynamics restricted to $\mathcal C_i$ is $o(t_{a,\gamma})$; in fact it is at most $e^{- N^{\eta}} t_{a,\gamma}$ for some $\eta>0$. On the other hand, the escape time from $\mathcal C_i$ is at least $\Theta(t_{a,\gamma})$. 
    \end{theorem}

    Much of the work of the paper goes into establishing this sharp bound on the mixing time restricted to a single (or collection of) $t_{a,\gamma}$-metastable well(s). Since the escape time of such a well is on the order $e^{ a N^{\gamma}}$, it is crucial that the the mixing time inside the well have the sharp constant in front of $N^{\gamma}$ in its exponent for it to be negligible as compared to $t_{a,\gamma}$. In particular, crude bounds on the mixing time like exponential in the maximum of the Hamiltonian restricted to the smaller bonds are insufficient, because the sum of all bonds smaller some $J_{(\ell)}$ typically will far exceed $J_{(\ell)}$. 
{ Theorem ~\ref{thm:mixing-within-a-well} will follow from a more general mixing time estimate (see Theorem~\ref{thm:mixing-below-frozen-bond}); we provide a short proof of Theorem~\ref{thm:mixing-within-a-well} at the end of Section~\ref{sec:mixing-block-estimate}.}

   In order to get this sharp bound, we leverage the fact that strong bonds in the heavy-tailed spin glass have a sparse structure. In particular, we split the non-$t_{a,\gamma}$-relevant coordinates into those stronger than {$N^{\rho}$ for $\rho \approx \frac{1}{2\alpha} <\gamma_0$} and those smaller than this. Those larger than $N^{\rho}$ evolve essentially independently of one another, so the restriction of the dynamics to those bonds can be treated as a product chain and its mixing time can be bounded by exponential in its largest bond, rather than the sum of all its constituent bonds. On the other hand, the sites below the threshold $N^{\rho}$ can interact in a mean-field manner, but their total Hamiltonian is $o(t_{a,\gamma})$ so their mixing time is also sufficiently small. These two estimates are stitched together by a block dynamics (see e.g.,~\cite{FM} which talked about the block dynamics as an essential tool in the Ising Glauber dynamics literature) analysis that alternates between updating the sparse structure of bonds between $[N^{\rho},N^{\gamma}]$ and those smaller than $N^{\rho}$. A separate coupling argument bounds this block dynamics' mixing time, showing that it too can be controlled. 
    These arguments are implemented in Section~\ref{sec:mixing-within-well}.  
    
    Altogether, these Markov chain analyses reduce things like Theorem~\ref{thm:mixing-within-a-well} to careful estimates on row sums of the coupling matrix, including conditional on that row's corresponding site having a heavy bond so that $i\in V_{a,\gamma}$. They also require estimates on spacings in the point process given by the point process, as this separation ensures the gap between the mixing time in a well and its mean time to exit in the easiest direction. These estimates are presented first in Section~\ref{sec:coupling-matrix-preliminaries}. 

    With steps (1)--(2) above in hand, we construct the coupling between the projection process $\mathsf{S}(X(st_{a,\gamma}))$, and the simplified Markov chain on wells $Y(s)$ in Section~\ref{sec:proof-of-main-results} to prove Theorem~\ref{thm:main1} and Theorem~\ref{thm:main-autocorrelation}. On a timescale $s_N t_{a,\gamma}$ there will be at most one coordinate $L$ in the skeleton that may be flipping, but not fast enough to equilibrate given the coordinates with larger bonds. The heart of the coupling is coupling the jumps of this critical coordinate between $\SK$ and $Y$. This coupling requires showing concentration for the jump time of $\SK_L$ (which is a heterogeneous exponential clock depending on the full realization of $X(t)$), as well as for $Y_L$ to a single exponential clock with deterministic rates given by a certain Gibbs average. This concentration is provided by the fact that the mixing time on lower coordinates is much faster than the jump time of the $L$'th coordinate. 
    
    Finally, Appendix~\ref{appendix:high-temp} establishes that when $\beta<\beta_0$, one gets high-temperature behavior  as claimed in Remark~\ref{rem:temperature-transition}, and the autocorrelation limit $q$ of Theorem~\ref{thm:main-autocorrelation} is zero.

\subsection*{Acknowledgements. } {The authors are grateful to the anonymous referees for helpful comments and suggestions.} 
The authors thank Antonio Auffinger for helpful comments. The research of R.G. is supported in part by National Science Foundation, Division of Mathematical Sciences grant number 2246780.  The research of C.G is supported in part by a Natural Sciences and Engineering Research Council of Canada Post-Graduate Scholarship Doctoral award.

\section{Preliminary estimates on the coupling matrix}\label{sec:coupling-matrix-preliminaries}

In this section we state and prove the estimates on maximal entries and maximal row sums of the heavy-tailed random matrix $\J$, and on gaps in the histogram of its entries at large scales.  

We begin with a notational disclaimer that throughout this paper, statements are understood to hold for all $N$ sufficiently large, and constants like $c,C,D$ may depend on $\alpha,\beta$ but not on $N$.

\subsection{Row sum estimates}
Our first lemma is a refinement of Lemmas 1--2 from~\cite{AS}, showing that the rows with the largest entries don't have large row sums otherwise.

\begin{lemma} \label{lem:energy-drop-off-J} 
    Fix any $1/(2\alpha) < \rho < 1/\alpha$, and let $J= (J_{ij})$ be an i.i.d symmetric matrix sampled from $\p_\J$. We have the following: 
\begin{enumerate}
        \item With high probability, no row has a pair of entries larger than $N^{\rho}$ and $N^{1/2\alpha}$. I.e.,
        \[
\p_{\J} \big(\exists i \ \exists j_1 \neq j_2 : \abs{J_{ij_1}} \geq N^{\rho} , \abs{J_{ij_2}} \geq N^{1/2 \alpha } \big) = o(1)\,.
        \]
        \item With high probability, any row with large entry, has a small row sum when its largest entry is excluded; I.e., there exists $C(\alpha)>0$ such that:
        \[
\p_{\J} \bigg( \exists i \ \max_{1 \leq j \leq N} \abs{J_{ij}} > N^{\rho}, \sum_{1 \leq j \leq N} \abs{J_{ij}} - \max_{1 \leq j \leq N} \abs{J_{ij}} >   C N^{1/(2\alpha)} \bigg) =o(1)\,.
        \]
\item There exists a constant $C_{\alpha}>0$ such that, 
 \[
    \p_\J\Big(\sum_{i,j: \abs{J_{ij}} < N^{\rho}} \abs{J_{ij}} > C_{\alpha} N^{1-\alpha \rho + \rho} \log N\Big) = o(1)\,.
 \]
    \end{enumerate}
\end{lemma}

\begin{proof} First note by definition of $\nu_{\alpha,N}$ from~\eqref{eq:coupling-distribution} we have that: 
\begin{align} \label{eq:tails-of-J} 
\p_\J( \abs{J_{ij} }> rN^{-\frac{1}{\alpha}} ) = r^{-\alpha}  \, ,    
\end{align}
for any $r \geq 1$. 
For (1) applying \eqref{eq:tails-of-J} and a union bound over all rows, one has,
\[
\p_{\J} (\exists i , \ \exists j_1 \neq j_2 : \abs{J_{ij_1}} \geq N^{\rho}, \abs{J_{ij_2}} \geq N^{\frac{1}{2\alpha}}) \leq \frac{N^3}{N^{3/2}N^{1+\alpha\rho }} \, , 
\] 
and since $\rho > \frac{1}{2\alpha}$, the bound above tends to $0$ as $N \to \infty$.

For proving (2) we proceed as follows:
 write $b_N= N^{2/\alpha}$, and let $T$ be an integer so that $\eta = (2T+1)^{-1} < \alpha/8$. For a fixed row $i$ consider the random variables $M_{i,k}$ given by:
\[
M_{i,k} = \big| \{ j: \abs{J_{ij}} > b_N^{\eta k} N^{-1/\alpha} \} \big|\,.
\]
By the tail probabilities ~\eqref{eq:tails-of-J}, $\E_{\J} [M_{i,k}] = (N-1) b_N^{-\alpha \eta k}$.  By a Chernoff bound, we have for $0 \leq k \leq { T}$,
    \begin{align*}
 \p_{\J}( M_{i,k} &\geq 2\E [M_{i,k}] ) \leq \exp(-N^{\theta})\,,
    \end{align*}
for $\theta = 1/(4T+2)$. Thus with $\p_{\J}$ probability $1-o(1)$ every row of $\J$ satisfies: 
\begin{align*}
\sum_{j: \abs{J_{ij}}  < b_N^{(T+1)/(2T+1)} N^{-\frac{1}{\alpha}}} \abs{J_{ij}}  \leq \sum_{k=0}^{T} M_{i,k} \frac{b_N^{(k+1) \eta}}{N^{\frac{1}{\alpha}}} 
&\leq 2N^{1-\frac{1}{\alpha}} \sum_{k=0}^T b_N^{(k+1)\eta} b_N^{-\alpha k \eta} \\
&\leq { 2(T+1)} N^{-\frac{1}{\alpha}} 
b_N^{1/2+\alpha/4}  \\
&\le { 2(T+1) } N^{-\frac{1}{\alpha}} b_N^{3/4} \\ 
&\le { 2(T+1)} N^{\frac{1}{2\alpha} }\,,
\end{align*}
where the inequality in the third line goes through the bound $$\frac{1}{N^{1/\alpha}} (2N {(T+1)}) b_N^{(1-\alpha) \eta T +\eta} \, ,$$ uses the bound $\eta T<1/2$, and then uses the term $b_N^{-\alpha/2}$ to kill the $N$ prefactor.

For the contribution from larger terms, we {  use a Chernoff bound to obtain}  
\begin{align*}
\p_{\J} \bigg( \exists i: \sum_{b_N^{\frac{T+1}{2T+1}} \leq N^{\frac{1}{\alpha}}\abs{J_{ij}} \leq b_N^{\frac{3}{4}} } \!\!\! \abs{J_{ij}} > C b_N^{3/4} N^{-\frac{1}{\alpha}}  \bigg)
&\leq N \p_{\J} \Big( \Big|\Big\{j : \abs{J_{ij}} > b_N^{\frac{T+1}{2T+1}} N^{-\frac{1}{\alpha}} \Big\}\Big| \geq C  \Big) \\
&\leq N^{1+C} (b_N)^{-\alpha C \frac{T+1}{2T+1} }
\\
&= N^{1+ \frac{1}{2} C -  \frac{2T+2}{2T+1} C } \,,
 \end{align*}
since $(T+1)/(2T+1) > 1/2$ the exponent above is negative if $C$ is sufficiently large. 

Since $b_N^{3/4} N^{-1/\alpha} = N^{1/(2\alpha)}$, if $C$ is chosen sufficiently large we have with $\p_\J$ probability $1-o(1)$ that every row of $\J$ satisfies,
\[
\sum_{j : \ \abs{J_{ij}} < N^{\frac{1}{2\alpha} }} \abs{J_{ij}} < CN^{\frac{1}{2\alpha} } \, .
\]
By part (1), every row contains at most one entry larger than $N^{\frac{1}{2\alpha}}$ with $\p_\J$ probability $1-o(1)$, so item (2) follows.

For part (3) we proceed as follows: { for $l\geq 0$ }define $N_l= |\{(i,j): {\abs{J_{ij}}} \in [2^{-l-1} N^{\rho}, 2^{-l} N^{\rho})\}|$. By~\eqref{eq:tails-of-J}, $\E[N_l] = {\Theta( N^2 2^{\alpha l} N^{-\alpha \rho -1})}$, and so by a Chernoff bound one has, for an absolute constant $c>0$, 
\[
\p_{\J}( N_l > \E [N_l ]\log N ) \leq \exp\Big( -c 2^{\alpha l} N^{1-\alpha \rho} \log N ) \Big)\,.
\]
Summing this over {$l \ge 0$} and using $\alpha<1$, with probability $1-o(1)$ (in fact $1- e^{ - c N^{1-\alpha \rho}}$) we have,
\[
\sum_{\abs{J_{ij}}<N^{\rho} } \abs{J_{ij}} \le \sum_{l=0}^{\infty}  2^{\alpha l} \frac{N^{\rho}}{2^l} \frac{N}{N^{\alpha \rho}} \log N  \le  C_{\alpha} N^{1+\rho - \alpha \rho} \log N\,,
\]
which gives the desired result.
\end{proof}

\subsection{Gaps in the point process of large bonds}

The next two lemmas of this section describe gaps in the point process of $\J$ (in absolute value) when working at certain scales. We want to ensure that at large values of $|J_{ij}|$, there are large gaps between consecutive bonds in the order statistics, and that perturbations of some large $|J_{ij}|$ by the absolute sum of other entries in its row does not destroy the gap. 
{Henceforth, let \begin{align}\label{eq:xi}\xi(\alpha):= \frac{1-\alpha}{1+2\alpha}\end{align} so that for any $\gamma$ in the range $\gamma_0 < \gamma < \frac{1}{\alpha}$ one has: 
     $2\alpha \gamma + \xi >2$ and $\gamma -\xi > \frac{1}{2} + \frac{1}{2\alpha }$.}
%
For any {$\gamma > \gamma_0$} 
we can let $\rho:= \rho(\gamma,\xi)$ be any number such that  
\begin{align} \label{eq:rho-def}
\rho>\frac{1}{2\alpha}\,, \qquad \text{and} \qquad 1-\alpha \rho + \rho < \gamma -\xi \, ;
\end{align}
such a $\rho$ exists as $\gamma-\xi > \frac{1}{2} + \frac{1}{2\alpha}$, which is the value of $1-\alpha \rho + \rho$ at $\rho = \frac{1}{2\alpha}$.

\begin{lemma} \label{lem:gaps} Let $\xi$ be as in~\eqref{eq:xi}, let $\gamma > \frac{1}{1+\alpha} + \frac{1}{2\alpha}$, $\rho$ be as in \eqref{eq:rho-def}, and set $\vartheta_N = N^{-\xi}$. 
\begin{enumerate}
 \item For any $a>0$, the number of pairs $(i,j)$ such that $\abs{J_{ij}} \in [\frac{a(1-\vartheta_N)}{2\beta} N^{\gamma}, \frac{a(1+\vartheta_N)}{2\beta} N^{\gamma} ]$ is zero with probability $1-o(1)$. I.e.,
 \[
\p_{\J} \Big( \exists (i,j) : \abs{J_{ij} } \in \big[ \tfrac{a(1-\vartheta_N)}{2\beta} N^{\gamma}, \tfrac{a(1+\vartheta_N)}{2\beta} N^{\gamma} \big] \Big) = o(1)\,. 
 \]
\item  The sum over all non-maximal entries in rows containing an element $|J_{ij}|\ge N^{\rho}$, is bounded by $DN^{1-\alpha \rho + 1/2\alpha}$ for some $D$ independent of $N$. I.e., if $A= \{i: \max_j |J_{ij}| \ge N^{\rho}\}$,
    \[
\p\Big( \sum_{i \in A} \Big(\sum_{j=1}^{N} \abs{J_{ij}} - \max_{1 \leq j \leq N} \abs{J_{ij}} \Big)> D N^{1-\alpha \rho + 1/2\alpha} \Big) =o(1)\,. 
    \]
\end{enumerate}    
\end{lemma}

\begin{proof}
In order to prove part (1), note that the expected number of {$(i,j)$ pairs such that $\abs{J_{ij}} \in [\frac{a(1-\vartheta_N)}{2\beta} N^{\gamma}, \frac{a(1+\vartheta_N)}{2\beta} N^{\gamma}]$} is 
\[
 (2\beta)^{\alpha} \frac{N-1}{2a^{\alpha}N^{\alpha \gamma} } ( (1-\vartheta_N)^{-\alpha}-(1+\vartheta_N)^{-\alpha} ) = 
 O(N^{1-\xi - \alpha\gamma})\,.
\]
Since $\gamma> (1-\xi)/\alpha$, the above is $o(1)$. An application of Markov's inequality finishes the proof. 

To prove (2) we proceed as follows.
The number of entries of $\J$ whose absolute value is at least $N^\rho$ is distributed as a binomial random variable with parameters $\binom{N}{2}$ and $N^{-1-\alpha \rho}$. Consequently, by standard tail bounds on binomial random variables, we have with probability $1-o(1)$, the number of such $ij$ is at most $100 N^{1-\alpha \rho}$. By item (2) of Lemma~\ref{lem:energy-drop-off-J}, also with probability $1-o(1)$, for any row containing an entry larger than $N^\rho$, the sum of the non-maximum elements is at most $C N^{1/(2\alpha)}$. On the intersection of these events, an upper bound on the sum of interest is given by
\[
\sum_{i \in A} \Big( \sum_{j=1}^{N} \abs{J_{ij}} - \max_{1 \leq j \leq N } |J_{ij}| \Big) < 100C N^{1-\alpha \rho + 1/(2\alpha) }\,,
\]
which gives the desired result.    
\end{proof}



{ 
\begin{lemma}\label{lem:spacing-of-J}
Suppose that $\gamma_0 < \gamma < \frac{1}{\alpha}$. For $l \geq 0$, let $I_l$ be given by: 
 \begin{align*}
        I_l &= [N^{\gamma} + l N^{\gamma - \xi}, N^{\gamma} + (l+1) N^{\gamma - \xi}]\,, \qquad \text{and} \qquad M_l = | \{ i<j : \abs{J_{ij}} \in I_l \}| \,.
    \end{align*}
    Then with probability $1-o(1)$, there is no $l$ with $M_l + M_{l+1} \ge 2$. 
\end{lemma}
\begin{proof}
 Since $d\nu_{\alpha,N}$ is decreasing on $\mathbb R_+$, for any $l \geq 0$ one has, 
\begin{align*}
     \p_\J ( \abs{J_{12}} \in I_l) &\le N^{\gamma -\xi  } \frac{d\nu_{\alpha,N}}{dx}(N^\gamma + lN^{\gamma -\xi }) 
     \\
     &= \alpha \frac{ N^{\gamma-\xi} }{N(N^{\gamma}+lN^{ \gamma-\xi  })^{1+\alpha}} 
     \\
     &\leq \frac{\alpha N^{\gamma-\xi}}{N \max( N^{(1+\alpha)\gamma}, (l N^{\gamma-\xi  })^{1+\alpha} )  ) }\,.
\end{align*}
    Consequently by union bounding over all pairs, and all intervals $I_l$ for $l \geq 0$ we have: 
     \begin{align*}
        \p_{\mathbf{J}} ( \exists l  \geq 0, & \, M_l + M_{l+1} \geq 2 ) \\
        &\leq 2 \binom{N}{2}^2 \sum_{l=0}^{\infty} \Big(\p_{\mathbf{J}} ( \abs{J_{12}} \in I_l, \abs{J_{13}} \in I_l) + \p_{\mathbf{J}} ( \abs{J_{12}} \in I_l, \abs{J_{13}} \in I_{l+1} )\Big) \\ 
        &\le  4 \binom{N}{2}^2 \sum_{l=0}^{\infty} \p_{\mathbf{J}} ( \abs{J_{12}} \in I_l)^2 \, .
    \end{align*}
    Let $\ell^{\ast} = N^{\xi}$, then breaking the sum up for $l < \ell^{\ast}$ and $l \geq \ell^{\ast}$ we have:
    \begin{align*}
        4 \binom{N}{2}^2 \sum_{ l< \ell^{\ast} } \p_{\mathbf{J}} (\abs{J_{12}} \in I_l)^2 &\leq   \frac{4 \alpha^2   N^{4 + 2( \gamma-\xi) +\xi } }{N^{2+2(1+\alpha)\gamma}  }
        \\
        &= \frac{4\alpha^2 }{N^{ 2\alpha \gamma +\xi - 2 }} \, , 
    \end{align*}
which tends to zero {by the facts that $\gamma>\gamma_0$ and the definition of $\xi$--see the comment immediately following~\eqref{eq:xi}.}
Similarly for $l \ge \ell^{\ast}$ one has
\begin{align*}
     4 \binom{N}{2}^2 \sum_{ l \ge \ell^{\ast} } \p_{\mathbf{J}} (\abs{J_{12}} \in I_l)^2 &\leq \sum_{l \ge \ell^{\ast} }  \frac{4 \alpha^2  N^{4 + 2(\gamma-\xi)  } }{ N^{2} N^{2(1+\alpha)( \gamma-\xi ) } l^{2(1+\alpha)}    }
     \\
     &=  \frac{4\alpha^2 N^{2 + 2(\gamma-\xi ) } }{  N^{2(1+\alpha) \gamma}}  \sum_{l \ge \ell^{\ast} }  \frac{1}{ (l/\ell^{\ast})^{2(1+\alpha)}} \, .
\end{align*}
The sum in the last line above is bounded as:
\[
\sum_{l \geq l^{\ast} } \frac{1}{ (l/\ell^{\ast} )^{2(1+\alpha)} } \leq \ell^{\ast} \sum_{r \geq 1} \frac{1}{r^{2(1+\alpha)} } \, .
\]
Plugging in the definition of $\ell^{\ast}$ we then have: 
\[
 4 \binom{N}{2}^2 \sum_{ l \ge \ell^{\ast} } \p_{\mathbf{J}} (\abs{J_{12}} \in I_l){^2}
 \leq 
  \frac{4 \alpha^2  }{N^{2\alpha\gamma +\xi -2  } } \sum_{r \geq 1} \frac{1}{r^{2(1+\alpha)}} \, ,
\]
and again the conditions on $\gamma,\xi$ imply this quantity is $o(1)$. Combining both estimates completes the proof.     
\end{proof}
}

We conclude this section with a (known) straightforward consequence of Lemma~\ref{lem:energy-drop-off-J} that was shown in \cite[Theorem 2.10]{CKS}  and will be used later. 

\begin{theorem}\label{THM:Allignment-of-Bonds} Fix $\beta>0$, $\gamma \in (\frac{1}{2\alpha} + \frac{1}{1+\alpha}, \frac{1}{\alpha})$, $a >0$,  and let $\sat (\sigma) = \{(i,j): \sgn{J_{ij} \sigma_i \sigma_j} =1 \}$. There exists a constant  $c=c(\beta,\alpha,a)>0$, such that with $\p_\J$ probability $1-o(1)$ we have  
    \[
 \pi_{\beta,\J} (  E_{a,\gamma} \subset \sat(\sigma)   )   > 1-e^{-cN^{\gamma}} \,.
    \]
    
\end{theorem}
For completeness, we provide a short proof of Theorem \ref{THM:Allignment-of-Bonds} in Appendix \ref{AP:Allignment}.

\section{Bounding the mixing time in a well} \label{sec:mixing-within-well}
Towards proving Theorem~\ref{thm:mixing-within-a-well} we start by bounding the mixing time of single site Glauber dynamics restricted to stay in a well $\mathcal C$. In what follows, for a fixed $\gamma \in (\frac{1}{2\alpha} + \frac{1}{1+\alpha}, \frac{1}{\alpha} )$ and $a \geq 1$, we shall write $K= \abs{E_{a,\gamma} }$. Recall that we use $J_{(l)}$ for the $l$'th largest bond, $e_l$ for the associated edge, and $v_l,w_l$ for the ordered pair of endpoints of $e_l$ ({recall that we are working on the high probability event} that for $l\le K$, the edges $e_1,...,e_K$ are vertex disjoint, and thus there is no confusion here).

Throughout this section, we fix a realization of $\J$ satisfying the bounds of Section~\ref{sec:coupling-matrix-preliminaries} (note, this happens with $\p_\J$ probability $1-o(1)$). 
Our goal in this section is the following.

\begin{theorem} \label{thm:mixing-below-frozen-bond} 
    Let ${ 0 \leq L \leq K }$ be a fixed index, fix an assignment $\sigma_{v_i},\sigma_{w_i}$ for $i \leq L$ and let $\mathcal{C}$ be the set of all configurations in $\{ \pm 1 \}^{N}$ taking those values on $\{v_i,w_i \}_{i \leq L}$. 

    Let $\ol{X}(t)$ denote Glauber dynamics restricted to $\mathcal{C}$, i.e., initialized in $\mathcal{C}$ and rejecting all updates to vertices $\{v_i,w_i\}_{i \leq L}$, and let $\tmix^{\mathcal C}$ denote its mixing time (see~\eqref{eq:tmix} for a formal definition). Then there is $D(\beta,\alpha)>0$ so that one has 
    \[
\tmix^{\mathcal C} \leq  \exp(2\beta \abs{J_{(L+1)}} + D N^{\frac{1}{2}+\frac{1}{2\alpha} } )\,.
    \]
\end{theorem}

(Notice that the mixing time upper bound in Theorem~\ref{thm:mixing-within-a-well} is a special case of Theorem~\ref{thm:mixing-below-frozen-bond}, when you combine this with item (1) of Lemma \ref{lem:gaps}; the bound on the escape time essentially follows also from Lemma~\ref{lem:gaps}, and is formally  a consequence of Lemma~\ref{lem:spin-structure}.)

\subsection{Mixing time and spectral gap tools}\label{subsec:mixing-time-prelims}
We quickly recall formal definitions of the mixing time, inverse spectral gap, and their relationship, and introduce a few tools from Markov chain mixing time analysis to which we will appeal. See~\cite{LP} for more background on Markov chain mixing times. Consider a continuous-time Markov chain with transition rates $P 
= {(p_{\sigma,\sigma'})_{\sigma\ne \sigma'}}$ {with diagonal such that all row-sums are $N$}, reversible with respect to a stationary distribution $\pi$. The mixing time of this chain is defined as 
\begin{align}\label{eq:tmix}
    \tmix(\varepsilon)= \inf\{t>0: \max_{x_0} d_\tv ( \mathbb P_{x_0}(X_t \in \cdot) , \pi) <\varepsilon\}\,,
\end{align}
where we are using the notation $\mathbb P_{x_0}(X_t\in \cdot)$ as the law of the Markov chain initialized from $x_0$.

It is standard to define $\tmix = \tmix(1/4)$ because a boosting argument using sub-multiplicativity of the total-variation distnace to stationarity shows that for $\varepsilon <1/4$, one has 
\begin{align}\label{eq:tv-distance-submultiplicativity}
\tmix(\varepsilon) \le \tmix(1/4) \log_2 (1/\varepsilon)\,.
\end{align}
A closely related quantity is the inverse spectral gap. The spectral gap, which we denote by $\gap$, is defined as the smallest non-zero eigenvalue of the operator ${L}$ where $L =  NI - P$. Equivalently, it has a variational characterization 
\begin{align}\label{eq:gap-variational-form}
    \gap = \inf_{f: \text{Var}_{\pi}(f) \ne 0} \frac{\mathcal E(f,f)}{\text{Var}_\pi(f)}\,, \qquad \text{where} \qquad \mathcal E(f,f) = \sum_{\sigma,  \sigma'} \pi(\sigma) p_{\sigma,\sigma'}(f(\sigma')  - f(\sigma))^2\,,
\end{align}
is the Dirichlet form. We then have the following relation between the mixing time and the inverse spectral gap: 
\begin{align}\label{eq:tmix-trel-comparison}
	(\gap^{-1}-1) \log\big(\tfrac{1}{2\epsilon}\big) \le \tmix(\epsilon) \le \gap^{-1} \log\big(\tfrac{1}{\epsilon \pi_{\min}}\big)\,,
\end{align}
where $\pi_{\min} = \min_{x} \pi(x)$. In particular, up to polynomial factors in $N$, it suffices to upper bound $\gap^{-1}$ to upper bound $\tmix$. 

A key tool for proving mixing time upper bounds, that we will rely on for proving Theorem~\ref{thm:mixing-below-frozen-bond} is what is known as \emph{block dynamics}. 

\begin{definition}
    Suppose $B_1,..., B_k$ {is a (not necessarily disjoint) cover of} the underlying vertex set $V$. The block dynamics $\mathcal B$ on $V$ is the continuous-time Markov process that assigns the blocks $B_i$ independent Poisson clocks of rate $1$, and when the clock of block $B_i$ rings, the configuration on $B_i$ is resampled conditional on the configuration on $V\setminus B_i$. 
\end{definition}

The following bound compares the inverse spectral gap of the Glauber dynamics to that of a block dynamics Markov chain, times the worst possible inverse gap of the individual blocks. 

\begin{lemma}[{\cite[Proposition 3.4]{FM}}]\label{lem:block-dynamics}
    Suppose $\mathcal B$ is a block dynamics with blocks $B_1,...,B_k$. Let $\mathsf{gap}_{B_i^\zeta}$ be the inverse gap of Glauber dynamics on $B_i$ with boundary conditions $\zeta$, and let $\mathsf{gap}_{\mathcal B}$ be the inverse gap of the block dynamics. Then, 
    \begin{align*}
        \gap_{V}^{-1} \le \chi_{\mathcal B}\, \gap_{\mathcal B}^{-1} \max_{i} \max_{\zeta} \gap_{B_i^\zeta}^{-1}\,,
    \end{align*}
    where $\chi_{\mathcal B} = \max_v |\{i: v\in B_i\}|$. 
\end{lemma}

\subsection{Inverse gap of the small bond block}
In what follows we shall leave the index $L$ fixed. We define three intervals $I_{V_1},I_{V_2},I_{V_3}$  by:
\begin{align}\label{eq:I-intervals}
 I_{V_1} = [0,N^{\rho})\,, \qquad I_{V_2} = [N^{\rho}, \abs{J_{(L+1)}} ]\,,  \qquad \text{and} \qquad I_{V_3} =  (\abs{J_{(L+1)}}, \infty)\,,
\end{align}
where $\rho=\rho(\gamma)$ is as in \eqref{eq:rho-def}. We split the vertex set into three blocks $V_1,V_2,V_3$ where $V_i$ consists of all vertices whose largest incident bond lies in $I_{V_i}$. {By item (2) of Lemma~\ref{lem:energy-drop-off-J}, for $i=2,3$ it is equivalent to just ask that it have an incident bond in $I_{V_i}$, and $I_{V_1}$ will be the remaining vertices.} A configuration $\sigma$ then naturally splits into its projections onto these blocks, $\sigma= (\sigma_{V_1},\sigma_{V_2},\sigma_{V_3})$. 

Note that by definition we have $V_3 = \bigcup_{i \leq L } \{v_i,w_i \}$. In proving Theorem~\ref{thm:mixing-below-frozen-bond}, we will be freezing an arbitrary configuration $\sigma_{V_3}$, dictating the well $\mathcal C$, and using the block dynamics on $\sigma_{V_1},\sigma_{V_2}$ with blocks $V_1,V_2$ to bound the mixing time of $\ol{X}(t)$. 

In this subsection, we bound the inverse spectral gap of Glauber dynamics run on the spins in $\sigma_{V_1}$ with arbitrary boundary conditions $\tau$ on $V_2 \cup V_3$, denoted $\gap_{V_1^\tau}^{-1}$. This can be bounded crudely by taking exponential in the maximum value of the Hamiltonian, because the sum of all bonds in $I_{V_1}$ is still smaller order than $|J_{(L+1)}|$. 

\begin{proposition} \label{Prop:Spectral-gap-on-complement}
Consider Glauber dynamics run on the spins in $V_1$ subject to boundary conditions $\tau = (\sigma_{V_2},\sigma_{V_3})$. Then there is a constant $C_1(\alpha,\beta)$ so that the inverse spectral gap $\gap_{V_1^\tau}$ satisfies,
\[
\gap_{V_1^\tau}^{-1} \leq  \exp( C_1 N^{1-\alpha \rho + \rho} \log N)\,.
\]
\end{proposition}

\begin{proof} The Gibbs' measure conditioned on boundary conditions $\tau$ has relative weights 
\[\pi^{\tau} (\sigma_{V_1}) = \frac{1}{Z_{N,\tau}} \exp(  \beta H_{1}^{\tau}(\sigma_{V_1})) \, ,
\]
where $Z_{N,\tau}$ is its partition function to make it a probability distribution, and 
    \[
    H_{1}^{\tau} (\sigma_{V_1}) = \sum_{i,j \in V_1} J_{ij} \sigma_i \sigma_j + \sum_{i \in V_1} \sum_{j \in V_{2} \cup V_{3}} J_{ij} \sigma_i \tau_j \,.
    \]
    Lemma~\ref{lem:energy-drop-off-J} part (3) implies that this Hamiltonian satisfies a uniform upper bound,
    \begin{align}\label{eq:H-upper-bound}
    \max_{\sigma_{V_1}}\abs{H^{\tau}_1 (\sigma_{V_1} ) } \leq  C_{\alpha} N^{1+\rho -\alpha \rho} \log N \,.
    \end{align}
    We will bound $\gap_{V_1^\tau}^{-1}$ by exponential in the right-hand side of~\eqref{eq:H-upper-bound} via spectral gap comparison.  
    Let $\ol{\mathcal{E}}$ be the Dirichlet form of the continuous-time random walk on the hypercube with vertex set $V_1$, and let $\mathcal{E}$ denote the Dirichlet form of the Glauber dynamics run on $V_1$. That is, 
    \begin{align*}
        \ol{\mathcal{E}}(f) &= \sum_{\substack{x \sim y}} (f(x)-f(y))^2 \frac{1}{ 2^{\abs{V_1}}}\,,  \qquad \text{and}\qquad  \mathcal{E}(f) = \sum_{\substack{x \sim y}} (f(x)-f(y))^2  \frac{\pi^{\tau}(x) \pi^{\tau}(y)}{\pi^{\tau}(x)+\pi^{\tau}(y)}\,.
    \end{align*}
In order to compare $\ol{\mathcal E}$ and $\mathcal E$, it suffices to upper bound 
\begin{align*}
 \frac{1}{ 2^{\abs{V_1}}} \Big( \frac{\pi^{\tau}(x) \pi^{\tau}(y)}{\pi^{\tau}(x)+\pi^{\tau}(y)} \Big)^{-1}
 & \le  \frac{1}{2^{\abs{V_1}}} \max_{\substack{x \sim y }} \Big(\frac{1}{\pi^{\tau}(x)} + \frac{1}{\pi^{\tau}(y)}\Big) \\ 
 & = { \frac{Z_{N,\tau}}{2^{\abs{V_1}}} \max_{ x\sim y}\Big( \exp(-\beta H_{1}^{\tau}(x)) + \exp(-\beta H_{1}^{\tau}(y) )\Big) }
 \\
 &\leq 2\frac{Z_{N,\tau}}{2^{\abs{V_1}}} \exp(C_\alpha \beta N^{1+\rho -\alpha \rho} \log N ) \, ,
\end{align*}
where the last line follows by~\eqref{eq:H-upper-bound}.  

Comparing variances between the uniform measure on $\{\pm 1\}^{V_1}$ and $\pi^{\tau}$ by similarly using~\eqref{eq:H-upper-bound} and using the variational form of the spectral gap~\eqref{eq:gap-variational-form}, we get a cancellation between the terms $Z_{N,\tau}/2^{|V_1|}$, leaving us with {(see e.g.,~\cite[Lemma 13.18]{LP})   } 
\[
\gap_{V_1^\tau}^{-1} \leq  2 \exp( C_\alpha \beta N^{1+\rho -\alpha \rho} \log N ) \gap_{V_1,\mathsf{srw}}^{-1}\,,
\]
where $\gap_{V_1,\mathsf{srw}}$ is the spectral gap of continuous-time random walk on $\{\pm 1\}^{V_1}$. 
This is in turn exactly $1$ by diagonalizing, and so, we get  the claimed bound. 
\end{proof}

\subsection{Inverse gap of the intermediate bond block} 
We next bound the inverse gap of the chain restricted to update $\sigma_{V_2}$, uniformly in the choice of the boundary values $\tau$ on $\sigma_{V_1},\sigma_{V_3}$. This bound uses the sparsity of the interactions within $V_2$ to argue that its inverse gap is dominated by its worst transition time, which is exponential in $2\beta |J_{(L+1)}|$. 

\begin{proposition} \label{Prop:Largeish-bonds-spectral-gap}
   Consider Glauber dynamics run on the spins in $V_2$ with arbitrary boundary conditions $\tau= (\sigma_{V_1},\sigma_{V_3})$. There exists $C_2(\alpha,\beta)>0$ such that,
  \[
  \max_{\tau} \gap_{V_2^\tau}^{-1} \leq  \exp \left( 2\beta \abs{J_{(L+1)}} +C_2N^{\frac{1}{2}+\frac{1}{2\alpha}} \right)\,.
  \]
\end{proposition}

\begin{proof}
Fix $i_1<i_2$ so that $\abs{J_{i_1 i_2}} \in I_{V_2}$ and consider the 2-spin Hamiltonian given by: 
\begin{align}\label{eq:4-state-chain-Hamiltonian}
H_{i_1 i_2} (\sigma) =  J_{i_1 i_2} \sigma_{i_1} \sigma_{i_2} +  \sum_{l=1}^{2} \sum_{j\in V_1 \cup V_3} \tau_j J_{i_l j} \sigma_{i_l}\,.
\end{align}
The Glauber dynamics run at inverse temperature $\beta$ corresponding to this Hamiltonian, on $\{\pm 1\}^{\{i_1,i_2\}}$ is a $4$-state Markov chain, which is easily analyzed explicitly. In particular, by the method of canonical paths for the Glauber Dynamics with this Hamiltonian, 
there is $c>0$ so that its spectral gap $\gap_{i_1,i_2}$ satisfies (See Appendix \ref{MCP} for a proof) 
\begin{align}\label{eq:4-state-chain-gap}
\gap_{i_1,i_2}^{-1} \leq  C \exp \Big( 2\beta \max_{l\in \{1,2\}} \sum_{j} \abs{J_{i_l j} } \Big)\,.
\end{align}
By Lemma~\ref{lem:energy-drop-off-J} part (2), the choice of $\gamma$, and the definition of $V_2$, this sum is uniformly bounded over all pairs $(i_1,i_2)\in V_2$ by $\abs{J_{(L+1)}}+CN^{1/2\alpha}$ for some other $C$. Consequently, absorbing the prefactor coefficient $C$ into the one in the exponent, for all $i_1,i_2$ with $\abs{J_{i_1 i_2}} \in I_{V_2}$: 
\[
\gap_{i_1,i_2}^{-1} \leq  \exp \Big( 2\beta \abs{J_{(L+1)}}+2\beta CN^{1/2\alpha} \Big)\,.
\]
Consider the continuous-time product chain over $\{\pm 1\}^{\{v,w\}}$ for $v,w\in V_2$ having $|J_{vw}|\in I_{V_2}$. By tensorization of the spectral gap, its spectral gap (which we denote by $\gap_{\otimes}$) satisfies 
\begin{align}\label{eq:product-chain-bound}
\gap_{\otimes}^{-1}  = \max_{i_1,i_2\in V_2} \gap_{i_1,i_2}^{-1} \leq   \exp \left( 2\beta \abs{J_{(L+1)}}+2\beta CN^{1/2\alpha} \right)\,.
\end{align}
The difference between this product chain and the true Glauber dynamics on $V_2$ given $\tau$ on $V_1 \cup V_3$ is that we have disregarded all the interactions between pairs of vertices  $v,w\in V_2$ that do not have $|J_{vw}|\in I_{V_2}$. In particular, the product chain is a Glauber dynamics on $V_2$ for the following Hamiltonian, 
\[
H_{\otimes}(\sigma) = \sum_{\substack{i_1,i_2: |J_{i_1 i_2}| \in I_{V_2}}} H_{i_1 i_2} (\sigma)\,.
\]
We use a spectral gap comparison to translate this into a bound on the inverse gap of the Glauber dynamics on $V_2^\tau$ itself. Let us write 
\[
 H_2(\sigma)= H_2^{\tau}(\sigma):= H_{\otimes}(\sigma) + F(\sigma) \qquad \text{where} \qquad  F(\sigma) = \sum_{\substack{i_1,j_1 \in V_{2} :  \abs{J_{i_1 j_1} } \not \in I_{V_2}}}  J_{i_1 j_1} \sigma_{i_1} \sigma_{j_1}\,.
\] 
We will compare the Dirichlet forms of Glauber dynamics for Hamiltonians $H_2^{\tau}$ and $H_{\otimes}$ to get a bound on $\gap_{V_2^\tau}^{-1}$. Let $\sigma,\sigma'$ be configurations differing in one coordinate of ${V_2}$. Write $H_2,H_2'$, $F,F'$, and $H_{\otimes},H_{\otimes}'$ respectively for these Hamiltonians evaluated at either $\sigma$ or $\sigma'$. Also, let  $Z_2,Z_{\otimes}$ be the respective partition functions. In order to show for some $R_N$ that $\mathcal E_{\otimes}(f,f) \le R_N \mathcal E(f,f)$, it suffices to establish that for every $\sigma, \tau$,  
\[
\frac{1}{Z_\otimes} \frac{e^{\beta H_\otimes} e^{\beta H_\otimes'}}{e^{\beta H_\otimes} + e^{\beta H_\otimes'} } \le  \frac{R_N}{Z_2} \frac{e^{\beta H_2} e^{\beta H_2'}}{e^{\beta H_2} + e^{\beta H_2'} } \,.
\]
Equivalently (with a bit of algebra) we can use,
\begin{align*}
R_N = \max_{\tau} \max_{\sigma, \sigma'}\frac{Z_2}{Z_\otimes} \Big( \frac{e^{-\beta F'}}{1+e^{\beta (H_\otimes'-H_\otimes)}} + \frac{e^{-\beta F}}{1+e^{\beta(H_\otimes-H_\otimes')}} \Big)\,.
\end{align*} 
where the maximum over $\sigma,\sigma'$ is over configurations on $V_2$ differing on only one coordinate. 
From a variance comparison and the variational form of the spectral gap, {  (see e.g., \cite[Lemma 13.18]{LP}) } 
\[
\gap_{V_2^\tau}^{-1} \leq  \Big(\max_{\sigma\in \{\pm 1\}^{V_2}} \frac{\pi^{\tau}(\sigma)}{\pi_\otimes(\sigma)}\Big) R_N \gap_\otimes^{-1}\,.
\]
Note that $R_N$ is bounded by $\frac{2Z_2}{Z_\otimes} \max_{\sigma} e^{-\beta F(\sigma)}$. 
Canceling the ratio of the partition functions with those coming from the ratio of $\pi^\tau$ to $\pi_{\otimes}$, we get the upper bound, 
\[
\gap_{V_2^\tau}^{-1} \leq 2 \big(\max_{\sigma} e^{ \beta F(\sigma)}\big)\big(\max_{\sigma} e^{-\beta F(\sigma)} \big) \gap_\otimes^{-1}\,.
\]
So now to bound $\max_{\sigma} F(\sigma)$, note we can replace the $J_{ij}$'s appearing in $F$ with their absolute values. By Lemma~\ref{lem:gaps} the sum of the absolute values of the $J_{ij}$'s appearing in $F$ is bounded by $D N^{1-\alpha \rho + 1/(2\alpha)}$, so the product of the maximums of $e^{F(\sigma)}$ and $e^{-F(\sigma)}$ is at most $\exp(2D N^{1-\alpha \rho + 1/(2\alpha)})$. Combining with~\eqref{eq:product-chain-bound}, and using $1-\alpha \rho < 1/2$, we get uniformly over $\tau$, 
\[
(\gap_{V_2^\tau})^{-1} \leq  \exp( 2\beta \abs{J_{(L+1)} } + C_2 N^{\frac{1}{2}+\frac{1}{2\alpha}} ) \, ,
\]
for $C_2$ a constant depending only on $\beta$ and $\alpha$.
\end{proof}

\subsection{Bounding the block dynamics} ~\label{sec:mixing-block-estimate}
The remaining step is to bound the inverse gap of the block dynamics with blocks $V_1,V_2$, whence we could use Lemma~\ref{lem:block-dynamics} to complete the proof of Theorem~\ref{thm:mixing-below-frozen-bond}.

\begin{proof}[Proof of Theorem~\ref{thm:mixing-below-frozen-bond}]
    By Lemma~\ref{lem:block-dynamics}, together with the bounds of Propositions~\ref{Prop:Spectral-gap-on-complement}--\ref{Prop:Largeish-bonds-spectral-gap}, the spectral gap on the single site Glauber dynamics on $V_1 \cup V_2$ with boundary condition $\eta$ for $V_3$, equivalently restriction of the configuration to the well $\mathcal{C}$ induced by $\eta$ (written $\gap_{V_1 \cup V_2}^{\eta}$), is bounded by 
    \begin{align}\label{eq:Glauber-block-bound}
 (\gap_{V_1 \cup V_2}^{\eta})^{-1} & \leq 2 \gap_{\mathcal{B} }^{-1} \max( \max_{\tau} (\gap_{V_1^\tau} )^{-1}, \max_{\tau} (\gap_{V_2^\tau})^{-1}) \nonumber \\
 & \leq  \exp\big( 2\beta \abs{J_{(L+1)} } +  C_2 N^{\frac{1}{2}+ \frac{1}{2\alpha}} \big) \gap_{\mathcal{B}}^{-1} \, ,
    \end{align}
where we used the fact that $1-\alpha \rho + \rho <\gamma$ and $N^{\gamma} = O(|J_{(L+1)}|)$.
 To bound the inverse spectral gap of the block dynamics, by~\eqref{eq:tmix-trel-comparison}, it is sufficient to bound its mixing time. 
    Towards that end, let us construct a coupling of the block dynamics chains started from any pair of configurations $\sigma=(\sigma_{V_1},\sigma_{V_2})$ and $\sigma' = (\sigma'_{V_1}, \sigma'_{V_2})$ in $\mathcal C$. Write $X(t)= (X_{V_1}(t),X_{V_2}(t))$ and $X'(t) = (X'_{V_1}(t), X'_{V_2}(t))$ for the block dynamics processes started from $\sigma$ and $\sigma'$. 

     For two block dynamics chains with initializations $\sigma,\sigma'$, we couple them as follows: assign the blocks $V_1,V_{2}$ exponential rate-$1$ clocks, and when the clock of block $V_i$ rings, resample $V_i$ in both chains in the following manner depending on which block was chosen: 
    \begin{itemize}
        \item $V_1$: If $X_{V_2} \equiv X_{V_2}'$, use the same randomness to draw their configurations on $V_1$ so that they also agree on $V_1$; else, sample on $V_1$ independently. 
        \item $V_{2}$: Noticing that $V_2$ is decomposed as $(\{v_{l},w_l\})_{l\ge L+1}$. Iteratively, for $l \geq L+1$, sample $\{X_{v_l}, X_{w_l}\}$ and $\{X_{v_l}',X_{w_l}'\}$ from the stationary measure conditional on $\eta$ and their respective configurations on $V_{1}\cup \bigcup_{L+1\le j<l}\{v_j,w_j\}$, using the optimal coupling on each such sample.  
    \end{itemize}
    By construction, under an event which we call $A$, that in the time window $[0,1]$ at least two block updates occur, and
    \begin{enumerate}
        \item The first update occurs in block $V_{2}$ (say at time $t$) after which $X_{V_2}(t) \equiv X_{V_2}'(t)$, and 
        \item The second update occurs (say at time $t'$) in block $V_1$,
    \end{enumerate}
    we will have that $X(s) \equiv X'(s)$ for all $s\ge t'$. By the Markov property, we can reattempt for the above event in the time window $[i,i+1]$ for all $i$, until it occurs. The mixing time is therefore bounded by 
    $$\min\{t \ge 0: (1-p_A)^t< 1/4\} = O(p_{A}^{-1})\,,
    $$
    where we use $p_A$ to be the minimal (over pairs of initializations $\sigma,\sigma'$) probability of $A$ occurring. Our task turns to lower bounding $p_A$. The probability that there are at least two updates in $[0,1]$, the first being on $V_2$ and the second being on $V_1$ is at least some universal constant, say $1/10$. 

    Given the times of the clock rings satisfying the above requirement, call them $\tau_{V_2}<\tau_{V_1}$. Under the coupling, the event that $X_{V_2}(\tau_{V_2}) = X'_{V_2}(\tau_{V_2})$ is evidently independent of the clock ring times, and so we just need to lower bound that probability under the coupling of the samples drawn on $V_2$. By construction of the coupling on an update on $V_2$, we have uniformly over the configurations at time $\tau_{V_2}^-$ (given by filtration $\mathcal F_{\tau_{V_2}^-}$), 
    \begin{align}\label{eq:V-2-coupling-prob}
        \mathbb  P(X_{V_2}(\tau_{V_2}) & \ne X_{V_2}'({\tau_{V_2}})  \mid \mathcal F_{\tau_{V_2}^-}, \mathcal T) \nonumber \\
        & \le 1- \prod_{l \ge L+1} \Big(1-\max_{\zeta,\zeta'} \TV{\pi^{\eta,\zeta}(\sigma(\{\sigma_{v_l},\sigma_{w_l}\}))}{ \pi^{\eta,\zeta'}(\sigma(\{\sigma_{v_l},\sigma_{w_l}\}))}\Big)\,,
    \end{align}
    where the conditioning on $\mathcal T$ is on the clock ring sequence, and where    
    the maximum is over $\zeta,\zeta'$ that are configurations on the complement of $\{v_l,w_l\}$ in $V_1 \cup V_2$. To arrive at this, we used the fact that the optimal coupling attains the total-variation distance, and the total-variation distance when we haven't conditioned on certain sites in $V_{2}$ (those with index bigger than $l$) is at most the distance when taking a worst-possible pair of configurations on them.

We now bound the total-variation distances in~\eqref{eq:V-2-coupling-prob}. 
For notational simplicity, let us call the $v_l,w_l$ pair we're considering $x,y$ and replace the induced boundary conditions by $\tau = (\eta,\zeta)$ and $\tau' = (\eta,\zeta')$.  
Let us assume without loss of generality that $J_{xy} > 0$. Then the worst case for the total-variation distance occurs if $\tau$ and $\tau'$ are such that for every $v$,  
\[
\sgn{ J_{xv} \tau_v} = \sgn{J_{yv} \tau_v }=1 \qquad \text{and} \qquad \sgn{ J_{xv} \tau_v' } = \sgn{ J_{yv} \tau_v' } = -1 \,.
\]
For notational brevity, in what follows, write $\tau_{ab}$ for the configuration $\sigma$ equal to $\tau$ for all $v \neq x,y$ and $\sigma_x=a, \sigma_y = b$. In this case we may write out the total-variation distance as 
    \[d_\tv(\pi((\sigma_x,\sigma_y)\in \cdot | \tau),\pi( (\sigma_x,\sigma_y) \in \cdot | \tau' ){)}
    =  \frac{1}{2}\sum_{ a,b = \pm 1} \abs{\pi( \tau_{ab} | \tau ) - \pi ( \tau'_{ab} | \tau')}\,.
    \]
The main contribution will come from the $a=b$ terms so we deal with those first. E.g., the $a=b=+1$ term on the right-hand side above becomes 
\begin{align*}
    \frac{e^{\beta H(\tau_{++})}}{e^{\beta H(\tau_{++})} + e^{\beta H(\tau_{+-})} + e^{\beta H(\tau_{-+})} + e^{\beta H(\tau_{--})}} - \frac{e^{\beta H(\tau_{++}')}}{e^{\beta H(\tau_{++}')} + e^{\beta H(\tau_{+-}')} + e^{\beta H(\tau_{-+}')} + e^{\beta H(\tau_{--}')}}\,.
\end{align*}
Denote by $L_x = 
  \sum_{v} \abs{J_{xv}}$, $L_y =  \sum_{v} \abs{J_{yv}}$ and $L_{x,y} =  \sum_{v \neq x,y} \abs{J_{xv}} + \abs{J_{yv}}$. If we factor $e^{H(\tau_{++})}$ and $e^{H(\tau_{++}')}$ out of the denominators, the above is equal to: 
\[
\frac{1}{1 + e^{-2\beta L_y} + e^{-2\beta L_x} + e^{-2\beta L_{xy}}} - \frac{1}{1+e^{2\beta L_y -4\beta \abs{J_{xy}}} +e^{2\beta L_x-4\beta \abs{J_{xy}}} + e^{2\beta L_{xy}}} = (\star)\,.
\]
Rearranging, we see that 
\begin{align*}
 (\star) &=    \frac{e^{2\beta L_y -4\beta \abs{J_{xy}}} +e^{2\beta L_x-4\beta \abs{J_{xy}}} + e^{2\beta L_{xy}} - e^{-2\beta L_y} - e^{-2\beta L_x} - e^{-2\beta L_{xy}}}{(1 + e^{-2\beta L_y} + e^{-2\beta L_x} + e^{-2\beta L_{xy}})(1+e^{2\beta L_y -4\beta \abs{J_{xy}}} +e^{2\beta L_x-4\beta \abs{J_{xy}}} + e^{2\beta L_{xy}})}
    \\
    &\leq  \frac{1 + e^{2\beta L_y-4\beta \abs{J_{xy}} - 2\beta L_{xy}} + e^{2\beta L_x - 4\beta \abs{J_{xy}} -2\beta L_{xy} } - e^{-4\beta L_{xy} }} {(1 + e^{-2\beta L_y} + e^{-2\beta L_x} + e^{-2\beta L_{xy}})} 
    \\
    &\leq  \frac{1 + e^{2\beta L_y-4\beta \abs{J_{xy}} - 2\beta L_{xy}} + e^{2\beta L_x - 4\beta \abs{J_{xy}} -2\beta L_{xy} }} {(1 + e^{-2\beta L_y} + e^{-2\beta L_x} + e^{-2\beta L_{xy}})} - \frac{1}{4} e^{-4\beta L_{xy} } \leq 1- \frac{1}{4}e^{-4\beta L_{xy}}\,.
\end{align*}
Here, the last inequality follows as  
{ 
\[
\frac{1 + e^{2\beta L_y-4\beta \abs{J_{xy}} -2\beta L_{xy}} + e^{2\beta L_x - 4\beta \abs{J_{xy}} -2\beta L_{xy} }}{(1 + e^{-2\beta L_y} + e^{-2\beta L_x} + e^{-2\beta L_{xy}})} = \frac{1+e^{-2\beta L_x} + e^{-2\beta L_y}}{1 + e^{-2\beta L_y} + e^{-2\beta L_x} + e^{-2\beta L_{xy}}} \leq 1 \, ,
\]
where we used $L_y-2\abs{J_{xy} } - L_{xy} = - L_x $. 
}
A symmetric argument deals with the term corresponding to $a= b= -1$ and gives a matching upper bound. To deal with the terms $(1,-1),(-1,1)$, we note that
\[
\frac{e^{\beta H(\tau_{+-})}}{e^{\beta H(\tau_{++})} + e^{\beta H(\tau_{+-})} + e^{\beta H(\tau_{-+})} + e^{\beta H(\tau_{--})}} \leq \frac{1}{1+e^{\beta (H(\tau_{++}{ )} -H {(\tau_{+-} )}) } } \leq \exp (-\beta N^\rho/2)\,, 
\]
where the last inequality follows from part (2) of Lemma~\ref{lem:energy-drop-off-J}. The same inequality holds for $\tau'$ and for $(-1,+1)$. 
Combining, the total-variation distance we are interested is bounded above as,
\[
\frac{1}{2} \sum_{ a,b = \pm 1} \abs{\pi( \tau_{ab} | \tau ) - \pi ( \tau'_{ab} | \tau')} \le 
1- \frac{1}{4} e^{-4\beta L_{xy} } + 2 e^{-\beta N^{\rho}/2}\,.
\]
In particular, the optimal coupling gives the probability of equality on $\{x,y\}$ of at least  
\[
\frac{1}{4} e^{-4\beta  L_{xy} } - 2 e^{-\beta N^{\rho}/2} = \frac{1}{4} e^{-4\beta L_{xy}} (1-8 e^{-\beta N^{\rho}/2 + 4 \beta L_{xy} }) \geq \frac{1}{8} e^{-4\beta  L_{xy} }\,,
\]
where again the last inequality follows from applying part (2) of Lemma~\ref{lem:energy-drop-off-J}. Taking the product over $\{v_l,w_l\}_l$ and plugging in to~\eqref{eq:V-2-coupling-prob}, we get that 
\[
 \mathbb P(X_{V_2}(\tau_{V_2}) \ne X_{V_2}'(\tau_{V_2}^-) \mid \mathcal F_{\tau_{V_2}}, \mathcal T) \ge {8^{-|\{x,y: \abs{J_{xy}} \in I_{V_2} \}| }} e^{-4\beta \sum_{x,y:|J_{x,y}| \in I_{V_2}} L_{xy} }\,.
\]
The right-hand-side times $1/10$, say, therefore lower bounds $p_A$ uniformly over $\sigma,\sigma'$, yielding an upper bound on the mixing time by inverting the right-hand-side above (up to a factor of a universal constant). Namely, 
 by applying item (2)
 of Lemma~\ref{lem:gaps},
 and noting by a Chernoff bound that with $\p_{\J}$ probability $1-o(1)$  the number of pairs $(i,j)$ with $\abs{J_{ij}} \in I_{V_2}$ is bounded by $ 8N^{1-\alpha \rho}$,
 the mixing time of the block dynamics is bounded above by, 
\[
\tmix(\mathcal{B}) \leq 4 \exp( { \beta  } C_{\alpha} N^{1-\alpha \rho +1/(2\alpha)} \log N  + 8\log(8)  N^{1-\alpha \rho } ) \leq \exp({ C }N^{\frac{1}{2}+ \frac{1}{2\alpha}} )   \,.
\]
This implies the same bound on the inverse spectral gap of the block dynamics, up to a factor of $2$ by~\eqref{eq:tmix-trel-comparison}. 
Plugging in to~\eqref{eq:Glauber-block-bound}, the inverse spectral gap of Glauber dynamics restricted to $V_1 \cup V_2$ with boundary condition $\sigma_{V_3}$ on $V_3$ is bounded by  
\[ 
\exp( 2\beta \abs{J_{(L+1)}} + D N^{\frac{1}{2\alpha}+\frac{1}{2}} ) \, ,
\]
for some $D>0$ depending only on $\beta$ and $\alpha$. 
{ To translate this to a mixing time bound, if we denote by $Z_{V_3}$ the partition function with $\sigma_{V_3}$ frozen, then $\log (Z_{V_3})$ and $\log \exp(\beta H_N(\sigma))$ are both $O(N^{1/\alpha})$. Then applying~\eqref{eq:tmix-trel-comparison} incurs a factor of $O(N^{2/\alpha})$ which can be absorbed into the constant $D$.} 
\end{proof}

{
Having proved Theorem ~\ref{thm:mixing-below-frozen-bond} we may now provide a short proof of Theorem ~\ref{thm:mixing-within-a-well}. 

\begin{proof}[Proof of Theorem~\ref{thm:mixing-within-a-well}] Fix a $t_{a,\gamma}$ and a metastable well $\mathcal C_i$ corresponding to a satisfying assignment to the top $K$ edges, and restrict the dynamics to $\mathcal C_i$. Lemma ~\ref{lem:gaps} implies that $ \abs{J_{(K-1)} } <  \frac{a}{2\beta}(N^{\gamma} - N^{\gamma-\xi} ) $ with $\p_{\mathbf{J}} $ probability $1-o(1)$. Theorem ~\ref{thm:mixing-below-frozen-bond}  with $L=K$ then implies that: 
\[
\tmix^{\mathcal{C}} \leq e^{-\frac{a}{2} N^{\gamma-\xi} } t_{a,\gamma} \, ,
\]
as claimed. 
The escape time bound follows as a consequence of Lemma ~\ref{lem:spin-structure} and the fact that by Lemma~\ref{lem:gaps}, $|J_L| \ge \frac{a}{2\beta} (N^{\gamma}+ N^{\gamma - \xi})$ with probability $1-o(1)$. 
\qedhere  
\end{proof}

}
\subsection{Analagous mixing time bounds for $Y$}

In this section we prove an analogue of Theorem~\ref{thm:mixing-below-frozen-bond} for the approximating well-process $Y$ defined in Theorem~\ref{thm:main1}. We remark that the proof is simpler than Theorem~\ref{thm:mixing-below-frozen-bond}, and follows the  approach of Proposition~\ref{Prop:Largeish-bonds-spectral-gap}. We begin by showing that $Y$ is reversible with respect to the restriction of the Gibbs' measure on $t_{a,\gamma}$-metastable wells.

\begin{lemma} \label{lem:Reversibility-of-Y} For any $\gamma \in (\frac{1}{2\alpha} + \frac{1}{1+\alpha}, \frac{1}{\alpha} )$ and any $a >0$ the process $Y(t)=Y_{a,\gamma}(t)$ is reversible with respect to the measure $\pi^Y$ assigning a  $t_{a,\gamma}$-metastable well $\mathcal C$ probability 
$$\pi^Y(\mathcal C) \propto \sum_{\sigma \in \mathcal C} \exp(\beta H (\sigma) )\,. $$ 
\end{lemma}

\begin{proof}
It will suffice to check the detail balanced equations hold. Fix a $t_{a,\gamma}$-metastable well $C$, and let $\mathcal C_+$ and $\mathcal C_-$ be two wells which differ only by one edge's satisfying assignment, say the $l$'th, with the $+$ or $-$ subscript indicating the spin taken by $\sigma_{v_l}$. 
 We shall assume without loss of generality that $J_{(l)}>0$ to simplify the presentation. By definition of the rates for $Y$ from \eqref{eq:Z(Y_s)}, it will suffice to check that, 
\[
\pi^{Y}( \mathcal{C}_+) \E_{\pi^{\mathcal{C}_+}} [Z_{v_l}(\sigma)+Z_{w_l}(\sigma)] = \pi^Y ( \mathcal{C}_-) \E_{\pi^{{\mathcal C}_-}} [Z_{v_l}(\sigma) + Z_{w_l}(\sigma)]
\]
 where $\pi^{\mathcal C}= \pi(\cdot \mid \sigma \in \mathcal C)$.
By expanding out both sides and cancelling terms this reduces to showing that, 
\[
\sum_{\sigma \in \mathcal{C}_+} (Z_{v_l}(\sigma)+Z_{w_l}(\sigma)) e^{\beta H(\sigma) } = \sum_{\tau \in \mathcal{C}_-} (Z_{v_l}(\tau)+Z_{w_l}(\tau)) e^{\beta H(\tau)} \, .
\]
Now note there is a bijection from $ \mathcal{C}_+$ to $ \mathcal{C}_-$ given by the map,
\[
\Phi_l(\sigma)_z =  
\begin{cases}
 \sigma_z &\text{if} \ z \not \in \{v_l,w_l\} \\
 -\sigma_z &\text{if} \ z \in \{v_l,w_l\} 
\end{cases} \, .
\]
In any $t_{a,\gamma}$ metastable well, $e_l$ is satisfied, and we can thus decompose the Hamiltonian $H$ as,
\[
H(\sigma) = \abs{J_{(l)}} + m_{v_l}'(\sigma) + m_{w_l}'(\sigma) + H_l(\sigma) \, ,
\]
where $m_{v}$ was defined in~\eqref{eq:def-of-Z-and-m}, $m_{v_l}'(\sigma) = m_{v_l}(\sigma) - |J_{(l)}|$, and $H_l$ is given by: 
\[
H_l(\sigma) = \sum_{v,w \not \in \{v_l,w_l\} } J_{vw} \sigma_v \sigma_w \, .
\]
 Then under the bijection $\Phi_l$, one has that 
\[
m_{v_l}'(\sigma) = - m_{v_l}' (\Phi_l(\sigma)) \, , \quad m_{w_l}'(\sigma) = - m_{w_l}' (\Phi_l(\sigma)) \, , \quad 
H_l(\sigma) = H_l(\Phi_l(\sigma)) \, .
\]
Therefore, writing,
\[
Z_{v_l}(\sigma) = \exp\left(-2\beta \abs{J_{(l)}} - 2\beta m_{v_l}'(\sigma) \right) \, , \quad Z_{w_l}(\sigma) = \exp(-2\beta \abs{J_{(l)}} - 2\beta m_{w_l}'(\sigma) ) \, , 
\]
we have that,
\begin{align*}
    \Big(Z_{v_l}(\sigma)& +Z_{w_l}(\sigma) \Big) e^{\beta H(\sigma) } \\
    &= \left(e^{-2\beta \abs{J_{(l)}} -2\beta m_{v_l}'(\sigma) } + e^{-2\beta \abs{J_{(l)}} -2\beta m_{w_l}'(\sigma)}\right)e^{\beta (\abs{J_{(l)}} + m_{v_l}'(\sigma) + m_{w_l}'(\sigma) + H_l(\sigma))} \, ,
\end{align*}
and similarly,
\begin{align*}
     \Big(Z_{v_l}(\Phi_l & (\sigma))+Z_{w_l}(\Phi_l(\sigma))\Big) e^{\beta H(\Phi_l(\sigma)) } 
     \\
     &= \left(e^{-2\beta \abs{J_{(l)}} +2\beta m_{v_l}'(\sigma) } + e^{-2\beta \abs{J_{(l)}} +2\beta m_{w_l}'(\sigma)}\right)e^{\beta (\abs{J_{(l)}} - m_{v_l}'(\sigma) - m_{w_l}'(\sigma) + H_l(\sigma))} \, .
\end{align*}
Expanding both identities above we see that: 
\[
 \Big(Z_{v_l}(\sigma)+Z_{w_l}(\sigma) \Big) e^{\beta H(\sigma) } =  \Big(Z_{v_l}(\Phi_l(\sigma))+Z_{w_l}(\Phi_l(\sigma))\Big) e^{\beta H(\Phi_l(\sigma))} \, .
\]
Since $\Phi_l$ is a bijection from $\mathcal{C}_+$ to $\mathcal{C}_-$, summing over $\sigma \in \mathcal{C}_+$ proves the result.
\end{proof}

The following is the relevant estimate for mixing of $Y$ when freezing its top $L$ coordinates, analogous to Theorem~\ref{thm:mixing-below-frozen-bond}. 

\begin{proposition} \label{Prop:Mixing-in-well-Y} 
    Let $0 \leq L < K$ be a fixed index, and fix an assignment $\varpi_{i} \in \{ \pm 1\}$ for $i \leq L$. Let $\mathcal W$ be the set of all configurations $y\in \{\pm 1\}^K$ having $y_i = \varpi_i$ for $i\le L$. 

    Let $\ol{Y}(t)$ be the process $Y$ restricted to $\mathcal W$ i.e., initialized with $Y(0)\in \mathcal W$ and rejecting all updates to coordinates $1 \leq i \leq L$, and denote its mixing time by $\tmix^{\mathcal W}$. There is $D(\beta,\alpha)>0$ so that 
    \[
 \tmix^{\mathcal W} \leq  \frac{1}{t_{a,\gamma}} \exp( 2\beta \abs{J_{(L+1)}} + DN^{\frac{1}{2}+\frac{1}{2\alpha} } ) \, . 
    \]
\end{proposition}

\begin{proof}
 Fix an $L$ and a $\mathcal W$ (associated to $\varpi_i$ for $i\le L$), and let $\varsigma_{v_i} = \varpi_i$ and $\varsigma_{w_i}$ taking the spin that satisfies $e_i$. Then for $L+1\le m\le K$, define 
 \[
F_m( {x_m }) :=  \sum_{ l \leq L } (J_{v_m v_l} \varsigma_{v_l} + J_{v_m w_l} \varsigma_{w_l}) x_m \qquad \text{for $x_m\in \{\pm 1\}$}  \, ,
 \]
which captures the effective field from the restriction to $\mathcal W$ on the site $v_m$. Now define a Markov chain $\ol{Y}_m$ with the following rule: when at state $x\in \{\pm 1\}$, follow an exponential clock of rate $t_{a,\gamma} \exp(-2\beta \abs{J_{(m)}} ) \exp( \beta F_m(x_m) )$, and when this clock rings uniformly resample $x_m$ from $\{\pm 1\}$.
Notice that this chain is stationary with respect to the measure assigning $x$ weight proportional to $\pi_m(x_m) \propto \exp({ -} \beta F_m(x_m) )$

Similarly to the proof of Proposition~\ref{Prop:Largeish-bonds-spectral-gap}, the method of canonical paths and an application of Lemma~\ref{lem:energy-drop-off-J} shows that the inverse spectral gap (denoted $\gap_{m}^{-1}$) is bounded above by:
\[
\gap_m^{-1} \leq \frac{1}{t_{a,\gamma}} \exp(2\beta \abs{J_{(m)}} + CN^{\frac{1}{2\alpha} }) \, .
\]
We now consider the product chain of these $\ol{Y}_m$ for $L< m < K$, and denote its spectral gap by $\gap_{\otimes}$. By tensorization of the spectral gap, it satisfies,
\[
\gap_{\otimes}^{-1} \leq \frac{1}{t_{a,\gamma} } \exp( 2\beta \abs{J_{(L+1)} } + CN^{\frac{1}{2\alpha}} ) \, .
\]
To complete the estimate on the spectral gap of $\ol{Y}$, which we denote by $\gap_{\mathcal{W}}$ we will compare its Dirichlet form to that of the product chain. For this let us write the rates of $\ol{Y}$ and the product chain as $\ol{c}(x,y)$ and $c_{\otimes}(x,y)$, respectively. Let $\pi_{\otimes}$ denote the product measure of the $\pi_m$, and $\pi^{\mathcal{W}}= \pi^{Y}(\cdot \mid \mathcal{W} )$.
As before we seek bound on 
\[
 \frac{c_{\otimes}(x,y) \pi_{\otimes}(x)}{\ol{c}(x,y) \pi^{\mathcal{W}}(x)}  \, .
\]
which is uniform over pairs $x,y \in \{\pm 1\}^{K-L}$ where $x\sim y$, i.e., they only differ in one coordinate.
Taking the maximum over adjacent pairs, the bound on the spectral gap that we get is given by: 
\begin{align}\label{eq:Y-gap-bound-above}
\gap_{\ol{Y}}^{-1} \leq \gap_{\otimes}^{-1} \max_{x \sim y}  \frac{c_{\otimes}(x,y) \pi_{\otimes}(x)}{\ol{c}(x,y) \pi^{\mathcal{W}}(x)} \max_{z} \frac{ \pi^{\mathcal W} (z)}{\pi_{\otimes}(z)} \, .
\end{align}
Then by item (2) of Lemma~\ref{lem:energy-drop-off-J}, the ratio of rates is uniformly bounded by 
\[
\frac{c_{\otimes}(x,y)}{\ol{c}(x,y)} \leq  \exp(2CN^{1/2\alpha} ) \, .
\]
Now we  bound the relative weights of the measures. In the bound of~\eqref{eq:Y-gap-bound-above}, the partition functions cancel, and it suffices to bound,
\begin{align}\label{eq:need-to-bound-max-Y}
 \max_{x,y} \frac{ e^{{ -} \beta \sum_{m=L+1}^{K} F_m(x_m) } }{\sum_{\sigma \in A_x} e^{\beta H(\sigma)}} \frac{ \sum_{\sigma \in A_y} e^{\beta H(\sigma)}}{e^{{ -} \beta \sum_{m=L+1}^{K} F_m(y_m)}} \, ,
\end{align}
where $A_x$  (respectively, $A_y$) are used to denote the collection of $\sigma$ in $\mathcal{W}$ such that $\sigma_{v_m}=x_m$. 
Now write $H(\sigma)$ as a sum of six terms according to three groups:
\[
H= H_{11} + H_{12} + H_{13} + H_{22} + H_{23} + H_{33} \, ,
\]
where the groups are defined as follows: 
\begin{itemize}
    \item group $1$ consists of the vertices belonging to edges $e_{1},...,e_{L}$, 
    \item group $2$ consists of the  vertices belonging to edges $e_{L+1},...,e_{K}$,
    \item group $3$ consists of all remaining spins.
\end{itemize}
With this notation, $H_{ij}$ corresponds to the total interaction between group $i$ and group $j$. When the top $L$ bonds are frozen, $H_{11}$ is independent of $\sigma$.{  We further note that for any $\sigma$ in a $t_{a,\gamma}$ metastable well one has that: 
\[
H_{22}(\sigma) = \sum_{m=L+1}^{K} \abs{J_{(m)}} + H_{22}'(\sigma) \, ,
\]
and by part (3) of Lemma ~\ref{lem:energy-drop-off-J} , we have uniform upper bounds: 
\begin{align} \label{eq:bounds-on-Hi}
\max_{\sigma \in \{ \pm 1 \}^{N} } \max(\abs{H_{12}(\sigma)},\abs{H_{13}(\sigma)},\abs{H_{22}'(\sigma)},\abs{H_{23}(\sigma)} ) \leq CN^{1-\alpha \gamma + \frac{1}{2\alpha} } \,.
\end{align} 
Therefore, the maximum in~\eqref{eq:need-to-bound-max-Y} satisfies, 
\[
 \max_{x,y} \frac{ e^{ \beta \sum_{m=L+1}^{K} F_m(x_m) } }{\sum_{\sigma \in A_x} e^{\beta H(\sigma)}} \frac{ \sum_{\sigma \in A_y} e^{\beta H(\sigma)}}{e^{\beta \sum_{m=L+1}^{K} F_m(y_m)}} 
 \leq \Big(\max_{x,y}  \frac{\sum_{\sigma \in A_x } e^{\beta H_{33}(\sigma)}}{\sum_{\sigma \in A_y } e^{\beta H_{33}(\sigma)}} \Big) e^{6\beta C N^{1-\alpha \gamma + \frac{1}{2\alpha} }} \, ,
\]
where the inequality follows from an application of \eqref{eq:bounds-on-Hi} and the cancellation of $H_{11}$ and $\sum_{m=K}^{L+1} \abs{J_{(m)} }$.
}
 The terms in the maximum depend only on the vertices in group $3$ and hence are equal. Consequently { we obtain the following bound on the spectral gap} 
\begin{align*}
\gap_{Y_L}^{-1}  & \leq  \frac{1}{t_{a,\gamma}} \exp(2\beta \abs{J_{(L+1)}} + CN^{\frac{1}{2\alpha}} + 6\beta C N^{1-\alpha \gamma + \frac{1}{2\alpha}} ) \\
& \leq \frac{1}{t_{a,\gamma}} \exp(2\beta \abs{J_{(L+1)}} + D N^{\frac{1}{2}+\frac{1}{2\alpha} }) \, .
\end{align*}
Finally, since the logarithm of the partition function is $O(N^{\frac{1}{\alpha}})$ , we may combine with~\eqref{eq:tmix-trel-comparison} to get the claimed mixing time bound. 
\end{proof}

\section{Metastability of  Glauber dynamics}\label{sec:proof-of-main-results}
We now turn to using our sharp mixing time within a well results from the previous sections to establish our main theorems. As in Section~\ref{sec:mixing-within-well}, we work on the intersection of all the $1-o(1)$ $\mathbb P_{\mathbf{J}}$-probability events of Section~\ref{sec:coupling-matrix-preliminaries}. Consequently, all results in this section are implicitly holding with $\mathbb P_{\mathbf{J}}$-probability $1-o(1)$. 

\subsection{Limit of the auto-correlation function}\label{subsec:autocorrelation-limit}
The aim of this subsection is to prove Theorem~\ref{thm:main-autocorrelation}.
For a given instance $X(t)$ running continuous-time Glauber dynamics, we define the (random) set $\sat(t)$ to be $\sat(X(t))$.

We begin with the following lemma which precisely characterizes the trajectories of $(X_i(t),X_j(t))$ when the bond value $J_{ij}$ is large in absolute value.

\begin{lemma} \label{lem:spin-structure} Fix $1 \leq L \leq K$ and let 
\begin{align}\label{eq:t-L}
t_L:= \exp(2\beta \abs{J_{(L)}} + CN^{1/2\alpha}) \, ,
\end{align}
for $C$ as in Lemma \ref{lem:energy-drop-off-J}. {Let $b>0$ and fix any $N^{2/\alpha} < T \leq e^{-bN^{\gamma-\xi}} t_L$}. 
Then, with probability at least {$1-e^{ - \frac{b}{2} N^{\gamma - \xi}}$}, the trajectory of the pairs $(X_{v_i}(t),X_{w_i}(t))$ for $1\le i \le L$ are described as follows:
\begin{itemize}
    \item If $(v_i,w_i) \in \sat(0)$ then $(X_{v_i}(t),X_{w_i}(t)) = (X_{v_i}(0),X_{w_i}(0))$ for all $t \in [0,T]$.
    \item If $(v_i,w_i) \not \in \sat(0)$ then the first attempted update at $v_i$ or $w_i$ is successful, and thereafter the pair remains constant until (at least) time $T$.
\end{itemize}
In particular, with probability at least {$1-e^{-\frac{b}{2}N^{\gamma-\xi} }$} one has $e_i \in \sat(T)$ for $1 \leq i \leq L$, and if $X(0) \sim \text{Unif}(\Sigma_N)$ then the law of $(X_{v_l}(T),X_{w_l}(T))_{l \leq L}$ has total-variation distance at most {$e^{- \frac{b}{2} N^{\gamma-\xi} }$}, to the uniform distribution on all assignments on these vertices that satisfy all edges $e_1,...,e_L$. 
\end{lemma}

\begin{proof}
    Since we are working on the events of Lemma~\ref{lem:energy-drop-off-J} and Lemma~\ref{lem:gaps}, for any configuration $\eta$ having $e_i \in \sat(\eta)$ for some $i\in \{1,...,L\}$,  
    \begin{align}
\pi_{\beta,\J}( \sigma_{v_i}=-\eta_{v_i} | (\sigma_{z})_{z\ne v_i}= (\eta_{z})_{z \neq v_i} ) &\leq \exp(-2\beta \abs{J_{(i)}} + CN^{1/2\alpha} )\,, \label{eq:prob-unsatisfy-at-update-v}\\
\pi_{\beta,\J}( \sigma_{w_i}=-\eta_{w_i} | (\sigma_{z})_{z\ne w_i}= (\eta_{z})_{z \neq w_i} ) &\leq \exp(-2\beta \abs{J_{(i)}} + CN^{1/2\alpha} ) \label{eq:prob-unsatisfy-at-update-w} \, .
    \end{align}
Now let us assume that $e_i \in \sat(0)$. In this case, by the upper bounds above, we have by splitting the Poisson process of clock rings at vertices $v_i,w_i$ that:
\[
\sup_{s\in [0,T]} \mathbf{1}\{ (X_{v_i}(0),X_{w_i}(0))\neq (X_{v_i}(s),X_{w_i}(s) ) \} \leq \mathbf{1}\big\{ \text{Poiss}\big( 2T e^{-2\beta \abs{J_{(i)}} + CN^{1/2\alpha} } \big) \geq 1 \big\} \, .
\]
 Taking expectations on both sides implies by the bound on {$T\le e^{-bN^{\gamma-\xi}} t_L$}, that
\[
\p_{X(0)}\Big( \bigcup_{s\in [0,T]}  (X_{v_i}(s),X_{w_i}(s)) \neq (X_{v_i}(0),X_{w_i}(0)) \Big) \leq 2 T\exp(-2\beta \abs{J_{(i)}} + CN^{1/2\alpha}) \, , 
\]
which is at most  {$e^{-\frac{3b}{4}N^{\gamma-\xi}}$}, providing the bound in the case $e_i \in \sat(0)$.

In what follows, let $\tin_1(i) =\inf\{t\ge 0: e_i \in \sat(t)\}$. 
If $e_i \notin \sat(0)$, the complementary bound to~\eqref{eq:prob-unsatisfy-at-update-v}--\eqref{eq:prob-unsatisfy-at-update-w} implies that except with probability $e^{ - 2\beta |J_{(i)}|+ CN^{1/2\alpha}}$, the first time one of the clocks of $v_i,w_i$ rings is equal to $\tin_1$, i.e., the update is successful and $e_i \in \sat(\tin_1)$. 

From here the strong Markov property together with the argument for the case $e_i \in \sat(0)$, implies that $X_{v_i},X_{w_i}$ will thereafter remain unchanged until (at least) time $T$. By a coupon collecting bound,  with probability at least $1- N^2 e^{-N^{2/\alpha}}$,
by time $N^{2/\alpha}$ every site has had at least one clock ring. Combining the two cases, if $\mathcal{E}_L$ is given by: 
\begin{align} \label{eq:def-of-eL}
 \mathcal{E}_L := \bigcup_{1 \leq i \leq L}\Bigg(\{\tin_1(i) >N^{2/\alpha}\}\cup \bigcup_{s\in [\tin_1(i),T]}  \{(X_{v_i}(s),X_{w_i}(s)) \ne (X_{v_i}(\tin_1),X_{w_i}(\tin_1))\}\Bigg) \, ,
\end{align}
 then,
{\begin{align}\label{eq:E_L-bound}
\p(\mathcal{E}_L) \leq e^{-\frac{b}{2}N^{\gamma-\xi}} \, .
\end{align}}
To conclude, note that on the event $\mathcal{E}_L$, from a uniform initialization, every $e_i$ with $1 \leq i \leq L$ with $e_i \in \sat(X(0))$ is equally likely to take either of its satisfying assignments at time $T$, independent of all others. For those $e_i \not \in \sat(X(0))$, we have conditional on $X(0)$, independent of all other clocks, that the first clock ring at one of $\{v_i,w_i\}$ is equally likely to be on $v_i$ as $w_i$, and that determines which of the two permissible satisfying assignments it takes. Consequently the total-variation distance to uniformity on satisfying assignments is bounded above by $\p(\mathcal{E}_L)$. 
\end{proof}

The following obvious observation will also be useful to refer to repeatedly. 
\begin{remark} \label{rem:coupling-of-frozen}
   On the event that edges $e_i$ for $1 \leq i \leq L$ satisfy $e_i \in \sat(t)$ for all $t \in [s_1,s_2]$ we have that $X(t)$ is perfectly coupled to a restricted Glauber dynamics $\ol{X}$ that rejects any updates on vertices in $e_i$ for $1 \leq i \leq L$, run for time $s_2-s_1$ from initialization $X(s_1)$.
\end{remark}

The full metastability result of Theorem~\ref{thm:main1} will require a delicate understanding of a 2-state Markov chain approximating the transits between wells given by the flips of the $L+1$'th bond, with the top $L$ bonds frozen per Remark~\ref{rem:coupling-of-frozen}. Before getting to this analysis, let us prove the easier Theorem~\ref{thm:main-autocorrelation} using the results of Section~\ref{sec:mixing-within-well} together with Lemma~\ref{lem:spin-structure} and Remark \ref{rem:coupling-of-frozen}.

\begin{proof}[Proof of Theorem~\ref{thm:main-autocorrelation}]
Let $s>1$. Throughout we use the notation introduced in Lemma \ref{lem:spin-structure}. As in \eqref{eq:def-of-eL}, define an event $\mathcal{E}$ (see equation \eqref{eq:def-of-eL} with $L=K = \abs{E_{a,\gamma}}$) by, 
\[
\mathcal{E} := \bigcup_{1 \leq i \leq K}\Bigg(\{\tin_1(i) >N^{2/\alpha}\}\cup \bigcup_{s\in [\tin_1(i),T]}  \{(X_{v_i}(s),X_{w_i}(s)) \ne (X_{v_i}(\tin_1),X_{w_i}(\tin_1))\}\Bigg) \, ,
\]
where $T=st_{a,\gamma}$. By Lemma \ref{lem:gaps} we have that $t_{a,\gamma} \leq e^{- { {b} }N^{\gamma-\xi}}t_K$ {for some $ {b} >0$} (see~\eqref{eq:t-L} with $L=K$ for the definition of $t_K$), and hence { by} ~\eqref{eq:E_L-bound},
\[
\p( \mathcal{E} ) \leq e^{-\frac{{ b} }{2}N^{\gamma-\xi} } \, .
\]
For the remainder of the proof we shall work on the complement of this event. 
Also, by Lemma~\ref{lem:spin-structure}, 
\[
\TV{ \p((X_{v_1}(N^{2/\alpha}),....,X_{v_K}(N^{2/\alpha} )) \in \cdot )}{\text{Unif}(\{ \pm 1 \}^{K}) }   \le e^{ - \frac{{ b}}{2} N^{\gamma-\xi}}\, .
\]
  By Remark~\ref{rem:coupling-of-frozen}, on $\mathcal{E}^c$ we have that $(X(t))_{t\in [N^{2/\alpha},st_{a,\gamma}]}$ is equal to Glauber dynamics freezing the bonds in $E_{a,\gamma}$ for $t \in [N^{2/\alpha}, st_{a,\gamma}]$, corresponding to restricting to the well $\mathcal C$ dictated by $X_{v_1}(N^{2/\alpha}),...,X_{v_K}(N^{2/\alpha})$. Denoting this process by $\ol{X}(t)$, by Theorem~\ref{thm:mixing-below-frozen-bond} the mixing time of $\ol{X}$ is bounded above by: 
\[
\tmix^{\mathcal{C}} \leq  \exp(2\beta \abs{J_{(K+1)}} + DN^{\frac{1}{2}+\frac{1}{2\alpha}} ) \, ,
\]
and furthermore {  Lemma~\ref{lem:gaps} }implies  for every fixed $r>0$, that
{ 
\begin{align}\label{eq:tmix-over-t_agamma}
\frac{\tmix^{\mathcal C}}{rt_{a,\gamma}-N^{2/\alpha}} &= (1+o(1)) \frac{\tmix^{\mathcal{C}}}{r \exp(aN^{\gamma} )} \nonumber\\ 
&\leq \frac{1+o(1)}{r} \exp(2\beta \abs{J_{(K+1)}} + DN^{\frac{1}{2} + \frac{1}{2\alpha}} - aN^{\gamma} ) \nonumber \\
&\leq \frac{1+o(1)}{r} \exp( a(N^{\gamma} -N^{\gamma-\xi}) + DN^{\frac{1}{2} + \frac{1}{2\alpha}} - aN^{\gamma} ) \nonumber\\
&\leq \frac{2}{r} e^{-\frac{a}{2} N^{\gamma-\xi} }\,.  
\end{align}
}
Consequently, \eqref{eq:tv-distance-submultiplicativity} with $r=1$ implies that:
\[
\TV{\p_{\text{Unif}(\Sigma_N)}( X(t_{a,\gamma}) \in \cdot)}{ \p(\sigma^1 \in \cdot)} \leq 2^{-e^{\frac{{ a}}{4} N^{\gamma-\xi}} } + 2e^{ - \frac{{ b }}{2}N^{\gamma - \xi}} \, ,
\]
where the law of $\sigma^1$ is given by the following sampling scheme:
\begin{enumerate}
    \item Draw a uniform configuration in $\{ \pm 1 \}^{V_{a,\gamma}}$ conditional on $e\in \sat(\sigma^1)$ for all $e \in E_{a,\gamma}$. 
    \item Draw the remaining coordinates from $\pi_{\beta,\J}( \cdot \mid \sigma^1(V_{a,\gamma}))$.
\end{enumerate}

Next we show that the joint law of $X(t_{a,\gamma})$ and $X(st_{a,\gamma})$ is within total-variation distance $o(1)$ of $(\sigma^{1},\sigma^2)$ where $\sigma^2$ is a copy of $\sigma^1$ where step (1) uses the same draw, but step (2) uses an independent draw. 
The event $\mathcal E^c$ already implies that for any vertex $v$ in $V_{a,\gamma}$, $X_v(t_{a,\gamma}) = X_{v}(st_{a,\gamma})$. By the Markov property and a triangle inequality, it therefore suffices to control for any $\sigma^1$, the total-variation between $(\sigma^1,\sigma^2)$ and $(\sigma^1, \ol{X}((s-1)t_{a,\gamma}))$ where $\ol{X}$ is initialized from $\sigma_1$ (and has all vertices belonging to $E_{a,\gamma}$ frozen as before).
By the Markov property and~\eqref{eq:tmix-over-t_agamma} with $r=s-1$, 
\[
\TV{\p((\sigma^1, \ol{X}((s-1)t_{a,\gamma})) \in \cdot)}{\p((\sigma^1,\sigma^2) \in \cdot)} \leq 2(2^{-e^{\frac{{ a}}{4}N^{\gamma-\xi}}} + 2e^{-\frac{{ b}}{4}N^{\gamma-\xi}}) \, .
\]
Note that the law of $\frac{1}{N} \langle \sigma^1,\sigma^2 \rangle$ is exactly $q_{a,\gamma}^{(N)}$. Recalling the definition of~\eqref{eq:auto-correlation}, the above implies 
\[
\TV{ \p( C_N(t_{a,\gamma},st_{a,\gamma}) \in \cdot )  }{\p( q_{a,\gamma}^{(N)} \in \cdot) } \leq Ce^{-cN^{\theta} } \, ,
\]
for some absolute constants $C,c,\theta>0$, concluding the proof.
\end{proof}

\subsection{Approximating exit times of wells by exponential random variables}
Towards proving Theorem~\ref{thm:main1}, in this section we establish a key step that when considering the Glauber dynamics (as well as the process $Y$) we may approximate the escape times from metastable wells by exponential random variables with a fixed rate, drawn from the stationary barrier height of that well.

\begin{lemma} \label{lem:Comparison-of-rates}  

Let $1 \leq L \leq K$ be a fixed index, fix a satisfying assignment $\sigma_{v_i},\sigma_{w_i}$ for $i\le L$ and let $\mathcal{C}$ be the set of all configurations in $\{\pm 1\}^N$ taking those values on $\{v_i,w_i\}_{i\le L}$.  

Let $\ol{X}(t)$ denote Glauber dynamics restricted to $\mathcal C$, i.e., initialized in $\mathcal C$ and rejecting all updates to vertices $\{v_i,w_i\}_{i\le L}$.   Define the following functions on $\mathcal{C}$, 
   \begin{align} \label{eq:rates-of-XL}
\lambda_{v_L}(\sigma) = \frac{e^{-\beta m_{v_L}(\sigma)}}{ e^{\beta m_{{v_L}}(\sigma)} + e^{-\beta m_{v_L}(\sigma)}} \, , \qquad \text{and} \qquad \lambda_{w_L}(\sigma) = \frac{e^{-\beta m_{w_L}(\sigma)}}{e^{\beta m_{w_L}(\sigma) } + e^{-\beta m_{w_L}(\sigma)}}\, , 
   \end{align}
and let $\lambda_L = \lambda_{v_L} + \lambda_{w_L}$ (see \eqref{eq:def-of-Z-and-m} for the definition of $m_v$). Then for all $\ol{X}(0) \in \mathcal C$, and for all   $t\geq T_0:= \exp(2\beta \abs{J_{(L)} }  -\frac{{ a}}{100} N^{\gamma-\xi} )$, one has: 
  \begin{align} \label{eq:ratio-of-rates-bound-X} 
\Bigg| 1- \frac{\E [\lambda_L(\ol{X}(t))]}{\E_{\sigma \sim \pi^{\mathcal{C}}} [\lambda_L(\sigma)]}  \Bigg| \leq e^{-\frac{{ a} }{32} N^{\gamma-\xi} }  \, .
   \end{align} 
Furthermore, the time average of $\lambda_L(\overline{X}(t))$ is well concentrated, i.e, if $T_1 = \exp(2\beta \abs{J_{(L)} } + N^{\frac{2}{\alpha} })$, for any initialization $\ol{X}(0)\in \mathcal{C}$,  

\[
\p\Bigg( \sup_{t \in [T_0,T_1]} \Bigg| \frac{\frac{1}{t} \int_{0}^{t} \lambda_L(\ol{X}(s)) ds}{ \E_{\sigma \sim \pi^{\mathcal C} } [\lambda_L(\sigma)]} -1 \Bigg| \geq { 8e^{-\frac{a}{32}N^{\gamma-\xi}} }  \Bigg) \leq  2 \exp(- e^{\frac{{a}}{32}N^{\gamma-\xi} })  \, , 
\]

\end{lemma}

\begin{proof} To simplify notation we write $\lambda = \lambda_L$ throughout the proof. 
 For the first statement, note that as $\lambda \leq 1$ we have 
\begin{align*}
\abs{ \E_{\sigma \sim \pi^{\mathcal{C} }} [\lambda(\sigma)] - \E [\lambda(\ol{X}(t))] } \leq \TV{\p( \ol{X}(t) \in \cdot)}{\pi^{\mathcal{C}}} \, .
\end{align*}
By stationarity of $\pi^{\mathcal C}$ for $\ol{X}$, Theorem \ref{thm:mixing-below-frozen-bond}, and ~\eqref{eq:tv-distance-submultiplicativity}, if $t\ge T_0$, the TV distance is bounded by,
\[
\TV{\p( \ol{X}(t) \in \cdot)}{\pi^{\mathcal{C}}} \leq 2^{-N^{2/\alpha} } \, .
\]
Since $|J_{(L)}| = O(N^{1/\alpha})$ for all $L\ge 1$, we have $\mathbb E_{\sigma \sim \pi^{\mathcal C}}[\lambda(\sigma)] \ge e^{ - O(N^{1/\alpha})}$ and therefore the above is at most 
$e^{-\frac{{ a}}{32} N^{\gamma-\xi}} \E_{\sigma \sim \pi^{\mathcal{C}}} [\lambda(\sigma)]$ 
which proves the first claim.

For the concentration, for any $t\in [T_0,T_1]$, we begin by defining $R_1 =o(R_2)$ as,  $$R_1 = N^{2/\alpha} \tmix\,, \qquad \text{and} \qquad R_2 = e^{- \frac{{ a}}{4} N^{\gamma-\xi} } t\,. $$
where $\tmix^{\mathcal{C}}$ denotes the mixing time of the restricted chain $\ol{X}$.
Let us subdivide the interval $[0,t]$ into $I_1,I_2,...$ by 
\begin{align*}
    I_{2k+1} & = [k (R_1+R_2), (k+1)R_1 + kR_2]\,, \qquad k\ge 0\,,\\ 
    I_{2k} & = [kR_1 + (k-1)R_2, k(R_1+ R_2)]\,, \qquad k \ge 1\,.
\end{align*}
Note that the required number of intervals is 
\begin{align}\label{eq:j_t-bound}
j_t:= { \frac{1}{2}} \lceil \tfrac{t}{R_1+R_2} \rceil  = \frac{1}{2} e^{\frac{{a}}{4} N^{\gamma - \xi}}(1+o(1))\, .
\end{align}

Explicitly the $o(1)$ term is at most  $e^{-\frac{{ a}}{2}N^{\gamma-\xi}}$ as: 
\[
\frac{t}{R_1+R_2} = e^{\frac{{ a}}{4} N^{\gamma-\xi} } \left( 1 -  \frac{R_1/R_2}{1+R_1/R_2} \right) \, ,
\]

and by definition of $R_1,R_2$ and an application of Theorem $\ref{thm:mixing-below-frozen-bond}$ and Lemma \ref{lem:gaps} one has:

\begin{align} \label{eq:o(1)-bound-Prop-int-average}
\frac{R_1}{R_2} =  \frac{N^{2/\alpha} e^{\frac{{a}}{4}N^{\gamma-\xi}} \tmix}{t} \leq  \frac{N^{2/\alpha} e^{\frac{1}{4} N^{\gamma-\xi}} \exp( 2\beta\abs{J_{(L+1)}} +DN^{\frac{1}{2}+\frac{1}{2\alpha}}  ) }{\exp(2\beta \abs{J_{(L)}} - \frac{{ a}}{100} N^{\gamma-\xi})} \leq  e^{-\frac{{ a}}{2} N^{\gamma-\xi} } \, .
\end{align}

To get the concentration estimate, we wish to make the following approximation:
\begin{align*}
\frac{1}{t} \int_{0}^{t}  \lambda(\ol{X}(s)) ds  \approx \frac{1}{j_t} \sum_{i=1}^{j_t} \Lambda_i \, , 
\end{align*}
for an i.i.d.\ sequence $\Lambda_i$ with mean $\mathbb E[\Lambda_i] = \mathbb E_{\sigma \sim \pi^{\mathcal C}}[\lambda(\sigma)]$.

By Theorem~\ref{thm:mixing-below-frozen-bond} we have that for any starting configuration $\ol{X}(0)$, by time $R_1$ the total-variation distance of $\ol{X}(t)$ to $\pi^{\mathcal C}$ is bounded by:
\[
\sup_{t \in I_2}\TV{\mathbb P(\ol{X}(t)\in \cdot)}{\pi^{\mathcal C}} \leq 2^{-N^{2/\alpha}} \, ,
\]
For general $k\ge 1$, by conditioning on the value of $\ol{X}(t)$ at time {$kR_1 + (k-1)R_2$}, and treating { $\ol{X}(kR_1+(k-1)R_2)$} as a fresh start, the same argument implies that 
\begin{align}\label{eq:rates-approximation-tv-bound}
\sup_{(t_i)_{i\le k} \in \otimes_{i\le k} I_{2i} } \TV{ \p(\ol{X}(t_1)\in \cdot,...,\ol{X}(t_k)\in \cdot) }{(\pi^{\mathcal C})^{\otimes k}} \leq k 2^{-N^{2/\alpha}} \, ,
\end{align}
Using the Markov property, this means that up to error of $j_t 2^{-N^{2/\alpha}}$,  on each of the intervals $I_{2i}$ we may replace $(\ol{X}(t))_{t\in I_{2i}}$ with the stationary chain $(\widetilde X^{(i)}(t))_{t\in [0,R_2]}$ where $\widetilde X^{(i)}$ are independent copies of $\ol{X}$, each independently initialized from $\widetilde X^{(i)}(0)\sim \pi^{\mathcal C}$. Define for each one, 
{
\begin{align*}
    \Lambda_i := \frac{1}{R_2} \int_{0}^{R_2} \lambda(\widetilde X^{(i)}(s)) ds\,,
\end{align*}
}
and note by stationarity that $\mathbb E[\Lambda_i] = \mathbb E_{\sigma \sim \pi^{\mathcal C}}[\lambda(\sigma)]$.
By the above coupling, we then have over intervals $I_{2i}$ that: 
\begin{align*}
    \Big| \frac{1}{t} \int_{I_{2i}} \lambda (\ol{X}(s))ds - \frac{1}{j_t} \Lambda_i \Big| &\leq \Big| \frac{1}{j_t} \frac{1}{R_2} \int_{I_{2i}} \lambda(\ol{X}(s))ds - \frac{1}{j_t} \Lambda_i \Big| + \frac{e^{-\frac{{a}}{2} N^{\gamma-\xi} }}{ { R_2} } \int_{I_{2i}} \lambda(\ol{X}(s))ds \\ 
    &\leq \frac{1}{j_t} 2^{-N^{2/\alpha}} + 2e^{-\frac{{a}}{2} N^{\gamma-\xi} } \exp(-2\beta \abs{J_{(L)}} + CN^{\frac{1}{2\alpha} }) \, ,
\end{align*}
where the last line follows from the total-variation bound~\eqref{eq:rates-approximation-tv-bound}, and the bounds
\begin{align} \label{eq:bounds-on-lambda}
\exp(-2\beta \abs{J_{(L)}} -CN^{\frac{1}{2\alpha} }) &\leq \lambda(\sigma) \leq \exp(-2\beta \abs{J_{(L)}} + CN^{\frac{1}{2\alpha} }) \qquad \text{for all $\sigma \in \mathcal C$, and} 
\\
{
 \Big| \frac{1}{t} - \frac{1}{j_t R_2} \Big| } &{  = \frac{ \abs{j_t R_2 -t}}{tj_tR_2} \leq \frac{1}{t j_t} = \frac{1+o(1)}{R_2 e^{\frac{a}{2} N^{\gamma-\xi} } } } \nonumber \, ,  
\end{align}
{ 
the first of which can be read off from Lemma ~\ref{lem:energy-drop-off-J} up to the change of constant $C$, and the second of which follows from~\eqref{eq:j_t-bound}.}

Combining the above, and using the bound on $j_t$ of~\eqref{eq:j_t-bound}, by a triangle inequality, 

\begin{align}
\Big| \frac{1}{t} \int_{0}^{t} & \lambda(\ol{X}(s))ds  - \frac{1}{j_t}  \sum_{i=1}^{j_t} \Lambda_{i} \Big|  \nonumber \\  
&\leq 2^{-N^{2/\alpha}} + 2e^{-\frac{{ a}}{4} N^{\gamma-\xi} } \exp(-2\beta \abs{J_{(L)}} 
+ CN^{\frac{1}{2\alpha} }) + \frac{1}{t} \sum_{k= 0}^{j_t} \int_{I_{2k+1}} \lambda(\ol{X}(s)) ds \label{eq:integral-concentration-about-iid-sum}
\end{align}

For the odd intervals, we apply Theorem \ref{thm:mixing-below-frozen-bond} along with \eqref{eq:j_t-bound}--\eqref{eq:o(1)-bound-Prop-int-average} to get that, 
\[
\frac{j_t R_1}{t} = \frac{j_t N^{2/\alpha} \tmix}{t} \leq e^{-\frac{{ a} }{6} N^{\gamma-\xi}}  \,.
\]

Combining with~\eqref{eq:bounds-on-lambda}, we have 
\begin{align*}
    \Big| \frac{1}{t} \int_{0}^{t} \lambda(\ol{X}(s))ds & - \frac{1}{j_t}  \sum_{i=1}^{j_t} \Lambda_{i} \Big| \le e^{-\frac{{a}}{8} N^{\gamma-\xi}} { \max_{\sigma \in \mathcal{C}} [\lambda(\sigma) ]}\,.
\end{align*}

 We have reduced things to the concentration of the i.i.d.\ average $ j_t^{-1} \sum_{i=1}^{j_t} \Lambda_i$. By Hoeffding's inequality, the bound $0 \leq \Lambda_i \leq 2\exp(- 2\beta \abs{J_{(L)}} + CN^{1/2\alpha})$ per~\eqref{eq:bounds-on-lambda}, and $1/2\alpha + 1/2 < \gamma-\xi$,
\[
\p \bigg( \Big|\frac{1}{j_t}  \sum_{i=1}^{j_t} \Lambda_{i} - \E_{\sigma \sim \pi^{\mathcal C} } [\lambda(\sigma)] \Big| \geq r  \bigg) \leq 2 \exp\left( -2 r^{2} e^{ 4\beta \abs{J_{(L)}} + \frac{{a}}{8} N^{\gamma-\xi} }  \right) \, .
\]

Combined with the bound of~\eqref{eq:integral-concentration-about-iid-sum}, we obtain for each $t\ge T_0$, 

\begin{align} \label{eq:single-time-ratio-bound} 
\p \bigg( \Big|\frac{ \frac{1}{t} \int_{0}^{t} \lambda(\ol{X}(s)) ds}{\E_{\sigma \sim \pi^{\mathcal C}}[\lambda(\sigma)] } -1 \Big| \geq e^{-\frac{{a}}{32} N^{\gamma-\xi} }  \bigg) \leq 2 \exp(-2 e^{ \frac{{ a} }{32} N^{\gamma-\xi} } ) \, ,
\end{align}

We now upgrade the estimate above to a bound on the supremum over $t\in [T_0,T_1]$. First, if we write $s=t+dt$, then one has deterministically,
\begin{align*}
    \Big|\frac{ \frac{1}{s} \int_{0}^{s} \lambda(\ol{X}(u)) du}{\E_{\sigma \sim \pi^{\mathcal C}}[\lambda(\sigma)] } -1 \Big| &\leq  \Big|\frac{ \frac{1}{t} \int_{0}^{t} \lambda(\ol{X}(s)) ds}{\E_{\sigma \sim \pi^{\mathcal C}}[\lambda(\sigma)] } -1 \Big| + 2 \frac{dt}{t+dt} \frac{\max_{\sigma \in \mathcal{C} } \lambda( \sigma)}{\E_{\sigma \sim \pi^{\mathcal{C}}} [\lambda(\sigma)] } \, .
\end{align*}
By applying Lemma~\ref{lem:energy-drop-off-J} to bound both numerator and denominator by $\exp( - 2\beta |J_{(L)}| \pm C N^{1/2\alpha})$ one has, 
\[
\frac{\max_{\sigma \in \mathcal{C} } \lambda( \sigma)}{\E_{\sigma \sim \pi^{\mathcal{C}}} [\lambda(\sigma)] } \leq \exp(2CN^{\frac{1}{2\alpha} }) \, ,
\]
and so taking $dt = 1 $, we see that on the complement of the event in~\eqref{eq:single-time-ratio-bound} for some $t$, then

\[
\sup_{s \in [t,t+1] }  \Big|\frac{ \frac{1}{s} \int_{0}^{s} \lambda(\ol{X}(u)) du}{\E_{\sigma \sim \pi^{\mathcal C}}[\lambda(\sigma)] } -1 \Big| \leq e^{-\frac{{ a}}{32}N^{\gamma - \xi}} + e^{ - 2\beta |J_{(L)}| + \frac{{ a}}{100}N^{\gamma - \xi} + 3CN^{1/2\alpha}} \le  2 e^{- \frac{{ a}}{32} N^{\gamma-\xi} } \, ,
\]
{ where} we have used $N^{\gamma-\xi}$ and $N^{\frac{1}{2\alpha}}$ are both $o(\abs{J_{(L)} })$.

Thus for any integer $m >0$ we have:
\begin{align*}
\p &\Bigg( 
  \sup_{s \in [T_0,T_0+m dt] } \Bigg|\frac{ \frac{1}{s} \int_{0}^{s} \lambda(\ol{X}(u)) du}{\E_{\sigma \sim \pi^{\mathcal C}}[\lambda(\sigma)] } -1 \Bigg| \geq {8 e^{- \frac{a}{32} N^{\gamma-\xi}} }
  \Bigg) \\
   & \leq
  \sum_{i=0}^{m} \p \left( \Bigg|\frac{ \frac{1}{t_i} \int_{0}^{t_i} \lambda(\ol{X}(u)) du}{\E_{\sigma \sim \pi^{\mathcal C}}[\lambda(\sigma)] } -1 \Bigg| \geq { 4 e^{- \frac{a}{32} N^{\gamma-\xi}} } \right)
\leq 2m \exp(-2e^{\frac{{ a} }{32}N^{\gamma-\xi} }) \, ,
\end{align*}

where $t_i= T_0+ i$. To complete the proof, take $m= \exp(2\beta \abs{J_{(L)}} + N^{2/\alpha})$ to get: 
\[
\p\left( \sup_{t \in [T_0,T_1]} \Bigg|\frac{ \frac{1}{t} \int_{0}^{t} \lambda(\ol{X}(u)) du}{\E_{\sigma \sim \pi^{\mathcal C}}[\lambda(\sigma)] } -1 \Bigg| \geq { 8 e^{- \frac{a}{32} N^{\gamma-\xi} } } \right) \leq 2 \exp(-e^{\frac{{ a}}{32} N^{\gamma-\xi} } ) \, ,
\]

proving the main claim.
\end{proof}

We now state the analogue of Lemma \ref{lem:Comparison-of-rates} for the process $Y$.

\begin{lemma} \label{lem:comparison-of-rates-Y} Let $1 \leq L \leq K$ be a fixed index, fix an assignment $y_1,...,y_L$, and 
let $\mathcal{W}$ denote the set of all configurations in $\{ \pm 1 \}^{N}$ taking those values on $\{Y_i \}_{i \leq L }$.

Let $\ol{Y}(t)$ denote the process $Y$ restricted to $\mathcal{W}$, i.e., $\ol{Y}$ is initialized in $\mathcal{W}$, and rejects all updates to vertices $1,...,L$. Define $\ol{Z}_L$ conditional on $\mathcal{W}$ by: 
\[
\ol{Z}_{L} = \pi_{\beta,\mathbf{J}} \bigg[ Z_{v_L}(\sigma) + Z_{w_L}(\sigma) \Big| E_{a,\gamma} \subset \mathsf{Sat}(\sigma)\,,\, \sigma_{v_i} = y_i \ 1 \leq i \leq L  \bigg]\,,
\]
where $Z_v(\sigma)$ is defined in \eqref{eq:def-of-Z-and-m}. Then uniformly in the choice of $\mathcal{W}$, and $s$ satisfying $st_{a,\gamma} > T_0:= \exp(2\beta \abs{J_{(L)}}  - \frac{{ a}}{100} N^{\gamma-\xi} )$ one has that,
 \[
\Big| 1- \frac{\E [Z_L(\ol{Y}(s))]}{\ol{Z}_L } \bigg| \leq e^{-\frac{{ a} }{32} N^{\gamma-\xi} } \, .
 \]
 
 Furthermore, for any initialization in $\mathcal{W}$  
\[
\p \bigg( \sup_{s: st_{a,\gamma} \in [T_0,T_1] } \Big|\frac{ \frac{1}{s} \int_0^s Z_L(\ol{Y}(u)) du}{\ol{Z}_L} -1 \Big| \geq { 8} e^{-\frac{{ a}}{32} N^{\gamma-\xi}} \bigg) \leq  2\exp( - e^{\frac{{ a }}{32} N^{\gamma-\xi} } )\, , 
\]

where $T_1 = \exp(2\beta \abs{J_{(L)} } + N^{\frac{2}{\alpha} } )$.
\end{lemma}

We omit the proof of Lemma \ref{lem:comparison-of-rates-Y}, because it is identical to that of Lemma \ref{lem:Comparison-of-rates}, using Proposition~\ref{Prop:Mixing-in-well-Y} in place of Theorem~\ref{thm:mixing-below-frozen-bond}, and  Lemma \ref{lem:Reversibility-of-Y} for identification of $\pi^{\mathcal W}$ as stationary for $\ol{Y}$.

\subsection{Coupling the skeleton process}
In this subsection, we will establish our main theorem, Theorem~\ref{thm:main1}. 
We begin by giving simpler characterizations of the laws of the skeleton of the Markov chain $\mathsf{S}(t)$ and of the claimed limiting process $Y$ at large times.

{
In what follows, let $\delta_N = e^{ -2D N^{\frac{1}{2\alpha}+\frac{1}{2}}}$, where $D(\alpha,\beta)$ is as in Theorem \ref{thm:mixing-below-frozen-bond}.  
For any $s=s_N$, define a function $\label{eq:definition-of-L} L: \real_+ \to \{1,...,K\}$ by
\begin{align*}
L(s_N) :=  
\begin{cases}
    K &\text{if} \ s t_{a,\gamma} < \delta_N^{-1} \exp(2\beta \abs{J_{(K)}}  ),  \\ 
    L & \text{if} \ st_{a,\gamma} \in [\delta_N^{-1} \exp(2\beta \abs{J_{(L+1)}} ), \delta_N^{-1} \exp(2\beta \abs{J_{(L)} })] \ \text{and} \ 1 \leq L \leq K-1, \\ 
    0 &\text{if} \ st_{a,\gamma} > \delta_N^{-1} \exp(2\beta \abs{J_{(1)} })\,.
\end{cases} 
\end{align*}
Lemma~\ref{lem:spacing-of-J} guarantees that the intervals above are disjoint and therefore $L$ is well defined. 
Heuristically, $L$ is such that all edges $e_i$ for $1 \leq i \le L-1$ will stay satisfied for timescales $s_N t_{a,\gamma}$ with probability at least $1-e^{-N^{\theta}}$ for $\theta>0$. 

\begin{remark} \label{rem:Frozen-above-L(s)-below-s}
For any $s= s_N$, let $L=L(s)>1$. We have that 
\[
st_{a,\gamma} \leq \exp( 2\beta \abs{J_{(L)}} + 2DN^{\frac{1}{2}+\frac{1}{2\alpha} } ) \, ,
\]
and hence Lemma \ref{lem:spacing-of-J} implies
\begin{align} \label{eq:s-L(s)-comparison} 
st_{a,\gamma} \leq \exp( 2\beta \abs{J_{(L-1)}} - \frac{\beta}{2} N^{\gamma-\xi} ) \, .
\end{align}
Thus, by Lemma~\ref{lem:spin-structure}, we have with probability $1-e^{-cN^{\gamma-\xi}}$ edges $e_1,...,e_{L-1}$ will become satisfied at their first clock rings and remain satisfied through time $s t_{a,\gamma}$.
\end{remark}
}

We now define a 2-state Markov chain $M$, which will approximately track the switches between the two possible satisfying assignments on the $L$'th bond. Fix a well $\mathcal{C}$ defined by a satisfying assignment to the edges $e_1,...e_{L-1}$, and let $\mathcal{C}_+,\mathcal{C}_-$ denote the subsets of $\mathcal{C}$ such that $e_L \in \sat(\sigma)$ for every $\sigma \in \mathcal{C}_{\pm}$, obtained by fixing the value of the spin at vertex $v_L$ to be $\pm 1$. Define the quantities, 
\begin{align} \label{eq:rates-of-M}
\ol{\lambda}_+ &= \E_{\sigma \sim \pi^{\mathcal{C}_+}} [\lambda_L(\sigma)] \, ,
 \\
\ol{\lambda}_- & = \E_{\sigma \sim \pi^{\mathcal{C}_-} } [\lambda_L(\sigma)] \, ,
\end{align}
where $\lambda_L$ was defined in~\eqref{eq:rates-of-XL}. 
The chain $M(t)$ takes values in $\{ \pm 1 \}$; if it is in state $+1$ (respectively $-1$) it waits an exponential clock with rate $\ol{\lambda}_+$ (respectively $\ol{\lambda}_-$) and when the clock rings it stays in its current state with probability $1/2$ and moves to the other state with probability $1/2$. 
 
We now prove a lemma which describes the distribution of $\SK(st_{a,\gamma})$ by coupling its $L(s)$'th coordinate to $M(st_{a,\gamma})$.

\begin{proposition} \label{prop:Behavior-of-S} 
    Let $s=s(N)>0$ be uniformly bounded away from zero, possibly diverging, and write $L=L(s)$.  
    Fix a well $\mathcal{C}$ defined by satisfying assignments to $e_1,...,e_{L-1}$. Let $\ol{X}$ denote Glauber dynamics with initialization in $\mathcal{C}_+ \cup \mathcal{C}_-$, restricted to the well $\mathcal{C}$ , i.e., $\ol{X}$ rejects updates to vertices in $e_1,...,e_{L-1}$. Write $\ol{\SK}$ for the projection of $\ol{X}$ onto coordinates $v_1,...,v_K$. Then at time $st_{a,\gamma}$ the law of $\ol{\SK}(st_{a,\gamma})$ is within $e^{-N^{\theta}}$ total variation distance to the law of $\widehat{\SK}$, where $\widehat{\SK}$ is defined according to cases for $s$: 
    
\begin{enumerate}
    \item \textbf{Case 1. $s t_{a,\gamma} \in [\delta_N \exp(2\beta |J_{(L)}| ), \delta_N^{-1}\exp(2\beta |J_{(L)}| ) ]$:} $\widehat \SK $ is defined as
    \begin{itemize}
    \item Coordinates $1 \leq i \le L-1$ have $\widehat \SK_i(s t_{a,\gamma}) = \ol{\SK}_i(0)$. 
    \item $\widehat \SK_{L(s)}(st_{a,\gamma})$ is drawn as follows: initialize $M(0)=\ol{\SK}_{L}(0)$, and set $\widehat{\SK}_{L}(st_{a,\gamma})= M(st_{a,\gamma})$. 
    \item Draw the remaining coordinates $\widehat \SK_i(st_{a,\gamma})$ for $L+1\le i \leq K$ according to $\pi_{\beta,\J}(\SK(\sigma) \in \cdot)$ conditioned on having $\sigma \in \mathcal C$ and $\sigma_{v_i} = \widehat \SK_i(st_{a,\gamma})$ for $1 \leq i \leq L$.
    \end{itemize}
    \item \textbf{Case 2. $s t_{a,\gamma} \notin [\delta_N \exp(2\beta |J_{(L)}| ), \delta_N^{-1}\exp(2\beta |J_{(L)}| ) ]$:} $\widehat \SK$ is defined as 
    \begin{itemize}
        \item Coordinates {$0 \leq i \leq L$} have $\widehat \SK_i(st_{a,\gamma})=\ol{\SK}_i(0)$.
        \item Coordinates { $L+1\leq i\leq K$} are drawn according to $\pi_{\beta,\J}(\SK(\sigma) \in \cdot)$ conditioned on having $\sigma \in \mathcal C$ and $\sigma_{v_i} = \widehat \SK_i(st_{a,\gamma})$ for {$0 \leq i \leq L$}.
    \end{itemize}
\end{enumerate}
    
\end{proposition}

{ 
\begin{remark}
Case 2 of Proposition \ref{prop:Behavior-of-S} in the case where $L(s)=0$ should be understood as $\mathcal C = \Sigma_N$, and the coordinates $1 \leq i \leq K$ are drawn according to the marginal of the full Gibbs' measure $\pi_{\beta,\mathbf{J}}$. 
\end{remark} }

\begin{proof}
    Fix $s$ and $L=L(s)$. We work case by case, showing how to couple $\ol{\SK}(st_{a,\gamma})$ to $\widehat \SK(st_{a,\gamma})$. 

    \textbf{Case 1:} In this case we will appeal to Lemma~\ref{lem:Comparison-of-rates} and show that $\ol{\SK}_L(st_{a,\gamma})$ is coupled to $M(st_{a,\gamma})$ with high probability. 
Let $\tau_{\text{flip}}^L$ denote the first flip time of one of $v_L,w_L$.

    We  compute an upper bound on the total-variation distance of the law of $\tau_{\text{flip}}^L$, and the exponential clock attached to $M(t)$ from the initialization of $\ol{\SK}_L(0)$. First note that one has $\p( \tau_{\text{flip}}^L > t) = \E [e^{-\int_0^t \lambda_L (\ol{X}(u)) du}]$ so that 
\[
\frac{d}{dt} \p(\tau_{\text{flip}}^L \leq  t  ) = \E \big[ \lambda_L (\ol{X}(t)) e^{-\int_{0}^{t} \lambda_L(\ol{X}(u)) du)}\big] \, ,
\]
where $\lambda_L(\sigma) = \lambda_{v_L}(\sigma) + \lambda_{w_L }(\sigma)$ as in~\eqref{eq:rates-of-XL}, and $\E$ is expectation over the law of the Markov chain with the given initialization. For the remainder of the proof, write $\lambda=\lambda_L$. Since the initializations are the same, let us consider the case that that ${ \ol{\SK}_L(0)= +1}$ and $M(0)=+1$ (the other case being symmetrical). We bound the total-variation distance between $\tau_{\text{flip}}^L$ and the exponential clock of rate $\ol{\lambda}_+ =\E_{\sigma \sim \pi^{\mathcal C_+}} [\lambda(\sigma)] $ associated to $M$. Using that $\TV{\mu}{\nu} = \frac{1}{2} \|\mu - \nu\|_1$, we aim to show 
\begin{align}\label{eq:nts-clocks-close}
\int_{0}^{\infty} \abs{\ol{\lambda}_+ e^{-\ol{\lambda}_+ t} - \E[\lambda(\ol{X}(t)) e^{-\int_{0}^{t} \lambda(\ol{X}(u)) du}] } dt \le Ce^{ - c N^{\gamma - \xi}}\,.
\end{align}
We begin by splitting the integral over three intervals $I_1,I_2,I_3$ given by: 
\[
I_1=[0,\epsilon \ol{\lambda}^{-1}_+ ] \qquad I_2:= [\epsilon \ol{\lambda}_+^{-1}, \epsilon^{-1} \ol{\lambda}_+^{-1})  \qquad I_3:=  [ \epsilon^{-1} \ol{\lambda}_+^{-1},\infty) \, ,
\]
where $\epsilon = \exp(- N^{\gamma-\xi}/50 )$. By the triangle inequality we then have: 
\begin{align}
    \int_{0}^{\infty} \abs{\ol{\lambda}_+ &  e^{-\ol{\lambda}_+ t} - \E[\lambda(\ol{X}(t)) e^{-\int_{0}^{t} \lambda(\ol{X}(u)) du}] } dt \nonumber \\
    & \leq \int_{I_1 \cup I_3} \abs{\ol{\lambda}_+ e^{-\ol{\lambda}_+ t} - \E[\lambda(\ol{X}(t)) e^{-\int_{0}^{t} \lambda(\ol{X}(u)) du}] } dt \label{eq:I_1-cup-I_3}\\
    &\qquad + \int_{I_2} | \ol{\lambda}_+ - \E[\lambda(\ol{X}(t))] | { e^{-\ol{\lambda}_+ t} }dt \label{eq:I_2-a} \\
    &\qquad + \int_{I_2} \E \Big[ \big| 1- e^{-\ol{\lambda}_+ t + \int_{0}^{t} \lambda(\ol{X}(u)) du} \big| \lambda(\ol{X}(t)) e^{-\int_{0}^{t} \lambda(\ol{X}(u)) du }\Big] dt\,.\label{eq:I_2-b}
\end{align}

To bound the first integral, break both terms via the triangle inequality to get,
\[
\int_{I_1 \cup I_3} \abs{\ol{\lambda}_+ e^{-\ol{\lambda}_+ t} - \E[\lambda(\ol{X}(t)) e^{-\int_{0}^{t} \lambda(\ol{X}(u)) du}] } dt \leq \p (\text{Exp}(\ol{\lambda}_+) \in I_1 \cup I_3) + \p( \tau_{\text{flip}}^L \in I_1 \cup I_3 ) \, ,
\]
where here and throughout the paper, we use $\text{Exp}(r)$ to denote an independent (of everything) exponential clock of rate $r$. 
The first term is {$1+(e^{ - \epsilon^{-1}} - e^{-\epsilon})$}, which is exponentially small in $N^{\gamma-\xi}$. For the second term, by~\eqref{eq:bounds-on-lambda} $\tau_{\text{flip}}^{L}$ stochastically dominates an exponential clock of rate $\exp(-2\beta \abs{J_{(L)}} +CN^{1/(2\alpha)} )$, and is stochastically dominated by an exponential clock of rate $ \exp(-2\beta \abs{J_{(L)}} - CN^{1/(2\alpha) })$. Hence, for a constant $c>0$, 
\[
\p (\tau_{\text{flip}}^{L} \in I_1 \cup I_3 ) \leq { (1+e^{- \epsilon^{-1} e^{-3CN^{1/2\alpha}}} - e^{ - \epsilon e^{3CN^{1/2\alpha}}}) } \le e^{-cN^{\gamma-\xi} } \, .
\]
Thus the term~\eqref{eq:I_1-cup-I_3} is at most  $Ce^{-cN^{\gamma-\xi} }$ 
for some $C,c>0$.

To bound~\eqref{eq:I_2-a}, factor $\ol{\lambda}_+$ and apply \eqref{eq:ratio-of-rates-bound-X} to get:
\[
\int_{I_2} |\ol{\lambda}_+ - \E [\lambda(\ol{X}(t))] | e^{-\ol{\lambda}_+ t} dt \leq e^{-\frac{1}{32} N^{\gamma-\xi} } \int_{I_2} \ol{\lambda}_+ e^{-\ol{\lambda}_+ t} dt \leq e^{-\frac{1}{32} N^{\gamma-\xi} } \, .
\]
Finally for~\eqref{eq:I_2-b}, define the event $A$ by,
\[
A = \left\{ \sup_{t \in I_2}  \Bigg| \frac{ \frac{1}{t}\int_{0}^{t} \lambda(\ol{X}(u)) du}{\ol{\lambda}_+} -1 \Bigg| \leq 2e^{-\frac{1}{32} N^{\gamma-\xi} }   \right\}\,.
\]
{Since $I_2 \subset [T_0,T_1]$ from Lemma \ref{lem:Comparison-of-rates}}, one has that,
\[
\p(A^c) \leq 2 \exp\Big(- e^{\frac{N^{\gamma-\xi}}{32} } \Big) \, .
\]
On the event $A$,  one has,
\[
\sup_{t \in I_2} \Bigg| \left(\ol{\lambda}_+ - \frac{1}{t} \int_{0}^{t} \lambda(\ol{X}(u))du \right) t \Bigg| \leq \sup_{t \in I_2} 2e^{-\frac{1}{32} N^{\gamma-\xi}} \ol{\lambda}_+ t  \leq 2\epsilon^{-1} e^{-\frac{1}{32} N^{\gamma-\xi}}  \, ,
\]
and so using the inequality $e^x-1 \leq \max(2x,\frac{1}{2} x)$ for $x \in [-1,1]$, on the event $A$ one has
\[
\sup_{t \in I_2} \abs{e^{-(\ol{\lambda}_+ -\frac{1}{t} \int_{0}^{t} \lambda(\ol{X}(u))du ) t  } -1} 
\leq  4 \epsilon^{-1} e^{-\frac{1}{32} N^{\gamma-\xi}} \, .
\]
When not on the event $A$, we can use the general bounds of~\eqref{eq:bounds-on-lambda} and $t\le \epsilon^{-1}\ol{\lambda}_+^{-1}$ to see that 
\[
\sup_{t \in I_2} \Big| e^{-(\ol{\lambda}_+ -\frac{1}{t} \int_{0}^{t} \lambda(\ol{X}(u))du ) t}   -1 \Big| \le e^{\epsilon^{-1}e^{2CN^{1/2\alpha}}}\leq \exp \left( e^{\frac{1}{50} N^{\gamma-\xi} +2CN^{\frac{1}{2\alpha}} } \right) \, .
\]
Combining the inequalities above (and splitting the expectation over $A$ and $A^c$), we then have:
\begin{align*}
\int_{I_2} \E\Big[ \big| 1- e^{-\ol{\lambda}_+ t + \int_{0}^{t}  \lambda(\ol{X}(u)) du} \big|  & \lambda(\ol{X}(t)) e^{-\int_{0}^{t} \lambda(\ol{X}(u)) du }\Big] dt  \\
& \leq 4\epsilon^{-1} e^{-\frac{1}{32} N^{\gamma-\xi} }  +
2\exp \Big( e^{\frac{1}{50} N^{\gamma-\xi} +2 CN^{\frac{1}{2\alpha}} } - e^{\frac{1}{32} N^{\gamma-\xi} } \Big) 
\\
&\leq e^{-cN^{\gamma-\xi} } \, .
\end{align*}
Hence combining the estimates from above we get the desired claim~\eqref{eq:nts-clocks-close}.

While $M$ is making a uniform choice of $\pm 1$ at a time that is $\text{Exp}(\ol{\lambda}_{\pm})$, the process $\ol{\SK}_L(t)$ starts two independent $\text{Exp}(1)$ clocks at $\tau_{\text{flip}}^L$, one for $v_L$ and one for $w_L$ and up to an error of $\exp( -N^{\gamma})$ per~\eqref{eq:prob-unsatisfy-at-update-v}, whichever rings first (each having probability $1/2$ of doing so) dictates the state among $\pm 1$ that $\ol{\SK}_L$ next takes. Therefore, both the first resampling time of $\ol{\SK}_L(t)$ and its choice among $\pm 1$ can be coupled perfectly to the first resampling time and choice among $\pm 1$ for $M$, except with probability 
\begin{align*}
 \TV{\mathbb P(\tau_{\text{flip}}^L + \text{Exp}(2) \in \cdot )}{\mathbb P(\text{Exp}(\ol{\lambda}_+)\in \cdot)} + e^{- N^{\gamma}}
\end{align*}
By~\eqref{eq:nts-clocks-close} and a triangle inequality, it suffices to bound 
\begin{align*}
    \TV{\p(\text{Exp}(\ol{\lambda}_+) \in \cdot)}{\p(\text{Exp}(\ol{\lambda}_+) +\text{Exp}(2) \in \cdot )} \, ,
\end{align*}
which we can do explicitly. Indeed, 
the density of the sum of a clock of rate $\ol{\lambda}_+$ and a independent clock of rate $2$ is given by:
\[
f^{+}(z) = \frac{2\ol{\lambda}_+}{2-\ol{\lambda}_+} (e^{-\ol{\lambda}_+ z} - e^{-2z} ) \mathbf{1}_{z \geq 0} \, ,
\]
and consequently the total-variation distance between a clock of rate $\text{Exp}( \ol{\lambda}_{\pm})$ and the sum of two independent clocks, $\text{Exp}(\ol{\lambda}_+) + \text{Exp}(2)$ is given by, 
\begin{align*}
    d_{TV}\big(\p(\text{Exp}(\ol{\lambda}_+) \in \cdot) & ,  \p(\text{Exp}(\ol{\lambda}_+) +\text{Exp}(2) \in \cdot )\big) \\
     & = \int_{0}^{\infty} \Big|\ol{\lambda}_+e^{-\ol{\lambda}_+ z} - \frac{2\ol{\lambda}_+}{2-\ol{\lambda}_+} (e^{-\ol{\lambda}_+ z} - e^{-2z} )\Big| dz \\
    &\leq \int_{0}^{\infty} \ol{\lambda}_+ \Big|1 - \frac{2}{2-\ol{\lambda}_+} \Big| e^{-\ol{\lambda}_+ z} dz + \int_{0}^{\infty} \frac{\ol{\lambda}_+}{2-\ol{\lambda}_+} 2 e^{-2z} dz\,.
    \end{align*}
   In turn, by definition of ${ \ol{\lambda}_+}$ and Lemma~\ref{lem:energy-drop-off-J}, this is at most 
    \begin{align*}
        2\Big| \frac{\ol{\lambda}_+}{2-\ol{\lambda}_+} \Big| \leq 2\ol{\lambda}_+ \leq 4\exp(-2\beta \abs{J_{(L)}} + CN^{1/2\alpha} ) \leq e^{-\frac{N^{\gamma}}{2}}\,.
    \end{align*}

 Finally, we iteratively apply the above coupling {or its analogue when $\ol{\SK}_L(t) = -1$}, between each update time of $\SK_L$ and $M$, and take a union bound over all updates. 
In order for the union bound not to have too many terms, notice that over an interval of the form 
$$\left[\exp\left(2\beta \abs{J_{(L)}} - 2DN^{\frac{1}{2} + \frac{1}{2\alpha}}\right),\exp\left(2\beta \abs{J_{(L)}} + 2DN^{\frac{1}{2} + \frac{1}{2\alpha}}\right) \right] \, ,$$ 
we may assume there are at most $\exp(100DN^{1/(2\alpha) + 1/2})$  successful transitions from satisfaction to dissatisfaction for $e_L$. 
Indeed, if $\tin_i$ denotes the $i^{\text{th}}$ time $e_L$ re-enters $\sat$, {then by~\eqref{eq:bounds-on-lambda}, $\tin_{i+1} -\tin_i$ stochastically dominates an exponential clock of rate $\exp(- 2\beta |J_{(L)}| + CN^{1/(2\alpha)})$ for all $i\ge 1$. Thus, if we let $E_i$ be i.i.d.\ $\text{Exp}(\exp( - 2\beta |J_{(L)}| + CN^{1/(2\alpha)}))$, the tails on the number of successful transitions between satisfaction and dissatisfaction are upper bounded by the simple large deviation estimate} 
\begin{align}\label{eq:number-of-clock-rings}
\p \Big(  \sum_{i=1}^m E_i < \frac{m}{2} \exp(2\beta \abs{J_{(L)}} - CN^{\frac{1}{2\alpha}} )   \Big) \leq e^{-cm}   \, , 
\end{align}
for some $c>0$.
Hence, by taking $m= \exp(100DN^{1/2+1/(2\alpha)})$, we see with with probability at least $1-e^{-N^{\gamma-\xi}}$  there are at most $\exp({100DN^{\frac{1}{2}+\frac{1}{2\alpha} }})$ {changes to the vertices at $\{v_L,w_L\}$} by time $st_{a,\gamma}$. 

 On the event that there are less than $\exp({100DN^{\frac{1}{2\alpha}+\frac{1}{2}}})$ updates over the interval of interest, we may pick an optimal coupling of the reentry into wells via coordinate $L$ and the clocks of $M$ for the first $\exp({100DN^{\frac{1}{2\alpha}+\frac{1}{2}}})$ many transits. Then by a union bound we get for any fixed time $u \in [\delta_N \exp(2\beta \abs{J_{(L)}}), \delta_N^{-1} \exp(2\beta \abs{J_{(L)}} )]$ that:
\begin{align} \label{eq:fixed-time-M-SL-coupling-bound}
\p(  M(u) \neq \ol{\SK}_L(u) ) \leq  e^{ 100 DN^{\frac{1}{2} + \frac{1}{2\alpha}} - cN^{\gamma - \xi}} + e^{ - N^{\gamma-\xi} } \leq  e^{-c'N^{\gamma-\xi}} \, ,
\end{align}
for some $c'>0$.
To finish the proof it remains to argue that the most recent update on the $L$'th coordinate gives enough room until time $st_{a,\gamma}$ for all the coordinates with indices bigger than $L$ to mix conditional on the states of $(\mathsf{S}_i)_{i\le L}$. Let us set:
\[
R_L = \exp\left(2\beta \abs{J_{L}} -\frac{1}{2} N^{\gamma-\xi}\right) \, ,
\]
then conditional on $\ol{X}(st_{a,\gamma} - 2R_L)$ we have by Lemma \ref{lem:spin-structure} that {$e_L \in \sat(\ol{X}(t))$} for all $t \in [st_{a,\gamma}-R_L,st_{a,\gamma}]$ , and furthermore $(\ol{X}_{v_L}(t),\ol{X}_{w_L}(t)) = (\ol{X}_{v_L}(st_{a,\gamma}-2R_L),\ol{X}_{w_L}(st_{a,\gamma}-2R_L)$), with probability $1-e^{-\frac{1}{4}N^{\gamma-\xi}}$. Furthermore note that conditional on the value of $M$ at time $st_{a,\gamma}-2R_L$, $M$ has no updates in the window $[s t_{a,\gamma} - 2R_L,st_{a,\gamma}]$ with probability $1-e^{-N^{\gamma-\xi}}$. 
On the event of no successful update to $\ol{X}(t)$ over $[st_{a,\gamma} - R_L,st_{a,\gamma}]$ and $e_i \in \sat(X(st_{a,\gamma}-R_L))$, we have that $\ol{X}(t)$
is equal to the restricted Glauber dynamics $\ol{\mathcal{X}}$  initialized at the configuration that $\ol{X}$ takes at $st_{a,\gamma} - R_L$, and rejecting all updates to coordinates $\{v_i,w_i\}$ for $i \leq L$. As $R_L$ is exponentially larger than the mixing time of $\ol{\mathcal{X}}$, Theorem~\ref{thm:mixing-below-frozen-bond} applies,  and so we may combine all of the estimates above to obtain: 
\[
\p ( \ol{\SK}(st_{a,\gamma}) \neq \widehat{\SK}_L(st_{a,\gamma})  ) \leq e^{-cN^{\gamma-\xi} } \, ,
\]
for some absolute constants $c>0$.

\medskip
    \textbf{Case 2:} This case is significantly simpler as the timescale is not on the same order as the typical flip-rate of any one coordinate. Let us consider the following three possible sub-cases of $s, L(s)$: 
\begin{enumerate}[(2a)]
   \item \label{eq:small-s} $L(s)=K$, and $s$ satisfies $ s t_{a,\gamma} < \delta_N \exp(2\beta \abs{J_{(K)} })$, 

   \item \label{eq:between-s}  { $1 \leq L(s) \leq K$, and $s$ satisfies $st_{a,\gamma} \in (\delta_N^{-1} \exp(2\beta \abs{J_{(L+1)}}), \delta_N \exp(2\beta \abs{J_{(L)}})]$, }
  
    \item \label{eq:large-s} {$L(s)=0$, and $s$ satisfies $s t_{a,\gamma} > \delta_N^{-1} \exp(2\beta \abs{J_{(1)} } )  $}. 
\end{enumerate}

For \ref{eq:small-s} Lemma~\ref{lem:spin-structure} implies $\ol{\SK}(s_1t_{a,\gamma})=\ol{\SK}(0)$ with probability at least $1- \exp( -\frac{1}{4} N^{\frac{1}{2}+\frac{1}{2\alpha}} )$, and so the statement is immediate.

For \ref{eq:between-s}, we shall assume without loss of generality  that $\ol{X}(0) \in \mathcal{C}_+$.  Lemma~\ref{lem:spin-structure} shows that no coordinates corresponding to edges $e_{(1)},..., e_{(L)}$ flip in $\ol{X}$ with probability $1-e^{-\frac{1}{4}N^{\gamma-\xi}}$ before time $st_{a,\gamma}$. On this event, the process $\ol{X}(t)$ is coupled perfectly to the process which rejects all updates to $\{v_i,w_i\}_{i\le L}$. Theorem~\ref{thm:mixing-below-frozen-bond} in combination with \eqref{eq:tv-distance-submultiplicativity} then implies:  
\[
\TV{ \p( \ol{\SK}(st_{a,\gamma} )\in \cdot)}{\pi^{\mathcal{C}} } \leq 2^{-\exp(DN^{\frac{1}{2}+\frac{1}{2\alpha}})} + e^{ - \frac{1}{4} N^{\gamma - \xi}} \, .
\]

Lastly for \ref{eq:large-s} we appeal to the mixing time bounds from Theorem~\ref{thm:mixing-below-frozen-bond} with $L=0$. By~\eqref{eq:tv-distance-submultiplicativity} one has that:
\[
\TV{\p(\ol{X}(s t_{a,\gamma}) \in \cdot)}{\pi_{\beta,\J}} \leq 2^{-e^{D N^{\frac{1}{2}+\frac{1}{2\alpha}} }} \, ,
\]
and immediately this implies that:
\[
\TV{ \p(\ol{\SK} \in \cdot)}{ \pi_{\beta,\J}( (\sigma_{v_1},...,\sigma_{v_K}) \in \cdot )  }  \leq 2^{-e^{D N^{\frac{1}{2}+\frac{1}{2\alpha}} }} \, ,
\]
proving the result.
\end{proof}

The following lemma is exactly analogous  to the previous one, but for the process $Y$ instead of $\SK$. We begin by defining the auxiliary process $M'$ which will track the value of a specific coordinate. Suppose we fix a well $\mathcal{W}$ defined by an assignment of values for $Y_1=y_1,...,Y_{L-1}=y_{L-1}$, and let $\mathcal{W}_+,\mathcal{W}_-$ denote the subsets of $\mathcal{W}$ with $Y_L=\pm 1$. Write $\pi^{\mathcal{W}}$ for the measure $\pi_{\beta,\J}$ conditional on the event that $e_1,...,e_K \in \sat(\sigma)$, and $\sigma_{v_i} = y_i$ for $1 \leq i \leq L-1$. 
The process $M'$ takes values in $\{\pm 1\}$ with rates given by: 
\begin{align}  \label{eq:rates-of-Mprime}
\ol{Z}_+ &=  \E_{\sigma \sim \pi^{\mathcal{W}_+}} (Z_{v_L}(\sigma) + Z_{w_L}(\sigma)) \\
\ol{Z}_- &=  \E_{\sigma \sim \pi^{\mathcal{W}_-}} (Z_{v_L}(\sigma) + Z_{w_L}(\sigma)) \, ,
\end{align}
where $Z_{v_l}$ is as in ~\eqref{eq:def-of-Z-and-m}.
$M'$ makes updates as follows: when in state $\pm 1$, run an exponential clock of rate $\ol{Z}_{\pm}$, and when this clock rings, uniformly resample its state from $\pm 1$. 

\begin{proposition} \label{prop:behavior-of-Y} 
    Let $s=s(N)>0$ be a sequence of times (uniformly greater than zero, possibly tending to infinity), and write $L=L(s)$. Fix a well $\mathcal{W}$ defined by an assignment to coordinates 
    $Y_1,...,Y_{L-1}$. Let $\ol{Y}$ denote the process $Y$ restricted to $\mathcal{W}$, i.e., it is 
    initialized in $\mathcal{W}$ and rejects updates to $Y_1,...,Y_{L-1}$. Then at time $s$ the law of $\ol{Y}(s)$ is within $e^{-N^{\theta}}$ total variation distance to the law of 
    $\widehat{Y}$, where $\widehat{Y}$ is defined according to cases for $s$:
    
\begin{enumerate}[(1)]
    \item \label{eq:mainthm-case1} \textbf{Case 1. $s t_{a,\gamma} \in \left[\delta_N \exp\left(2\beta \abs{J_{(L)}} \right), \delta_N^{-1}\exp\left(2\beta \abs{J_{(L)}} \right) \right]$:} define $\widehat Y$ as 
\begin{itemize}
    \item  Coordinates $i \le L- 1$ have  $\widehat Y_i(s) = \ol{Y}_i(0)$.
    \item Set $M'(0)=\ol{Y}_L(0)$, run $M'$ for time $st_{a,\gamma}$, and set $\widehat{Y}_L(s) =M'(st_{a,\gamma})$.
    \item The remaining coordinates of $\widehat Y(s)$ are drawn according to the stationary distribution conditioned on the values of $\widehat{Y}_i(s)$ for $1 \leq i \leq L$. 
\end{itemize}
    \item \label{eq:mainthm-case2} \textbf{Case 2:  $s t_{a,\gamma} \not \in   \left[\delta_N \exp\left(2\beta \abs{J_{(L)}} \right), \delta_N^{-1}\exp\left(2\beta \abs{J_{(L)}} \right) \right]$}: define $\widehat Y$ as,
    \begin{itemize}
        \item  Coordinates { $0 \leq i \leq L$} have $\widehat Y_i(s) = \ol{Y}_i(0)$ 
        \item The remaining coordinates of $\widehat Y$ are drawn according to the stationary distribution conditional on the values of $\widehat Y_i(s)$ for { $i=0,...,L$. }
    \end{itemize}
\end{enumerate}
\end{proposition}

\begin{proof}
The proof is identical to the proof of Proposition~\ref{prop:Behavior-of-S}, using Proposition~\ref{Prop:Mixing-in-well-Y} and Lemma~\ref{lem:comparison-of-rates-Y} in place of Theorem~\ref{thm:mixing-below-frozen-bond} and Lemma~\ref{lem:Comparison-of-rates}. 
\end{proof}

With Propositions~\ref{prop:Behavior-of-S} and \ref{prop:behavior-of-Y}, we are now in position to prove Theorem~\ref{thm:main1}.

\begin{proof}[Proof of Theorem~\ref{thm:main1}]
It suffices to prove the following claim inductively. 

\smallskip
\textbf{Inductive hypothesis}: For all $n$, all times $s_1 < s_2 < ... < s_n$ with $s_{i} - s_{i-1}$ bounded away from zero in $N$,  there exists a coupling such that $\SK(s_i t_{a,\gamma}) = Y(s_i)$ for all $i\le n$ with probability $1-n e^{-N^{\theta}}$ for some $\theta>0$.

\smallskip
\textbf{Base case:} We artificially introduce $s_0$ independent of $N$ (i.e., of constant order, and say less than $\frac{s_1}{2}$ so that $s_1 - s_0$ is also bounded away from zero uniformly in $N$), and we use this as our base case. 
We couple the value of $\SK (s_0t_{a,\gamma})$ with $Y(s_0)$.
First note that Lemma \ref{lem:gaps} implies that $s_0 t_{a,\gamma} < \exp(2\beta \abs{J_{(K)}} -\frac{{a} }{4} N^{\gamma-\xi} )$. 
By Lemma~\ref{lem:spin-structure} the distribution of $\SK (s_0 t_{a,\gamma})$ is within $e^{-cN^{\gamma-\xi}}$ in total-variation distance from $U \sim \text{Unif}(\{ \pm 1 \}^{K})$, so let $(\SK (s_0 t_{a,\gamma}),U)$ denote an optimal coupling between these random variables, and set $Y(0)=U$. Note that a simple bound on the first set of clock rings of $Y$ implies that $Y(0)=Y(s_0)$ with probability at least $1-e^{-cN^{\gamma-\xi}}$ for some $c>0$.

\smallskip
\textbf{Inductive Step:}
Suppose now that the inductive hypothesis holds, and we have couplings so that $Y(s_i)=\SK(s_i t_{a,\gamma} )$ for $1 \leq i \leq n$. To construct the coupling at time $s_{n+1}$, it will suffice by the Markov property to prove the following: Suppose that $\SK(0)=Y(0)$, and that $E_{a,\gamma} \subset \sat(X(0))$; then there is a coupling so that $\SK(s t_{a,\gamma}) = Y(s)$ with probability $1-e^{-N^{\theta}}$, where $s=s_{n+1}-s_n$ is bounded away from zero, but possibly growing with $N$. In what follows let $L = L(s)$.

By the Markov property, it is equivalent to appeal to Propositions~\ref{prop:Behavior-of-S}--\ref{prop:behavior-of-Y}, to determine the distribution of $\SK(st_{a,\gamma})$ and $Y(s)$ conditional on the values of $X(0)$ and $Y(0)$.
First note that $st_{a,\gamma} < \exp(2\beta \abs{J_{(L-1)}} - \frac{\beta}{2} N^{\gamma-\xi} )$ (see Remark \ref{rem:Frozen-above-L(s)-below-s}), and hence via Lemma \ref{lem:spin-structure} we have with probability $1-e^{-cN^{\gamma-\xi}}$ that $X(t)$ is equal to the restricted dynamics $\ol{X}(t)$ rejecting all updates to coordinates $\{v_i,w_i \}$ for $i < L$ for $t \in [0,st_{a,\gamma}]$. Similarly the bound on $st_{a,\gamma}$ implies that $Y$ is equal in distribution to the process $\ol{Y}$, rejecting updates to coordinates $Y_1,...,Y_{L-1}$ with probability $1-e^{-cN^{\gamma-\xi}}$. 

Now depending on the value of $s$, we either fall into case (1) or (2) from Propositions~\ref{prop:Behavior-of-S}--\ref{prop:behavior-of-Y}.

In case (1) the law of $\SK(st_{a,\gamma})$ conditional on time $0$ is given by fixing the coordinates $i < L(s)$, running $M(t)$ for time $ st_{a,\gamma}$ to sample $\SK_L(st_{a,\gamma})$, and then sampling the remaining points below according to the stationary distribution. The law of $Y$ conditional on time $0$ is given by fixing coordinates $i <L$, running $M'(t)$  for time $st_{a,\gamma}$ to sample $Y_L$, and then sampling the remaining coordinates below according to the stationary distribution conditioned on the values $Y_i(s)$ for $1 \leq i \leq L$.   
Note that Lemma~\ref{lem:energy-drop-off-J} implies the ratios of the rates for $M$ and $M'$ satisfy: 
\[
{\Bigg| \frac{\ol{Z}_+}{\ol{\lambda}_+ } -1 \Bigg| \leq 2 e^{-\beta N^{\gamma}} \, ,}
\]
so by Pinkser's inequality we have for $\ol{Z}_+,\ol{\lambda}_+$ (with the same holding for $\ol{Z}_-,\ol{\lambda}_-$) that: 
\begin{align*}
\TV{\text{Exp}(\ol{Z}_+)}{\text{Exp}(\ol{\lambda}_+)} & \leq \sqrt{\frac{1}{2} D_{\text{KL}}(\text{Exp}(\ol{Z}_+) \,\|\, \text{Exp}( \ol{\lambda}_+ ))} \\  & \leq \Big| \log\Big( \frac{\ol{Z}_+}{\ol{\lambda}_+}\Big) - \Big(1- \frac{\ol{\lambda}_+}{\ol{Z}_+} \Big) \Big|^{1/2} \leq C \exp(-c N^{\gamma} ) \, ,
\end{align*}
where $D_{\text{KL}}(\cdot\, \| \,\cdot)$ is the Kullback--Leibler divergence. 
Since we are in case (1) of Propositions~\ref{prop:Behavior-of-S}--\ref{prop:behavior-of-Y}, we have that,
\[
st_{a,\gamma} \in \left[\exp\left(2\beta \abs{J_{(L)}} - 2DN^{\frac{1}{2} + \frac{1}{2\alpha}}\right),\exp\left(2\beta \abs{J_{(L)}} + 2DN^{\frac{1}{2} + \frac{1}{2\alpha}}\right) \right] \, ,
\]
and over an interval of length $\exp(2\beta \abs{J_{(L)}} + 2D N^{1/2+1/(2\alpha)} )$ we may assume that there are at most $\exp(100D N^{1/2+1/(2\alpha) })$ clock rings for $M$ and $M'$. Indeed, treating the case for $M$, note we have that exponential clocks of rate $\ol{\lambda}_+,\ol{\lambda}_-$ both stochastically dominate an exponential clock with rate $\exp(-2\beta \abs{J_{(L)}} + CN^{\frac{1}{2\alpha} })$, and hence a large deviation estimate similar to~\eqref{eq:number-of-clock-rings} implies that the probability of having more than $\exp(100 DN^{1/2+1/(2\alpha) })$ clock rings over the interval above is exponentially small in $N^{\gamma-\xi}$. 

Consequently if $M(0)=M'(0)$, then by sequentially taking an optimal coupling on the clock rings, and using the same uniform choice of $\pm 1$ at each clock ring, we have that, 
{
\begin{align} \label{eq:bound-M-M'-distance}
\p( M(st_{a,\gamma}) \neq M'(st_{a,\gamma}) ) \leq e^{-cN^{\gamma-\xi} } \, ,
\end{align}
}
for some constant $c>0$. To finish the coupling of $\SK(st_{a,\gamma})$ and $Y(s)$ we note that if $\mathcal{C}$ denotes the well defined by the assignment of the values of $\SK_1,...,\SK_L$ at time $st_{a,\gamma}$, then one has,
{
\begin{align} \label{eq:TV-distance-Y-S}
\TV{\pi_{\beta,\J}( (\sigma_{v_i})_{1 \leq i < L} \in \cdot \mid \sigma \in \mathcal{C})}{\pi^Y( (Y_i)_{1  \leq i < L} \in \cdot \mid  Y \in \mathcal{C} )} \leq C_{\beta} \exp(-c_{\beta} N^{\gamma} ) \, .
\end{align}
To see that ~\eqref{eq:TV-distance-Y-S} holds, note by Proposition ~\ref{lem:Reversibility-of-Y} that on the event that a configuration $\sigma$ satisfies the top $K$ bonds, the law of $\pi_{\beta,\mathbf{J}}$ and $Y$ are equal, and this event has probability at least $1-C_{\beta} e^{-c_{\beta} N^{\gamma} }$ by Theorem ~\ref{THM:Allignment-of-Bonds}. 
}

Therefore, by taking an optimal coupling on the lower coordinates, { and combining ~\eqref{eq:bound-M-M'-distance} ~\eqref{eq:TV-distance-Y-S} with the bounds given by Lemma ~\ref{lem:spin-structure}}, we have that for some $\theta>0$, 
\begin{align} \label{eq:inductive-step-final-bound-a}
\p( \SK(st_{a,\gamma}) \neq Y(s) \mid  \SK(0)=Y(0) ) \leq e^{-N^{\theta}} \, .
\end{align}

Now for case (2) we proceed as follows: by Propositions~\ref{prop:Behavior-of-S}--\ref{prop:behavior-of-Y} , the distribution of $\SK(st_{a,\gamma})$ and $Y(s)$ are within total-variation distance $e^{-N^{\theta}}$ of the distributions determined by fixing the values of coordinates $i \leq L$, and sampling the remaining coordinates conditional on the values $\SK_i,Y_i$ for $i \leq L$. Since $\SK(0)=Y(0)$ and coordinates $1,...,L$ are frozen, we may appeal to \eqref{eq:TV-distance-Y-S} and take an optimal coupling on the lower coordinates to get 
\begin{align} \label{eq:inductive-step-final-bound}
\p( \SK(st_{a,\gamma}) \neq Y(s) \mid  \SK(0)=Y(0) ) \leq e^{-N^{\theta}} \, .
\end{align}
To finish the proof, we write,
{
\begin{align*}
\p\Big( & \bigcup_{i\le n+1}  \SK(s_i t_{a,\gamma})\ne Y(s_i)\Big)  \\ 
& \leq \p\Big(\bigcup_{i\le n} \SK(s_i t_{a,\gamma}\ne Y(s_i)\Big) + \p(\SK(s_{n+1}t_{a,\gamma})\ne Y(s_{n+1}) \mid \SK(s_{i}t_{a,\gamma}) = Y(s_{i}) \ \forall i \leq n )\,.
\end{align*}
}
We see the first term on the right-hand side is at most $ne^{-N^{\theta}}$ by the inductive hypothesis, and the second term is at most $e^{-N^{\theta}}$ by using Markov property to apply~\eqref{eq:inductive-step-final-bound-a}--\eqref{eq:inductive-step-final-bound} with $s= s_{n+1} -s_n$. Together we get that the above is at most $(n+1) e^{-N^{\theta}}$ 
as desired. 
\end{proof}

\appendix

\section{Deferred Proofs  }

\subsection{The High Temperature Regime} \label{appendix:high-temp}
In this section we prove that if $\beta<\beta_0$ from~\eqref{eq:beta-0} that the law of $Y$ from Theorem~\ref{thm:main1} is within $o(1)$ in total-variation distance to the uniform measure, and that $q_{a,\gamma}$ from Theorem~\ref{thm:main-autocorrelation} has limiting distribution $\delta_0$. Throughout this section, we work on the events of Section \ref{sec:coupling-matrix-preliminaries}, and write $K=N^{1-\alpha \gamma}$ for $\gamma \in (\frac{1}{2\alpha} + \frac{1}{1+\alpha}, \frac{1}{\alpha} )$. 

\begin{proposition} \label{prop:law-q-high-temp}
Suppose that $(2\beta)^{\alpha} \Gamma(1-\alpha) <1$, then with $\p_{\J}$ probability $1-o(1)$ we have: 
\[
\TV{\text{Law}(q_{a,\gamma}^{(N)})}{ \delta_0} \to 0 \, .
\]
\end{proposition}

We will appeal to a coupling of the Gibbs measure conditioned on a satisfying assignment to the top $K$ bonds, to a collection of random-cluster, or FK, measures: see e.g.,~\cite{CN} for more on this coupling. 
Fix an integer $L \leq N $, and let $\tau$ be any configuration on $\{ \pm 1 \}^{2L}$ so that $e_i \in \sat(\tau)$ for $1 \leq i \leq L$.

Write $\mathcal{C}_{\tau}$ for the collection of configurations in $\Sigma_N$, agreeing with $\tau_{v_i},\tau_{w_i}$ for $1 \leq i \leq L$. 
We now describe the coupling between $\pi_{\beta,\J}(\cdot \mid \mathcal{C}_{\tau} )$, and a random cluster measure which we denote as $\prc_{\tau}$.

Given $\sigma \sim \pi_{\beta,\J}( \cdot \mid \mathcal{C}_{\tau} )$ we may generate $\omega \in \Omega_N:=  \{0,1\}^{\binom{N}{2}}$ drawn from $\prc_{\tau}$ as follows: 
\begin{enumerate}
    \item If $e \in \sat(\sigma)$, set $\omega_e=1$ with probability $p_{e,L}$ independent of all other edges, where 
    \begin{align} \label{eq:percolation-parameters}
p_{e,L} = \begin{cases}
    1 &\text{if} \ e = e_i \ \text{for some} \ i \leq L \\ 
    1-e^{-2\beta \abs{J_{(e)}}} &\text{otherwise} 
\end{cases} \, .
\end{align}

    \item If $e \not \in \sat(\sigma)$, set $\omega_e=0$. 
\end{enumerate}

In the other direction, given a random-cluster configuration $\omega \sim \prc_{\tau}$ we may generate a spin glass configuration sampled from $\pi_{\beta,\J}(\cdot \mid \mathcal{C}_{\tau})$ as follows:
\begin{enumerate}
    \item Set $\sigma_{v_l}= \tau_{v_l}$ for $l \leq K$ ,
    \item For each connected component of $\omega$ disjoint from $\{v_l,w_l\}_{l \leq L}$ pick an arbitrary vertex  in the component and assign it $\pm 1$ independently of all other components, 
    \item assign values to the remaining spins by requiring $\{ e: \omega_e =1 \} \subset \sat(\sigma)$. 
\end{enumerate}

This coupling fully defines the joint law of $\Sigma_N \times \Omega_N$, and we can compute directly to show that $\prc_{\tau}$ is of the form,
\[
\prc_{\tau} (\omega)=  \frac{1}{Z} \mathbf{1}_{\mathbf{S}}(\omega) 2^{\mathsf{Comp}(\omega)} \p_{\ind,L}(\omega)\,, 
\]
where 
\[\mathbf{S}= \{ \omega \in \Omega_N : \exists \ \sigma \in \Sigma_N, \ J_{ij}\sigma_i \sigma_i \omega_{ij} \geq 0  \text{ for all $i,j$}  \}  \,,
\]
and where $\mathsf{Comp}(\omega)$ is the number of connected components of $\omega$, and $\p_{\ind,L}$ is independent percolation with parameters $p_{e,L}$ from~\eqref{eq:percolation-parameters}.

We note that $\mathbf{1}_{\mathbf{S}}$ and $2^{\mathsf{Comp}(\omega)}$ are decreasing functions on $\Omega_N$, and hence the FKG inequality for independent percolation implies that for any increasing event $A$,
\[
\max_{\tau} \prc_{\tau}(A) \leq \p_{\ind,L}(A) \, .
\]
Furthermore by monotonicity of independent percolation in the parameters $p$, we have a uniform bound  
\begin{align} \label{eq:domination-by-percolation}
\max_{0 \leq L \leq K} \max_{\tau  } \prc_{\tau}(A) \leq \p_{\ind}(A) \, ,
\end{align}
where $\p_{\text{ind}}=\p_{\ind,K}$.

Lastly we remark that if $\langle f(\sigma) \rangle_{\tau} $ denotes the expectation of $f$ under $\pi_{\beta, \J}(\cdot \mid \mathcal{C}_\tau)$, and $\langle f(\omega)\rangle_{\tau}^{\rc}$ denotes the expectation under $\pi_{\tau}^\rc$, then one has: 
\[
\langle \sigma_i \sigma_j \rangle_{\tau}  = \langle\eta_{ij}(\omega) \rangle_{\tau}^{\rc}\, ,
\]
where $\eta_{ij}(\omega)$ is defined as:
\[
\eta_{ij}(\omega) = 
\begin{cases}
\prod_{e \in \gamma_{ij}} \sgn{J_e} &\text{if} \ i \xleftrightarrow{\omega} j  \\
0 &\text{otherwise} 
\end{cases} 
\, ,
\]
with $\gamma_{ij}$ any path of open edges in $\omega$ connecting $i$ to $j$ (it is clear that the choice of $\gamma$ does not affect the product of signs of bond-values along open edges). 
\\

Having constructed the coupling of the FK measures to the conditional Gibbs' measures we may now shed some light on the definition of $\beta_0$. We have the following Lemma: 

\begin{lemma} \label{lem:asymptotics-of-pe} Suppose that $(2\beta)^{\alpha} \Gamma(1-\alpha) <1$, and that $p(J)= 1-e^{-2\beta \abs{J}}$, then $N \E[p(J)] < 1$ uniformly over $N$.      
\end{lemma}

\begin{proof}
By definition of $\p_\J$ we have: 
\[
\E [p] = 1 - \int_{1}^{\infty} \alpha e^{-\frac{2\beta x}{N^{1/\alpha}}} x^{-1-\alpha} dx \, , 
\]
looking just at the integral, write $a_N = \frac{2\beta}{N^{1/\alpha} }$, and making the change of variables $u = a_Nx$ one has: 
\[
\int_{1}^{\infty} \alpha e^{-\frac{2\beta}{N^{1/\alpha}} x} x^{-1-\alpha} dx = a_N^{\alpha} \int_{a_N}^{\infty} \alpha e^{-u} u^{-1-\alpha} du \, ,
\]
integrating by parts the integral above is equal to: 
\[
a_N^{\alpha} \int_{a_N}^{\infty} \alpha e^{-u} u^{-1-\alpha} du = e^{-a_N} - a_N^{\alpha} \int_{a_N}^{\infty} e^{-u} u^{-\alpha} du \, ,
\]
and hence (substituting back in $a_N$) we have:
\[
N \E p_e = N( 1-e^{-\frac{2\beta}{N^{1/\alpha} }}) + (2\beta)^{\alpha} \int_{\frac{2\beta}{N^{1/\alpha}}}^{\infty} e^{-u} u^{-\alpha} du \, .
\] 
Since $e^x$ is Lipschitz on $(-\infty,0]$ with Lipschitz constant $1$, and the integral on the right hand side converges to $\Gamma(1-\alpha)$ monotonically from below as $N \to \infty$ we have: 
\[
N \E p_e \leq  \frac{2\beta}{N^{1/\alpha-1}} + (2\beta)^{\alpha} \Gamma(1-\alpha) \, ,
\]
and hence $N \E p_e < 1$ uniformly for all large $N$.  
\end{proof}

We now state a precise form of the uniformity for the stationary measure of $Y$. 

\begin{theorem}\label{thm:high-temp-regime}
Suppose that $(2\beta)^{\alpha} \Gamma(1-\alpha ) <1$. Then there is $C(\beta),\epsilon(\gamma) >0$  so that: 
\[
\p_\J \Big( \sup_{L<K} \max_{\tau \in \Sigma_{N,L} } \TV{ \pi_{\beta,\J}( (\sigma_{v_{L+1}},...,\sigma_{v_{K}} ) \in \cdot \mid \mathcal{C}_{\tau} )}{ \text{Unif} \ ( \{ \pm 1 \}^{K-L}  )} > \frac{C}{N^{\epsilon}} \Big) \to 0 \,,
\] 
where $\text{Unif} \ (\{ \pm 1 \}^{K-L} )$ denotes the uniform measure on $\{\pm 1\}^{K-L}$.
\end{theorem}
Taking $L=0$, and noticing that with probability $1-o(1)$, $|E_{a,\gamma}|\le K$, Theorem \ref{thm:high-temp-regime} then implies that the stationary measure of $Y$ has $o(1)$ distance to uniform. 

Towards proving Theorem \ref{thm:high-temp-regime} and Proposition \ref{prop:law-q-high-temp} we begin with two lemmas. The first lemma is a statement about mean-field Bernoulli percolation, i.e., random graphs.  

\begin{lemma} \label{lem:decay-of-connection-prob}
    Fix a subset $F \subset \{ (i,j) : 1 \leq i < j \leq N \}$ of vertex-disjoint edges in $\{1,...,N\}$ of size $K=o(N)$. Let $G_F$ denote an Erd\H os-R\'enyi $G(N,c/N)$ with $c<1$ conditioned on every edge in $F$ being present. Then for any $x,y$ such that $(x,y) \not \in F$ one has:
    \begin{align} \label{eq:decay-for-connections}
\p( x \xleftrightarrow{G_F} y ) \leq \frac{C}{N} \, ,
    \end{align}
   for a constant $C>0$ depending only on $c$.  
\end{lemma}

\begin{proof}
Taking a union bound over all simple paths $\mathcal{P}$ from $x$ to $y$ one has that:
\[
\p(x \xleftrightarrow{G_F} y) \leq \sum_{m=1}^{N-1} \sum_{ \mid \mathcal{P} \mid = m} \Big( \frac{c}{N} \Big)^{m - \mid \mathcal{P} \cap F \mid } \, .
\]
We now bound the total number of paths $\mathcal{P}$ of length $m$ whose intersection with $F$ is empty by the total number of paths to get:
\begin{align} \label{eq:path-bound-empty}
\sum_{m=1}^{N-1} \sum_{ \substack{\mid \mathcal{P} \mid =m \\ \mid \mathcal{P} \cap F \mid = 0 } } \Big( \frac{c}{N} \Big)^{m} \leq \sum_{m=1}^{N-1} \binom{N}{m-1} (m-1)! { \Big( \frac{c}{N} \Big)^m } \leq \frac{1}{(1-c)N} \, .
\end{align}
It thus suffices to bound the contribution from those paths $\mathcal{P}$ with non-empty intersection with $F$. Note that a path $\mathcal{P}$ of length $m$ can use at most $\lceil m/2\rceil +1$ edges in $F$. We rewrite:
\[
\sum_{\substack{ \mid \mathcal{P} \mid=m  \\ \mathcal{P} \cap F \neq \emptyset } } \Big( \frac{c}{N} \Big)^{m-\abs{\mathcal{P} \cap F} } = \Big( \frac{c}{N} \Big)^m \sum_{l=1}^{\lceil m/2\rceil +1} (N/c)^l \abs{ \{ \mathcal{P}=m : \abs{\mathcal{P} \cap F} = l \} }
\]
The number of all paths of length $m$ from $x$ to $y$ using $l$ edges in $F$ is bounded above by: 
\[
\binom{m}{l} \binom{N -2K}{m-2l-1} \binom{K}{l} l! 2^l (m-2l-1)! \leq  \frac{(2m)^l}{l!} K^l N^{m-2l-1} = N^{m-1} \frac{(2m)^l}{l! (N^2/K)^l} \, .
\]
Consequently the contribution from paths with non-empty intersection with $S$ satisfies:
\begin{align} \label{eq:path-bound-non-empty}
\sum_{\substack{ \mid \mathcal{P} \mid=m  \\ \mathcal{P} \cap F \neq \emptyset } } \Big( \frac{c}{N} \Big)^{m-\abs{\mathcal{P} \cap F} } 
\leq \frac{1}{N} \sum_{m=1}^{N-1} \sum_{l=1}^{ \lceil m/2 \rceil +1} \frac{c^{m-l}(2m)^l}{l! (N/K)^l} \leq \frac{1}{N} \sum_{m=1}^{N-1} \exp \Big[ m \Big( \log(c) + \frac{2K}{cN}  \Big) \Big] \, . 
\end{align}
This is at most $C'/N$, since $0<c<1$ and $K=o(N)$, and so the partial sums of this upper bound are uniformly bounded in $N$. Combining \eqref{eq:path-bound-empty} and \eqref{eq:path-bound-non-empty} we then have:
\[
\p( x \xleftrightarrow{G_F} y ) \leq \frac{C}{N} \, , 
\]
for some constant $C$ depending only on $c$.
\end{proof}

We now prove a structural result for the random cluster models defined above. 

\begin{lemma} \label{lem:rc-high-temp-structure}
    Suppose $(2\beta)^{\alpha} \Gamma(1-\alpha) < 1$ and recall $K= N^{1-\alpha \gamma}$. Then with $\p_J$ probability $1-o(1)$, all distinct pairs of vertices $(v_l,w_l)_{l \le K} $ lie in distinct clusters with high probability uniformly over the choice of $L$ and $\tau$. Specifically there is $C(\beta),\epsilon(\gamma)>0$  such that
    \[
\p_{\J} \left( \sup_{0 \leq L <K} \max_{\tau \in \Sigma_{N,L}}  \prc_{\tau}( \exists  l \neq l' \leq K : \{v_l,w_l\} \leftrightarrow \{v_{l'},w_{l'}\} ) > C/N^{\epsilon}   \right) \to 0 \, ,
    \]
    { where $\Sigma_{N,L}$ corresponds to all configurations in $\{ \pm 1 \}^N$ such that $e_1,...,e_L \in \sat (\sigma). $}
    
\end{lemma}

\begin{proof}
Let $A$ denote the event two distinct pairs of vertices $(v_l,w_l), (v_{l'},w_{l'})$ lie in the same cluster. Then the event $A$ is increasing and so by \eqref{eq:domination-by-percolation} it will suffice to bound $\p_{\ind,K}(A)$.
Denote by $\mathcal{F}_K$ the collection $\{e_1,...,e_K \}$. 
Conditional on $\{J_e \}_{e \in \mathcal{F}_K }$, the remaining bonds $\{J_{ij} \}_{(i,j) \not \in \mathcal{F}_K}$ are i.i.d conditioned on their absolute value lying below $\abs{J_{(K)}}$. 
We can couple this measure with one where the bonds in $\mathcal{F}_K$ are infinity, and the remaining values are i.i.d sampled from $\nu_{\alpha,N}$.
We use the following sampling scheme to generate our random variables:

    First, generate a set $F \subset \{ (i,j): 1 \leq i < j \leq N \}$ of size $K$, and set $F= \mathcal{F}_K$. Second, generate $\{ J_{ij}' \}$ from $\p_\J$. Finally assign the values $J_e$ as follows: 
    
    \begin{enumerate}
        \item  Assign the edges $e \in \mathcal{F}_K$ the random variables $\abs{J_{(1)}'},...,\abs{J_{(K)}'}$ uniformly at random.
        \item   For $e \not \in \mathcal{F}_{K}$ sample an i.i.d collection $J_{e,l}$ from $\p_\J$. Let $l_e = \inf\{ l \geq 1 : \abs{J_{e,l}} \leq \abs{J_{(K)}'} \}$. Set $J_e= J_{l_e}$. 
    \end{enumerate}

For fixed $F$ the sampling algorithm above generates a sample from the distribution of the $\{J_{ij} \}$ conditioned on the position of the top $K$ bonds in the order statistics.
Note that by replacing $J_e$ for $e \not \in F$ by the value $J_{e,1}$, and $J_e$ by $\infty$ for $e \in F$, the percolation parameters $p_e$ only increase. 
Consequently we have by monotonicity that for any increasing event $\mathcal E$, 
\begin{align}\label{eq:monotonicity-resampling}
    \E_{\J}[\p_{\ind}(\mathcal E) \mid \mathcal F_K] \le  \p(G_{\mathcal F_K} \in \mathcal E)
\end{align}
where $G_{\mathcal{F}_K}$ is an independent percolation with parameters $p_e$ conditioned on the edges in $\mathcal{F}_K$ being present, and therefore, 
\[
\E_\J( \p_{\text{ind}}(A) | \mathcal{F}_{K}) \leq \p( \exists x,y \in V_{\mathcal{F}_{K}}, (x,y) \not \in \mathcal{F}_{K}: x \xleftrightarrow{G_{\mathcal{F}_{K}}} y )\,.
\]
 An application of Fubini's theorem and independence of the $J_e$'s for $e \notin \mathcal{F}_K$ gives us that 
\begin{align*}
\E_\J \p_{\ind}(A) &\leq \E_{\mathcal{F}_{K}} \E_{J_{\mathcal{F}_{K}^c}} \p (\exists x,y \in V_{\mathcal{F}_K}, (x,y) \not \in \mathcal{F}_{K}: x \xleftrightarrow{G_{\mathcal{F}_K} } y ) \mathbf{1}( \abs{V_{\mathcal{F}_K}} = 2K) \\
& \quad + \p( \abs{V_{\mathcal{F}_K}} \neq 2K) 
\\
&< \frac{C}{N} ( 4K^2 + 2K) \, ,
\end{align*}
where for the first term we note that $\E_{J_{S^c}} \p_{\text{ind}}$ is the law of mean field percolation measure with parameter $p= \E p_{12}$ conditioned on edges in $S$ being present. Since $N \E p_{1,2} < 1$ uniformly for large $N$, there is $c_{\beta}<1$ so that $\E_{J_{S^c}}$ is dominated by independent percolation with parameter $c_{\beta}/N$, conditioned on the edges $e_1,...,e_K$ being present. Lemma \ref{lem:decay-of-connection-prob} and a union bound, then bounds the first term. 
We conclude by Markov's inequality that:
\[
\p_\J\left( \sup_{L<K} \max_{\tau \in \Sigma_{N,L}}  \prc_{\tau}( \exists  l\neq l' \leq K : \{v_l,w_l\} \leftrightarrow \{v_{l'},w_{l'}\}  ) > C/N^{\epsilon}  \right) \to 0 \, ,
\]
for some $C(\beta),\epsilon(\gamma)>0$.
\end{proof}

We now prove Theorem \ref{thm:high-temp-regime}.  
\begin{proof}[Proof of Theorem~\ref{thm:high-temp-regime}]
   By Lemma \ref{lem:rc-high-temp-structure}, we have with $\p_\J$ probability $1-o(1)$ that for any $L <K$ and $\tau \in \Sigma_{N,K}$ that for a sample $\omega \sim \prc_{\tau}$, the vertices $(v_l,w_l)_{l \leq K }$ lie in distinct cluster with probability at least $1- CN^{-\epsilon}$. 
Hence via the coupling of $\prc_{\tau}$ and $\pi_{\beta,\J}(\cdot \mid \mathcal{C}_{\tau})$ we have that first sampling $\omega \sim \prc_{\tau}$ and then using $\omega$ to generate $\sigma \sim \pi_{\beta,\J}(\cdot \mid \mathcal{C}_{\tau})$, the random variables $\sigma_{v_{L+1}},...,\sigma_K$ are i.i.d uniform with probability $1-CN^{-\epsilon}$. Consequently the total variation distance to uniform is bounded above by $CN^{-\epsilon}$ on the complement of the event in Lemma \ref{lem:rc-high-temp-structure}. In other words:
\[
\p_\J \Big( \sup_{L<K} \max_{\tau \in \Sigma_{N,L} } \TV{ \pi_{\beta,\J}( (\sigma_{v_{L+1}},...,\sigma_{v_{K}} ) \in \cdot \mid \mathcal{C}_{\tau} )}{ \text{Unif}( \{ \pm 1 \}^{K-L}  )} > \frac{C}{N^{\epsilon}} \Big) \to 0 \,,
\]
completing the proof.
\end{proof}
We now prove Proposition~\ref{prop:law-q-high-temp}

\begin{proof}[Proof of Proposition \ref{prop:law-q-high-temp}] It will suffice to show the second moment of $q_{a,\gamma}$ tends to zero as $N$ tends to infinity. First note that a Chernoff bound implies $\abs{E_{a,\gamma} } < N^{1-\alpha \gamma}$ with probability $1-e^{-N^{\theta}}$, as \[N^{1-\alpha \gamma} -  \E \abs{E_{a,\gamma}}  = N^{1-\alpha \gamma} -    \frac{(2\beta)^{\alpha}(N-1)N^{-\alpha \gamma}}{2a^{\alpha}}   \, ,\]
and this diverges to $+\infty$ provided $\beta<\beta_0$ and $a \geq 1$. 
Now we may write, 
\[
\E [q_{a,\gamma}^2] = \frac{1}{2^{\abs{E_{a,\gamma}}}} \sum_{\tau \in \{ \pm 1 \}^{\abs{E_{a,\gamma}}}} \langle R_{1,2}^2 \rangle_{\tau}\,, \qquad \text{where} \qquad R_{1,2} = \frac{1}{N} \sum_{i=1}^{N} \sigma_i^1 \sigma_i^2 \, ,
\]
where $\sigma^1, \sigma^2$ are independent draws. By expanding $R_{1,2}^2$ one has on the event that $\abs{E_{a,\gamma}}<K$ that: 
    \begin{align*}
\langle R_{1,2}^2 \rangle_{\tau} &= \frac{1}{N} + \frac{1}{N^2} \sum_{1 \leq i < j \leq N } \langle \sigma_i \sigma_j \rangle_{\tau}^2  
\\
&= \frac{1}{N} + \frac{1}{N^2} \sum_{1 \leq i \leq j \leq N} (\langle\eta_{ij} (\omega) \rangle_{\tau}^\rc) ^2
\\
&\leq \frac{1}{N} + \frac{1}{N^2} \sum_{1 \leq i < j \leq N} \prc_{\tau} ( i \leftrightarrow j )^2 
\\
&\leq \frac{1}{N} + \frac{1}{N^2} \sum_{1 \leq i < j \leq N}  \p_{\ind}(i \leftrightarrow j )^2  \, ,
    \end{align*}
combining with the expression for $\E [q_{a,\gamma}^2]$ (dropping the square on $\p_{\ind}$ and using $q_{a,\gamma}^2 \leq 1$) one has,
\[
\E_{\J} \E [q_{a,\gamma}^2] \leq \frac{1}{N} + \frac{1}{N^2} \E_{\J} \Big[\sum_{1 \leq i < j \leq N}  \p_{\text{ind}} (i \leftrightarrow j ) \Big] + e^{-N^{\theta}}  \, . 
\] 

Now by~\eqref{eq:monotonicity-resampling}, and the FKG inequality to apply  Lemma \ref{lem:decay-of-connection-prob}, we have for any $i <j $ that 
\[
\E_\J [\p_{\text{ind} }( i \leftrightarrow j)]  = \E_{\J} \Big[ \mathbf{1}_{ij\notin \mathcal F_K} \E_\J [ \p_{\text{ind}}(i\leftrightarrow j)\mid \mathcal F_K]\Big] + \p_\J(ij\in \mathcal F_K)  \leq \frac{C}{N} + \frac{K}{ \binom{N}{2} } \, .
\]
 Combining these bounds, one has
\[
\E_{\J} [\E [ q_{a,\gamma}^2]] \leq e^{-N^{\theta}} + \frac{1}{N} + \frac{1}{N^2} \Bigg[ \binom{N}{2} \frac{C}{N} + K \Bigg] \to 0 \, ,
\]
and so an application of Markov's inequality implies that the total variation distance between $q_{a,\gamma}$ and a delta-mass at zero is $o(1)$.
\end{proof}

\subsection{Alignment of Bonds}  \label{AP:Allignment}

In this section we give a short proof of Theorem \ref{THM:Allignment-of-Bonds} for completeness. 

\begin{proof}[Proof of Theorem \ref{THM:Allignment-of-Bonds}]  We work on the events of section \ref{sec:coupling-matrix-preliminaries}. 
Consider the random cluster representation of the Levy-model defined in Appendix ~\ref{appendix:high-temp}, and consider the collection of random cluster configurations given by: 
\[
A = \{ \omega \in \Omega_N : \exists \  1 \leq i \leq K, \omega_{e_i}=0  \} \, , 
\]
which consists of random cluster configurations which are not guaranteed to sample $\sigma$ so that $e_i \in \sat(\sigma)$ for $1 \leq i \leq K$. 

Denote by $\mathcal{N}_K$ the collection of edges incident to one of the vertices in the top $K$ positions in the  order statistics that are not equal to $e_1,...,e_K$, that is: 
\[
\mathcal{N}_K := \bigcup_{1 \leq l \leq K } \{e : e=(v_l,v) \ \ \text{and} \  v \neq w_l, \ \text{or} \ e=(v,w_l) \ \ \text{and} \  v \neq v_l   \} \, .
\]
Now define a map $\Phi:A\to A^c$ which takes $\omega \in A$ and returns  
\[
\Phi(\omega)_{e} = 
\begin{cases}
    1 &\text{if} \  e \in  \{e_1,...,e_K \} \\ 
    0 &\text{if } \ e  \in \mathcal{N}_K \\ 
    \omega_e &\text{otherwise}
\end{cases} \, .
\]
As $\omega$ is a valid random-cluster configuration, and the vertex set of edges $e_1,...,e_K$ are disjoint, this will remain a valid random cluster configuration. 
Then we may compute for every $\omega \in A$, 
\[
\frac{\prc(\omega)}{\prc(\Phi(\omega))} = 2^{ \mathsf{Comp}(\omega) - \mathsf{Comp}(\Phi(\omega)) }  \frac{ \left(\prod_{e \in \mathcal{N}_K} p_e^{\omega_e}(1-p_e)^{1-\omega_e} \right) \left( \prod_{i=1}^{K} p_{e_i}^{\omega_{e_i}} (1-p_{e_i})^{1-\omega_{e_i }} \right)   }{\left(\prod_{e \in \mathcal{N}_K} (1-p_e) \right) \left( \prod_{i=1}^{K} p_{e_i} \right)} \, .
\]
Now since $\omega \in A$ there is some $1 \leq i \leq K$ so that $\omega_{e_i} = 0$, and hence we may factor out $e^{-2\beta \abs{J_{(i)}}}$ (which we bound above by $e^{-2\beta \abs{J_{(K)}}}$) from the numerator. Bounding the rest of the numerator by $1$, and bounding $p_{e_i}$ from below by $1/2$ (which follows as $\abs{J_{(K)}}> \frac{a}{2\beta} N^{\gamma}$, and $p_e=1-e^{2\beta \abs{J_{(e)}}}$) we then have for every $\omega \in A$,  
\[
\frac{\prc(\omega)}{\prc(\Phi(\omega))} \leq 2^{N+K} e^{-2\beta \abs{J_{(K)}} + 2\beta \sum_{e \in \mathcal{N}_K} \abs{J_{(e)}} } \leq e^{-0.99 N^{\gamma} } \, ,
\]
    where the last inequality follows by applying Lemma \ref{lem:gaps}, and the relative difference between $\gamma$ and $1-\alpha\gamma + \frac{1}{2\alpha}$.  
Furthermore, we may bound the size of the preimage of any image configuration $\Phi(\omega)$, i.e., $|\{\omega': \Phi(\omega')=\Phi(\omega)\}|$,  by $2^{NK}$ for the possible values $\omega'$ took on $e_1,...,e_K$ and on $\mathcal N_K$.  Thus, 

\begin{align*}
\prc(A) = \sum_{\omega\in A} \prc({ \omega})  & { \leq } \sum_{\Phi(\omega)\in \Phi(A)} \sum_{\omega':\Phi(\omega')=\Phi(\omega)} \frac{\prc(\omega')}{\prc(\Phi(\omega'))}\prc(\Phi(\omega)) \\
& \le \prc(\Phi(A)) 2^{NK} e^{ -0.99 N^{\gamma}}\,.
\end{align*}
By definition of $\gamma$ (and $K$), one can check that $2-\alpha \gamma < \gamma$, so bounding $\prc(\Phi(A))\le 1$, we get 
\[
\prc(A) < e^{-\frac{1}{2} N^{\gamma} } \, .
\]
Since $\omega \in A^c$ guarantees that when sampling the coupled spin values, with probability $1$ all edges $e_1,...,e_K$ are satisfied, with $\p_\J$ probability $1-o(1)$ one has that $\pi_{\beta,\J} \big( E_{a,\gamma} \subset  \sat(\sigma) \big) > 1- e^{-\frac{N^{\gamma}}{2} }$
as claimed. 
\end{proof}

\subsection{Canonical Paths Computation} \label{MCP} 
We consider the 4-state Markov chain used in the proof of Proposition~\ref{Prop:Largeish-bonds-spectral-gap}, defined by Glauber dynamics with respect to Hamiltonian~\eqref{eq:4-state-chain-Hamiltonian}, and bound its inverse spectral gap  using the method of canonical paths.

\begin{proof}[Proof of~\eqref{eq:4-state-chain-gap}]
 Recall that the state space is on configurations $\sigma$ in  $\{ \pm 1 \}^2$, and the Hamiltonian is of the form: 
\[
H (\sigma) = \beta J_{12} \sigma_{1} \sigma_{2} + \beta \sum_{l=1}^{2} \sum_{{ j \in V_1 \cup V_3 }} \tau_j J_{l j} \sigma_{l}\,.
\]
For simplicity in the calculation above we shall assume that $J_{12} >0 $.
Label the points as $$x_1= (1,1),\quad x_2=(1,-1),\quad x_3=(-1,1),\quad x_4=(-1,-1)\,.$$
For every pair of configurations $x_i,x_j$ we will need to specify a choice of paths $x_i=x^1 \to x^2 \to ... \to  x^n=x_j$ such that $P(x^l,x^{l+1})>0$ for each $l$. Let $\Gamma_{ij}$ denote an arbitrary shortest length path from $x_i$ to $x_j$. 
The quantity we must compute is the congestion, defined as
\[
B= \max_{e=x_i x_j} \frac{1}{Q(x_i,x_j)} \sum_{k,l: e \in \Gamma_{k,l}} \pi(x_k) \pi(x_l) \abs{\Gamma_{k,l} } { \, , }
\]
{ where} for an edge $e=ab$ for a pair $a,b\in \{\pm 1\}^2$ with positive probability of transit between them, $Q(e)= \pi(a) P(a,b)$. { Throughout} we assume that $J_{12} \in I_{V_2}$ as defined in \eqref{eq:I-intervals}  and that we are working on the events of Lemma~\ref{lem:energy-drop-off-J}. This has the consequence that for any $a,b \in \{\pm 1 \}^2$ we have a uniform upper bound: 
\[
\abs{H(a) - H(b) } \leq 2\beta \max_{l=1,2} \sum_{j} \abs{J_{lj} } \, ,
\]
which will be used to upper bound the congestion ratio above.
\\

Bounding $\abs{\Gamma_{k,l}} \leq 4$ for each pair $k,l$ we then have for $e=(x_i,x_j)$ with $P(x_i,x_j)>0$ that:
\begin{align*}
\frac{1}{Q(x_i,x_j)} \sum_{k,l: e \in \Gamma_{k,l} } \pi(x_k)\pi(x_l) \abs{\Gamma_{k,l} } &\leq 4 \sum_{k,l : e \in \Gamma_{k,l}} \left( \frac{1}{\pi(x_i)} + \frac{1}{\pi(x_j)} \right) \pi(x_k) \pi(x_l) 
\\
&= 4 \sum_{k,l: e \in \Gamma_{k,l}} \left( e^{H(x_k)-H(x_i)}+e^{H(x_k)-H(x_i)} \right) \pi(x_l) 
\\
&\leq 8 \exp \left(2\beta \max_{l=1,2} \sum_{j} \abs{J_{lj}} \right) \sum_{k,l : e \in \Gamma_{k,l} } \pi(x_l) 
\\
&\leq 128 \exp\left(2\beta \max_{l=1,2} \sum_{j} \abs{J_{lj} } \, , \right)
\end{align*}
where in the last line we have used that there are at most $16$ paths total, and $\pi(x_l) \leq 1$. We thus have:
\[
B \leq 128  \exp \Big( 2\beta \max_{l=1,2} \sum_{j} \abs{J_{lj} } \Big)\,.
\]
It is well-known that this congestion quantity upper bounds the inverse spectral {gap} of our chain (see \cite{LP}, Corollary 13.21). 
\end{proof}
\bibliographystyle{plain}
\bibliography{bib}

\end{document}